\pdfoutput=1
\newif\ifpersonal
\documentclass[11pt,a4paper,dvipsnames,x11names]{amsart} 

\usepackage{amsmath,amsthm,amssymb,mathrsfs,mathtools,tensor,eucal} 
\usepackage[all,cmtip]{xy}

\usepackage[LGR,T1]{fontenc} 
\usepackage[utf8]{inputenc} 
\usepackage{CJKutf8}

\usepackage{xr}
\usepackage{appendix}

\usepackage{microtype,inconsolata} 
\usepackage{enumerate,comment,braket,xspace,tikz,tikz-cd,csquotes} 
\usetikzlibrary{shapes.geometric}
\usepackage[centering,vscale=0.7,hscale=0.7]{geometry}
\usepackage{thmtools}
\usepackage[hidelinks,colorlinks=true,linkcolor=Red4,citecolor=RoyalBlue]{hyperref}
\usepackage[capitalize]{cleveref}
\usepackage{xcolor}

\usepackage{mathpazo}
\usepackage{euler}
\usepackage{tikzarrows}

\tikzcdset{
	cells={font=\everymath\expandafter{\the\everymath\displaystyle}},
}

\numberwithin{equation}{subsection}
\theoremstyle{plain}
\newtheorem{thm}[equation]{Theorem}
\newtheorem{thm-intro}{Theorem}
\newtheorem{lem}[equation]{Lemma}
\newtheorem{prop}[equation]{Proposition}
\newtheorem{conj}[equation]{Conjecture}
\newtheorem{cor}[equation]{Corollary}
\newtheorem{cor-intro}[thm-intro]{Corollary}

\theoremstyle{definition}
\newtheorem{defin}[equation]{Definition}
\newtheorem{notation}[equation]{Notation}
\newtheorem{eg}[equation]{Example}
\newtheorem{rem}[equation]{Remark}
\newtheorem{rem-intro}[thm-intro]{Remark}
\newtheorem{observation}[equation]{Observation}
\newtheorem{recollection}[equation]{Recollection}
\newtheorem{warning}[equation]{Warning}
\newtheorem{construction}[equation]{Construction}

\DeclareMathOperator{\LChyp}{LC^{\hyp}} 
\DeclareMathOperator{\LC}{LC}
\DeclareMathOperator{\Sing}{Sing}

\DeclareMathOperator{\exc}{exc}

\DeclareMathOperator{\hyp}{hyp}

\ifpersonal
\newcommand{\personal}[1]{\textcolor[rgb]{0,0,1}{(Personal: #1)}}
\newcommand{\discussion}[1]{\textcolor{violet}{(Discussion: #1)}}
\else
\newcommand{\personal}[1]{\ignorespaces}
\newcommand{\discussion}[1]{\ignorespaces}
\fi

\newcommand{\C}{\mathbb C}

\newcommand{\R}{\mathbb R}

\newcommand{\cA}{\mathcal A}
\newcommand{\cB}{\mathcal B}
\newcommand{\cC}{\mathcal C}
\newcommand{\cD}{\mathcal D}
\newcommand{\cE}{\mathcal E}
\newcommand{\cF}{\mathcal F}
\newcommand{\cH}{\mathcal H}
\newcommand{\cG}{\mathcal G}

\newcommand{\cM}{\mathcal M}

\newcommand{\cT}{\mathcal T}

\newcommand{\cW}{\mathcal W}
\newcommand{\cX}{\mathcal X}
\newcommand{\cY}{\mathcal Y}
\newcommand{\cZ}{\mathcal Z}
\DeclareFontFamily{U}{BOONDOX-calo}{\skewchar\font=45 }
\DeclareFontShape{U}{BOONDOX-calo}{m}{n}{<-> s*[1.05] BOONDOX-r-calo}{}
\DeclareFontShape{U}{BOONDOX-calo}{b}{n}{<-> s*[1.05] BOONDOX-b-calo}{}
\DeclareMathAlphabet{\mathcalboondox}{U}{BOONDOX-calo}{m}{n}

\makeatletter
\let\save@mathaccent\mathaccent
\newcommand*\if@single[3]{%
	\setbox0\hbox{${\mathaccent"0362{#1}}^H$}%
	\setbox2\hbox{${\mathaccent"0362{\kern0pt#1}}^H$}%
	\ifdim\ht0=\ht2 #3\else #2\fi
}
\newcommand*\rel@kern[1]{\kern#1\dimexpr\macc@kerna}
\newcommand*\widebar[1]{\@ifnextchar^{{\wide@bar{#1}{0}}}{\wide@bar{#1}{1}}}
\newcommand*\wide@bar[2]{\if@single{#1}{\wide@bar@{#1}{#2}{1}}{\wide@bar@{#1}{#2}{2}}}
\newcommand*\wide@bar@[3]{%
	\begingroup
	\def\mathaccent##1##2{%
		\let\mathaccent\save@mathaccent
		\if#32 \let\macc@nucleus\first@char \fi
		\setbox\z@\hbox{$\macc@style{\macc@nucleus}_{}$}%
		\setbox\tw@\hbox{$\macc@style{\macc@nucleus}{}_{}$}%
		\dimen@\wd\tw@
		\advance\dimen@-\wd\z@
		\divide\dimen@ 3
		\@tempdima\wd\tw@
		\advance\@tempdima-\scriptspace
		\divide\@tempdima 10
		\advance\dimen@-\@tempdima
		\ifdim\dimen@>\z@ \dimen@0pt\fi
		\rel@kern{0.6}\kern-\dimen@
		\if#31
		\overline{\rel@kern{-0.6}\kern\dimen@\macc@nucleus\rel@kern{0.4}\kern\dimen@}%
		\advance\dimen@0.4\dimexpr\macc@kerna
		\let\final@kern#2%
		\ifdim\dimen@<\z@ \let\final@kern1\fi
		\if\final@kern1 \kern-\dimen@\fi
		\else
		\overline{\rel@kern{-0.6}\kern\dimen@#1}%
		\fi
	}%
	\macc@depth\@ne
	\let\math@bgroup\@empty \let\math@egroup\macc@set@skewchar
	\mathsurround\z@ \frozen@everymath{\mathgroup\macc@group\relax}%
	\macc@set@skewchar\relax
	\let\mathaccentV\macc@nested@a
	\if#31
	\macc@nested@a\relax111{#1}%
	\else
	\def\gobble@till@marker##1\endmarker{}%
	\futurelet\first@char\gobble@till@marker#1\endmarker
	\ifcat\noexpand\first@char A\else
	\def\first@char{}%
	\fi
	\macc@nested@a\relax111{\first@char}%
	\fi
	\endgroup
}
\makeatother

\newcommand{\PSh}{\mathrm{PSh}}
\newcommand{\Sh}{\mathrm{Sh}}
\newcommand{\HSh}{\mathrm{Sh}^{\mathrm{hyp}}}

\newcommand{\Mod}{\mathrm{Mod}}

\newcommand{\CAlg}{\mathrm{CAlg}}

\newcommand{\Cat}{\categ{Cat}}
\newcommand{\cS}{\categ{Spc}}

\newcommand{\categ}[1]{\textbf{\textup{#1}}}

\newcommand{\cHom}{\cH \mathrm{om}}

\newcommand{\PrL}{\categ{Pr}^{\mathrm{L}}}

\newcommand{\PrLomega}{\categ{Pr}^{\mathrm L,\omega}}

\DeclareMathOperator{\Cons}{Cons}
\newcommand{\ConsPhyp}{\Cons_P^{\mathrm{hyp}}}

\newcommand{\ConsQhyp}{\Cons_Q^{\mathrm{hyp}}}
\newcommand{\ConsShyp}{\Cons_S^{\mathrm{hyp}}}
\newcommand{\Conshyp}{\Cons^{\mathrm{hyp}}}

\newcommand{\Exit}{\mathrm{Exit}}

\newcommand{\an}{^\mathrm{an}}

\newcommand{\inv}{^{-1}}
\newcommand{\id}{\mathrm{id}}

\newcommand{\op}{^\mathrm{op}}

\newcommand*{\longhookrightarrow}{\ensuremath{\lhook\joinrel\relbar\joinrel\rightarrow}}

\usetikzlibrary{decorations.markings} 
\tikzset{
  closed/.style = {decoration = {markings, mark = at position 0.5 with { \node[transform shape, xscale = .8, yscale=.4] {/}; } }, postaction = {decorate} },
  open/.style = {decoration = {markings, mark = at position 0.5 with { \node[transform shape, scale = .7] {$\circ$}; } }, postaction = {decorate} }
}

\DeclareMathOperator{\Fun}{Fun}
\DeclareMathOperator{\FunR}{Fun^R}
\DeclareMathOperator{\FunL}{Fun^L}

\DeclareMathOperator{\Hom}{Hom}

\DeclareMathOperator{\Map}{Map}

\DeclareMathOperator{\Sp}{Sp}

\DeclareMathOperator*{\colim}{colim}

\newcommand{\category}{$\infty$-category}
\newcommand{\categories}{$\infty$-categories}

\newcommand{\PrLotimes}{\mathcal{P}\mathrm{r}^{\mathrm L, \otimes}}

\newcommand{\bZ}{\mathbb Z}

\newcommand{\HomR}{\Hom^{\mathrm{R}}}

\begin{document}
\title{Topological exodromy with coefficients}

\author{Mauro PORTA}
\address{Mauro PORTA, Institut de Recherche Mathématique Avancée, 7 Rue René Descartes, 67000 Strasbourg, France}
\email{porta@math.unistra.fr}

\author{Jean-Baptiste Teyssier}
\address{Jean-Baptiste Teyssier, Institut de Mathématiques de Jussieu, 4 place Jussieu, 75005 Paris, France}
\email{jean-baptiste.teyssier@imj-prg.fr}


\begin{abstract}
	The exodromy equivalence relates the derived $\infty$-category of constructible sheaves on a stratified space $(X,P)$ with the $\infty$-category of representations of the \emph{exit path} $\infty$-category of $(X,P)$.
	Originally envisioned by MacPherson, it has been rigorously developed by Treumann and later improved by Lurie.
	This paper provides a new proof of the strongest version of this equivalence.
	This allows us to remove several limitations from Lurie's treatement; for instance we prove that the exodromy equivalence is functorial in arbitrary morphism of stratified spaces. We also remove all noetherianity assumptions, consider more general coefficients (e.g.\ compactly assembled or stable presentable $\infty$-categories), and we allow stratified spaces that have locally weakly contractible strata, rather than being locally of singular shape.
\end{abstract}

\maketitle


\tableofcontents

\section{Introduction}

\paragraph{\textbf{Context.}}

Let $X$ be a locally contractible topological space.
The monodromy correspondence is an equivalence between the abelian category $\LChyp(X;\mathrm{Ab})$ of locally constant sheaves of abelian groups on $X$ and the abelian category $\Fun(\Pi_1(X), \mathrm{Ab})$ of representations of the first fundamental groupoid of $X$.
This equivalence can be seen as a concrete way of bridging the topological world, incarnated by the locally constant sheaves, and the algebraic world, incarnated by the representations of the first fundamental groupoid of $X$.
Working with Verdier duality, it is very often desirable to work with a derived category of locally constant sheaves, but some care has to be applied: already when $X = S^2$ is the $2$-sphere, the monodromy correspondence shows that $\LC(S^2;\mathrm{Ab}) \simeq \mathrm{Ab}$, and so its derived category is unable to make the difference between $S^2$ and $\ast$.
Historically this has been fixed by defining the derived category of local systems as the full subcategory of the derived category of all sheaves spanned by those complexes with locally constant homologies.

\medskip

The advent of higher categorical techniques \cite{HTT} allowed for a more streamlined approach to this issue.
Replacing $\mathrm{Ab}$ by its derived category $\cD(\mathrm{Ab})$, it makes sense to consider the $\infty$-category of local systems with values in $\cD(\mathrm{Ab})$.
The resulting $\infty$-category $\LC(X;\cD(\mathrm{Ab}))$ allows to recover the singular cohomology of $X$ (see for instance \cite[Corollary 3.31]{HPT}), and it is therefore a far more complete invariant than $\LChyp(X;\mathrm{Ab})$.
It is equally possible to improve the monodromy correspondence, by replacing representations of the ($1$-)groupoid $\Pi_1(X)$ by representations of the $\infty$-groupoid $\Pi_\infty(X)$ (a.k.a.\ the homotopy type of $X$).
More precisely, there is a canonical equivalence
\[ \Fun(\Pi_\infty(X), \cD(\mathrm{Ab})) \simeq \LChyp(X;\cD(\mathrm{Ab})) \ , \]
where the right hand side denotes the $\infty$-category of hypersheaves that are locally the hypersheafification of a constant presheaf (the word ``hyper'' can be disregarded when $X$ is sufficiently finite dimensional).

\medskip

There is a second natural generalization of the monodromy correspondence, whose fundamental idea is due to MacPherson.
When $X$ is equipped with a stratification $P$, that is a continuous morphism $X \to P$ where $P$ is a poset endowed with the Alexandroff topology (cf.\ \cref{subsec:strat_spaces}), one can introduce a stratified variant of $\Pi_1(X)$, noted $\Pi_1(X,P)$.
This is a non-full subcategory of $\Pi_1(X)$ which is no longer a groupoid; its objects are the same (that is, the points of $X$), but the only morphisms that are allowed are those that \emph{exit} from lower strata to go to upper ones.
For this reason, the category $\Pi_1(X,P)$ is called the \emph{category of exit paths of $(X,P)$}.
In \cite[Theorem 5.7]{Treumann_Exit_paths}, D.\ Treumann studied this category and established an equivalence
\[ \Fun\big(\Pi_1(X,P), \mathrm{Ab}\big) \simeq \Cons_P(X;\mathrm{Ab}) \ , \]
where the right hand side denotes the abelian category of $\mathrm{Ab}$-valued constructible sheaves.
Following the influential work \cite{Barwick_Exodromy}, this equivalence is nowadays referred to as the \emph{exodromy equivalence}.
It was generalized to the higher categorical world by J.\ Lurie in \cite[Appendix A]{Lurie_Higher_algebra}.
His first step is to define a version $\Pi_\infty(X,P)$ of $\Pi_1(X,P)$ and establish a major result (Theorem A.6.4 in \emph{loc.\ cit.}) asserting that if the stratification $(X,P)$ is \emph{conical}, then $\Pi_\infty(X,P)$ is indeed an $\infty$-category.
Then, he proves:

\begin{thm-intro}[{J.\ Lurie, \cite[Theorem A.9.3 \& Proposition A.9.6]{Lurie_Higher_algebra}}]
	Let $X$ be a paracompact topological space which is locally of singular shape and equipped with a conical $P$-stratification, where $P$ is a partially ordered set satisfying the ascending chain condition.
	Then there is a natural equivalence
	\[ \Psi_{X,P} \colon \Fun\big(\Pi_\infty(X,P), \cS\big) \to \Cons_P(X;\cS) \ , \]
	where $\cS$ denotes the $\infty$-category of homotopy types (a.k.a.\ $\infty$-groupoids or animated sets).
	Moreover, this equivalence is functorial in maps of conically stratified spaces $f \colon (X,P) \to (Y,Q)$ such that $P = f\inv(Q)$.
\end{thm-intro}

\paragraph{\textbf{Main result.}}

The goal of this paper is to improve the above theorem as follows:

\begin{thm-intro}[{See \cref{thm:exodromy} \& \cref{cor:induced_stratification}}] \label{thm-intro:exodromy}
	Let $(X,P)$ be a conically stratified space with locally weakly contractible strata.
	Let $\cE$ be presentable $\infty$-category which is either stable or compactly assembled.
	Then there is a natural equivalence
	\[ \Psi_{X,P} \colon \Fun\big( \Pi_\infty(X,P), \cE \big) \stackrel{\sim}{\to} \Conshyp_P(X;\cE) \ , \]
	which is functorial in arbitrary maps of conically stratified spaces.
\end{thm-intro}

Let us briefly comment the differences between the two theorems.
To begin with, we replace the condition that $X$ is locally of singular shape with the weaker condition that the strata of $(X,P)$ are locally weakly contractible.
This is typically easier to check in situations arising from complex or $\R$-analytic geometry, as checking the former condition typically requires to build an explicit CW complex structure, while the local weak contractibility is simply a property that a space can have or not, and that is relatively easy to test.
More than an improvement, this is a \emph{trading}: the cost is to replace $\Cons_P(X;\cS)$ with its hypercomplete counterpart $\Conshyp_P(X;\cS)$.
As already observed by D.\ Lejay in \cite{Lejay_Constructible}, working with hypersheaves also allows to remove the ascending chain condition on the poset $P$.

\medskip

The two major improvements that our version offers are:
\begin{itemize}\itemsep=0.2cm
	\item allowing more general coefficients than just $\cS$;
	
	\item proving an unconditional functoriality in maps of conically stratified spaces.
\end{itemize}

\noindent Both these statements are consequence of a \emph{different proof} than the one given in \cite[Theorem A.9.3]{Lurie_Higher_algebra}.
Indeed, the proof in question relies on model categories, and is heavily built on the fact that functors from $\Pi_\infty(X,P)$ to $\cS$ can be realized as left fibrations over $\Pi_\infty(X,P)$.
These two facts together make it difficult to adapt the proof \emph{verbatim} to different coefficients than $\cS$.
It is also difficult to bootstrap on Lurie's result to deduce the same result for arbitrary presentable $\infty$-categories, although it \emph{is} possible to deal with the case of compactly generated $\infty$-categories in this way (see e.g.\ \cite[Theorem B.9]{Jansen_Stratified_Borel_Serre}).
We offer instead a \emph{purely $\infty$-categorical proof}, that is summarized below.

\medskip

\paragraph{\textbf{Summary of the proof.}}

We construct a correspondence interpolating between $\mathrm{Open}(X)\op$ and $\Pi_\infty(X,P)$ as follows.
Given a conically stratified space $(X,P)$, let $\mathrm E_X$ be the $\infty$-category informally defined as:
\begin{itemize}\itemsep=0.2cm
	\item its objects are pairs $(U,x)$, where $U$ is an open subset of $X$ and $x$ is an object in $\Pi_\infty(U,P)$, where $U$ is seen as a $P$-stratified space in the natural way;
	
	\item given two objects $(U,x)$ and $(V,y)$, the mapping space $\Map_{\mathrm E_X}((U,x), (V,y))$ is non-empty if and only if $V \subset U$, and in that case it coincides with $\Map_{\Pi_\infty(U,P)}(x,y)$.
\end{itemize}
Put otherwise,  $\mathrm E_X$ is the Grothendieck construction of the functor $\Pi_\infty \colon \mathrm{Open}(X) \to \Cat_\infty$.
We refer  to \cref{sec:categorical_framework} and to \cref{notation:key_correspondence} for the precise definition of $\mathrm E_X$.
This $\infty$-category comes with natural functors
\[ \begin{tikzcd}[column sep = small]
	{} & \mathrm E_X \arrow{dl}[swap]{\pi_X} \arrow{dr}{\lambda_X} \\
	\mathrm{Open}(X)\op & & \Pi_\infty(X,P) \ .
\end{tikzcd} \]
Then the functor $\Psi_{X,P}$ of Lurie can be identified with $\pi_{X,\ast} \circ \lambda_X^\ast$, where $\pi_{X,\ast}$ denotes the right Kan extension along $\pi_{X,\ast}$, and $\lambda_X^\ast$ denotes the restriction along $\lambda_X$.
The main bulk of the paper is dedicated to show, for more general coefficients than just $\cS$ and independently from Lurie's result, that this functor and its left adjoint $\Phi_{X,P} \coloneqq \lambda_{X,!} \circ \pi_X^\ast$ realize the exodromy equivalence.

\medskip

From a \emph{technical} viewpoint, we first study the functoriality behavior of both $\Psi_{X,P}$ and $\Phi_{X,P}$ under the operations of restriction to the strata.
The most challenging step is to prove that the formation of $\Phi_{X,P}$ is compatible with the restriction to open union of strata.
This is achieved in \cref{cor:Phi_restriction_open_strata}, and it depends crucially on \cref{technical_lemma_for_exit} which contains the key geometrical argument of the paper.

\medskip

Nevertheless, it is easier and perhaps more instructive to appreciate the difficulty in proving \cref{thm-intro:exodromy} reasoning along the following lines: given $F \in \HSh(X;\cE)$, a simple inspection reveals that $\Phi_{X,P}(F)$ evaluated at one point $x \in X$ seen as an object in $\Pi_\infty(X,P)$ can be written as
\[ \Phi_{X,P}(F)(x) \simeq \colim_{(U,y,\gamma)} F(U) \ , \]
where the colimit ranges over all morphisms $\gamma \colon y \to x$ in $\Pi_\infty(X,P)$ and over all possible open neighborhood $U$ of $y$.
The following drawing should help visualizing the situation:

	\begin{center}
	\begin{tikzpicture}
		\node (A) at (-2,3) {} ;
		\node (B) at (2,0) {} ;
		\node (C1) at (0.2,2.1) {} ;
		\node (C2) at (4,3.8) {} ;
		\node (C3) at (6,2.6) {} ;
		\node (C4) at (6,0.2) {} ;
		\node (C5) at (3.5,0.8) {} ;
		\draw (-2,3) .. controls (C1) .. node[pos=0.3,label=240:{\tiny $y$}] (y) {} (2,0) ;
		\draw[rounded corners] (-2,3) .. controls (C2) and (5.8,3.1) .. (6,2.6) .. controls (C4) and (C5) .. (2,0) ;
		
		\node (D1) at (-3.7,2.8) {} ;
		\node (D2) at (-4,2) {} ;
		\node (D3) at (-2,1.5) {} ;
		\node (D4) at (0,-1) {} ;
		\draw[rounded corners] (-2,3) .. controls (D1) .. (-3.8,2.3) .. controls (D3) .. (0,-1) .. controls (0.6,-0.6) .. (2,0) ;
		
		\filldraw[black] (y) circle (1pt) ;
		\draw[dashed] (y) circle (20pt) ;
		\draw[dashed] (y) circle (16pt) ;
		
		\node[label={[label distance=0.2pt]300:{\tiny$x$}}] (x) at (4,2.5) {} ;
		\filldraw[black] (x) circle (1pt) ;
		
		\draw[dashed] (x) circle(20pt) ;
		\draw[dashed] (x) circle (15pt) ;
		
		\draw[->] (y) .. controls (1,2) and (3,3) .. node[pos=0.5,above] {\tiny$\gamma$} (x) ;
		
		\node[label={[label distance=0.2pt]270:{\tiny$y'$}}] (y') at (2,1) {} ;
		\filldraw[black] (y') circle (1pt) ;
		\draw[dashed] (y') circle (15pt) ;
		\draw[->] (y') .. controls (2.6,2) and (3.6,1.8) ..  node[pos=0.5,below]{\tiny$\gamma'$}(x) ;
	\end{tikzpicture}
\end{center}

\noindent If we could simply limit ourselves to the case where $\gamma$ is the identity of $x$, this would then produce the stalk $F_x$ of $F$ at the point $x$ (informally, \cref{technical_lemma_for_exit} states that we can always restrict to paths $\gamma$ entirely contained in the stratum of $x$).
However, this is typically false and $\Phi_{X,P}(F)(x)$ should be rather thought of as an ``average over all the possible stalks at points close to $x$''.
Remarkably, when $F$ is $P$-hyperconstructible, we can prove that there is indeed an equivalence
\[ \Phi_{X,P}(F)(x) \simeq F_x \ . \]
See \cref{cor:Phi_formula}.
Although this is formally obtained as a consequence of our results, in many ways one should think of this equivalence as the key technical ingredient needed in the proof of our main theorem.
Interestingly, the left hand side only depends on the equivalence class of $x$ inside $\Pi_\infty(X,P)$, a statement that neatly encodes the \emph{parallel transport} on the right hand side.

\medskip

\paragraph{\textbf{Applications to stratified spaces.}}

In the rest of the paper, we explore the consequences of the exodromy correspondence.
First, we obtain a couple of structural results for constructible hypersheaves:

\begin{thm-intro}[See Corollaries \ref{cor:constructible_limits_colimits}, \ref{cor:criterion_constructibility}, \ref{cor:tensor_decomposition} and \ref{cor:categorical_Kunneth}]\label{thm-intro:structural_results}
	Let $(X,P)$ be a conically stratified space with locally weakly contractible strata.
	Let $\cE$ be a presentable $\infty$-category satisfying the assumptions of \cref{thm:exodromy}.
	Then:
	\begin{enumerate}\itemsep=0.2cm
		\item \emph{Stability under limits and colimits:} the $\infty$-category $\ConsPhyp(X;\cE)$ of $P$-hyperconstructible hypersheaves is presentable and closed under limits and colimits inside $\HSh(X;\cE)$;
		
		\item \emph{Recognition criterion:} a hypersheaf $F \in \HSh(X;\cE)$ is $P$-hyperconstructible if and only if for every pair of open subsets $U \subseteq V$ of $X$ for which the induced morphism $\Pi_\infty(U,P) \to \Pi_\infty(V,P)$ is a categorical equivalence, the restriction map
		\[ F(V) \to F(U) \]
		is an equivalence in $\cE$.
		
		\item \emph{Tensor decomposition:} the canonical equivalence $\HSh(X;\cE) \simeq \HSh(X) \otimes \cE$ restricts to an equivalence
		\[ \ConsPhyp(X;\cE) \simeq \ConsPhyp(X;\cS) \otimes \cE \ ; \]
		
		\item \emph{Categorical K\"unneth formula:} given a second conically stratified space $(Y,Q)$ with locally weakly contractible strata, there is a canonical equivalence
		\[ \Conshyp_{P \times Q}(X \times Y;\cS) \simeq \ConsPhyp(X;\cS) \otimes \Conshyp_Q(Y;\cS) \ . \]
	\end{enumerate}
\end{thm-intro}

Points (2), (3) and (4) are generalizations to the constructible setting of the analogous results already obtained in \cite{HPT} for locally hyperconstant hypersheaves.
Point (1) deserves a particular attention, as it is tightly related to more recent developments in stratified homotopy theory.
In the series of paper \cite{Jansen_Stratified_Borel_Serre,Clausen_Jansen,Jansen_Mgn,Jansen_toolbox} {\O}rsnes-Jansen and Clausen studied exit paths $\infty$-categories in certain examples very relevant in algebraic geometry (Borel-Serre compactifications and the moduli stack of stable curves $\overline{\cM}_{g,n}$).
In this situation, the topological framework of \cite{Lurie_Higher_algebra} does not immediately apply, and so they devised a way to \emph{define and control} exit paths out of the sole knowledge of the $\infty$-category of constructible sheaves.
We refer to \cite[Definition 3.5]{Clausen_Jansen} for the precise definition of what ``admitting an exit path'' means, but in a nutshell it amounts to ask that (i) $\ConsPhyp(X;\cS)$ is a category of presheaves (this is an intrinsic property that a category might or might not have), and that (ii) it is closed under limits and colimits in $\HSh(X)$.
In other words, the \cref{thm-intro:exodromy} and \cref{thm-intro:structural_results}-(1) say together that any \emph{conically} stratified space $(X,P)$ admits an exit path in the sense of Clausen and {\O}rsnes-Jansen.

\medskip

The general philosophy that emerges from their work and also highlighted by Ayala-Francis-Rozenblyum \cite[Problem 0.0.9]{Ayala_Francis_Rozenblyum_Stratified} is that one can trade the geometrically defined conical condition for a more abstract regularity condition on the $\infty$-category of constructible sheaves (namely, the closure under limits and colimits in the whole $\infty$-category of sheaves).
Adopting this point of view, in collaboration with P.\ Haine we introduced in \cite{Beyond_conicality} the notion of \emph{exodromic} stratified space, and we proved that every algebraic variety equipped with an algebraic stratification or every real analytic variety equipped with a locally finite subanalytic stratification are exodromic (notice that there are many such stratifications that fail to be Whitney).
We refer to the introduction of \cite{Beyond_conicality} for a more thorough discussion of the ideas involved.

\medskip

As a consequence of the unconditional functoriality of the exodromy equivalence obtained in \cref{thm-intro:exodromy}, we can prove the following structural result for the $\infty$-categories of exit paths:

\begin{cor-intro}[See \cref{cor:localization}]
	Let $(X,P)$ be a conically stratified space and let $(X,Q)$ be a conical refinement of $P$.
	Assume that the strata of $(X,P)$ and $(X,Q)$ are locally weakly contractible.
	Then the natural map
	\[ f \colon \Pi_\infty(X,Q) \to \Pi_\infty(X,P) \]
	is a localization.
\end{cor-intro}

\noindent Taking $P$ to be the trivial stratification, we deduce that $\Pi_\infty(X)$ is a localization of $\Pi_\infty(X,Q)$.
This result, which is of course extremely intuitive at the geometric level, was only known to the best of our knowledge in the conically smooth situation \cite[Proposition 1.2.13]{Ayala_Francis_Tanaka_Local_structures}.

\medskip

Finally, as further consequence of these results, we recover and generalize some of the results on Morita cohomology of J.\ V.\ Holstein \cite{Holstein_Morita_cohomology_I,Holstein_Morita_cohomology_II}:

\begin{cor-intro}[See \cref{thm:LC_modules_over_chains}]
	Let $X$ be a locally weakly contractible topological space and let $A$ be a $\mathbb E_\infty$-ring spectrum.
	Assume that $X$ is connected and let $x \in X$ be a point.
	Then there exists a canonical equivalence
	\[ \LChyp(X;\Mod_A) \simeq \Mod_{\mathrm C_\ast(\Omega_x(X);A)} \ , \]
	where $\mathrm C_\ast(\Omega_x(X);A) \coloneqq \Pi_\infty(\Omega_x(X)) \otimes A$ are the $A$-valued chains on the based loop space of $X$.
\end{cor-intro}

As further geometric consequences of this work, let us mention the construction of the derived stacks of perverse sheaves by Haine-Porta-Teyssier \cite{Haine_Porta_Teyssier_perverse} and their lagrangian structures by Christ-Lampetti \cite{Christ_Lampetti_Lagrangian},  as well as the construction  by Lampetti of good moduli spaces for the stacks of perverse sheaves \cite{Lampetti_Good_moduli} and constructible  sheaves \cite{Lampetti_Stokes}.

\medskip



\paragraph{\textbf{Acknowledgments.}}

We are grateful to Clark Barwick, Federico Binda, Damien Calaque, Denis-Charles Cisinski, Andrea Gagna, Julian V.\ S.\ Holstein, Damien Lejay, Alexandru Oancea, Marco Robalo, Francesco Sala, Olivier Schiffmann and Alberto Vezzani for useful conversations about this paper.
We especially thank Peter J.\ Haine, Enrico Lampetti, Guglielmo Nocera, Tony Pantev and Marco Volpe for their continuous support and interest in this work.
We wish to thank the Oberwolfach MFO institute that hosted the Research in Pairs ``2027r: The geometry of the Riemann-Hilbert correspondence'', where the initial bulk of this paper was conceived.
We also wish to thank the CNRS for the fundings PEPS ``Jeunes Chercheurs Jeunes Chercheuses'' and the ANR CatAG from which both authors benefited during the writing of this paper.

\section{Stratified spaces and hyperconstructible hypersheaves}

\subsection{Stratified spaces} \label{subsec:strat_spaces}

The main reference for the definitions below is \cite[Appendix A]{Lurie_Higher_algebra}.

\medskip

\begin{recollection}
If $P$ be a poset, we endow $P$ with the topology whose open subsets are the closed upward subsets $Q\subset P$.
     That is for every $a\in Q$ and $b\in P $ such that $b\geqslant a$, we have $b\in Q$.
\end{recollection}

\begin{defin}\label{notation_defin}
	Let $X$ be a topological space.
	Let $P$ be a poset.
	A  \textit{stratification of $X$ by $P$} is a continuous morphism $ X\to P$.
	For a subset $S\subset P$, we let $X_S\to S$ be the induced stratification and we denote by $i_S : X_S \to X$ the natural inclusion.
	For $a\in P$, the subset $X_a $ is the stratum of $(X,P)$ over $a$.
\end{defin}

\begin{rem}
	We abuse notations by denoting a stratification of $X$ by $P$ as $(X,P)$ instead of $ X\to P$ and refer to $(X,P)$ as a stratified space.
	The collection of stratified spaces organize into a category in an obvious manner.
\end{rem}

\begin{eg}
	Let $f \colon Y\to Q$ be a stratified space.
	Put $C(Y) \coloneqq \ast \sqcup \left( Y\times \mathbb{R}_{>0} \right)$.
	The set $C(Y)$ is endowed with the topology whose open subsets are the subsets $U\subset C(Y)$ such that $U\cap \left( Y\times \mathbb{R}_{>0}\right)$ is open and if $\ast \in U$, then 
	\[ C_{\varepsilon}(Y)\coloneqq \{\ast\} \bigsqcup \left( Y\times (0,\varepsilon) \right)\subset U \]
	for some $\varepsilon >0$.
	Let $P$ be the poset obtained from $Q$ by adding a smallest element $-\infty$.
	We define a continuous map $g \colon C(Y)\to P$ by sending $\ast$ to $-\infty$  and $(y,t)\in Y\times \mathbb{R}_{>0} $ to $f(y)$.
	We refer to $(C(Y),P)$ as the cone of $(Y,Q)$.
\end{eg}

Following \cite[A.6.2]{Lurie_Higher_algebra}, we introduce the

\begin{defin}
	Let $(X,P)$ be a stratified space.
	We define $\Exit(X,P)$ as the simplicial subset of $\Sing X$ formed by the simplices $\sigma \colon |\Delta^n|\to X$ such that there exists a chain $a_1\leqslant\cdots\leqslant a_n$ of elements of $P$ such that for every $(t_0,\dots, t_i,0,\dots, 0)\in |\Delta^n|$ with $t_i> 0$, we have
	$\sigma(t_0,\dots, t_i,0,\dots, 0)\in X_{a_i}$.
	
	\smallskip
	
	The simplicial set $\Exit(X,P)$ is the  \textit{exit-path simplicial set of $(X,P)$}.
\end{defin}

\begin{notation}\label{small_rem}
	If $X$ is trivially stratified $\Pi_\infty(X)$ is the homotopy type  of $X$.
	For this reason, if $\Exit(X,P)$ is an $\infty$-category, we write $\Pi_\infty(X,P)$ instead of $\Exit(X,P)$.
\end{notation}

The following lemma follows immediately from the definition of the exit-paths.

\begin{lem}\label{fully_faith_XS}
	Let $(X,P)$ be a stratified space.
	Let $S\subset P$ be a subset.
	Assume that $\Exit(X,P)$ and $\Exit(X_S,S)$ are $\infty$-categories.
	Then, the natural inclusion $\Exit(X_S,S)\to \Exit(X,P)$ is fully-faithful.
\end{lem}

\begin{lem}\label{Exit_cone_initial_object}
	Let $(Y,Q)$ be a stratified space and set $P \coloneqq Q^{\vartriangleleft}$.
	If $\Exit(C(Y),P)$ is an $\infty$-category, then $\ast$  is an initial object in $\Exit(C(Y),P)$.
\end{lem}

\begin{proof}
	Since $\ast$ is the only point of its stratum, $\Map_{\Exit(C(Y),P)}(*,*)$ is contractible.
	Let $(y,\varepsilon)\in Y \times \mathbb{R}_{>0}$.
	We are going to construct a deformation retract of
	\[  \Map_{\Exit(C(Y),P)}(*,(y,\varepsilon)) \]
	on the simplicial subset spanned by the exit path $\gamma \colon [0,1]\to C(Y)$ sending $0$ to $\ast$ and $t$ to $(y,t\varepsilon)$.
	Since $\Exit(C(Y), P)$ is an $\infty$-category by assumption, \cite[Proposition 2.2.2.13 \& Corollary 4.2.1.8]{HTT} imply that we can model the above mapping space via the right simplicial hom $\HomR(\ast, (y,\varepsilon))$ (see \cite[\S1.2.2]{HTT} for the notation).
	We will therefore construct a homotopy
	\[ H \colon \HomR(\ast,(y,\varepsilon)) \times \Delta^1 \to \HomR(\ast,(y,\varepsilon)) \]
	between the map constant to $\gamma $ and the identity of $\HomR(\ast,(y,\varepsilon))$.
	At the cost of writing $\HomR(\ast,(y,\varepsilon))$ as a colimit of its simplices $\sigma \colon \Delta^n \to \HomR(\ast,(y,\varepsilon))$, it is enough to construct an homotopy $H$ between the $n$-simplex constant to $\gamma$ and $\sigma$. 
	That is, we have to construct
	\[ H \colon \Delta^n\times \Delta^1 \to \HomR(\ast,(y,\varepsilon)) \]
	such that $H|_{\Delta^n\times \Delta^{\{0\}}}$ is constantly equal to $\gamma$ and $H|_{\Delta^n\times \Delta^{\{1\}}}=\sigma$.
	Set $Z \coloneqq |\Delta^n|\times [0,1]$ and $Z_\lambda \coloneqq |\Delta^n|\times \{\lambda\}$ for every $\lambda \in [0,1]$.
	By definition, the above construction amounts to the construction of a continuous map
	\[ H \colon Z\star \Delta^0 \to C(Y) \]
	satisfying the following conditions :
	\begin{enumerate}\itemsep=0.2cm
		\item $H|_{Z_0\star \Delta^0}$ is constantly equal to $\gamma$ and $H|_{Z_1\star \Delta^0} = \sigma$. 
		\item $H$ sends the base of $Z\star \Delta^0$ to $\ast$ and the tip of $Z\star \Delta^0$ to $(y,\varepsilon)$.
		\item $H$ sends $Z\times ]0,1]$ in the stratum of $(y,\varepsilon)$.
	\end{enumerate}
	For $(u,t)\in Z_1\times ]0,1]$, put
	\[ \sigma(u,t)=(y(u,t),\varepsilon(u,t)) \in Y \times \mathbb{R}_{>0} \]
	By definition $\sigma(u,1)=(y,\varepsilon)$.
	For  $(u,\lambda,t)\in Z\star \Delta^0$, we define
	\[ H(u,\lambda,t)= \begin{cases}
		\big( y(u,t/\lambda), \lambda \varepsilon(u,t/\lambda) \big) & \text{if } 0 < t \leqslant \lambda \ , \\
		\gamma(t) & \text{if } \lambda \leqslant t \ , \\
		\ast & \text{if } t = 0 \ .
	\end{cases} \]
	Observe that $H$ satisfies conditions (1), (2), (3).
	To conclude the proof of \cref{Exit_cone_initial_object}, we have to show that $H$ is continuous.
	Away from $t=0$, the map $H$ is continuous as the glueing of two continuous maps along $t=\lambda$.
	Let $(u_0,\lambda_0,0) \in Z\star \Delta^0$.
	We want to check that $H$ is continuous at 	$(u_0,\lambda_0,0)$.
	If $\lambda_0>0$, we have to check that $H$ is continuous on the open set $\lambda >t$, which is true since $\sigma$ is continuous.
	The continuity of $H$ at $(u_0,0,0)$ follows from the definition of the topology of $C(Y)$ and the observation that $\varepsilon$ is bounded. 
\end{proof}

\begin{defin}\label{def_conical}
	Let $(X,P)$ be a stratified space.
	We say that $(X,P)$ is \textit{conically stratified} if for every point $x\in X$ lying over $a\in P$, there exists a topological space $Z$, a stratified space $(Y,P_{> a})$ and a morphism of stratified spaces  $(Z \times C(Y), P_{\geqslant a}) \to (X,P)$ inducing an homeomorphism between  $Z \times C(Y)$ and an open neighbourhood of $x$ in $X$.
\end{defin}

\begin{rem}
	The open sub-stratified space $(Z \times C(Y), P_{\geqslant a})$ of $(X,P)$ in \cref{def_conical} is a conical chart of $(X,P)$ at $x$.
	By definition of the topology of $C(Y)$, the set of conical charts of  $(X,P)$ at $x$ form a fundamental system of open neighbourhoods of $x$ in $X$.
\end{rem}

Conically stratified spaces enjoy the following stability property:

\begin{lem}\label{union_of_strata_and_conicality}
	Let $(X,P)$ be a conically stratified space.
	Let $S\subset P$ be a subset.
	Then $(X_S,S)$ is a conically stratified space.
\end{lem}

\begin{proof}
	Let $x\in X_S$ and let us show that $(X_S,S)$ is conical at $x$.
	Since $X$ is conical, we can suppose  that $(X,P)$ is of the form $(Z\times C(Y), P)$ where $Z$ is a topological space and where $(Y,Q)$ is a stratified space with $P = Q^{\lhd}$.
	Then $X_S=Z\times C(Y_{Q\cap S})$ and \cref{union_of_strata_and_conicality} is proved.
\end{proof}

The following theorem due to Lurie \cite[A.6.4]{Lurie_Higher_algebra} provides a wide range of stratified spaces whose exit-paths form an $\infty$-category:

\begin{thm}\label{Lurie_Thm}
	Let $(X,P)$ be a conically stratified space.
	Then $\Exit(X,P)$ is an $\infty$-category.
\end{thm}

\begin{rem}
	In view of \cref{Lurie_Thm}  and \cref{small_rem}, the $\infty$-category of exit-paths of a conically stratified space $(X,P)$ will be denoted as $\Pi_\infty(X,P)$.
\end{rem}

\begin{defin}
	Let $(X,P)$ be a conically stratified space.
	We say that $(X,P)$ is  \textit{locally weakly contractible} if every point $x\in X$ admits a fundamental system of open neighbourhoods $U$ such that  $x$ is an initial object of $\Pi_\infty(U,P)$.
\end{defin}

\begin{lem} \label{lem:categorical_homotopy_invariance}
	Let $f \colon X \to X'$ be a weak homotopy equivalence and let $(Y,Q)$ be a stratified space.
	Then the induced map
	\[ \Exit(X \times Y, Q) \to \Exit(X' \times Y, Q) \]
	is a categorical equivalence of simplicial sets.
\end{lem}

\begin{proof}
	Unraveling the definitions, we find canonical \emph{isomorphisms} of simplicial sets
	\[ \Exit(X \times Y, Q) \simeq \Sing(X) \times \Exit(Y,Q) \quad \text{and} \quad \Exit(X' \times Y, Q) \simeq \Sing(X') \times \Exit(Y,Q) \ . \]
	Since the map $f$ induces a categorical equivalence $\Sing(X) \to \Sing(X')$, the conclusion now follows from \cite[Corollary 2.2.5.4]{HTT}.
\end{proof}

\begin{lem}\label{infty_cat_and_products}
Let $K$ and $L$ be non empty simplicial sets such that $K\times L$ is an \category.
Then $K$ and $L$ are \categories.
\end{lem}

\begin{lem}\label{Exit_chart_is_category}
Let $Z$ be a non empty topological space and let $(Y,Q)$ be a stratified space.
Put $P \coloneqq Q^{\lhd}$ and assume that the stratified space $(Z\times C(Y),P)$ is conical.
Then, 
    \begin{enumerate}\itemsep=0.2cm
    \item For every $\varepsilon \in (0,+\infty]$, the simplicial set $\Exit(C_{\varepsilon}(Y),P)$ is an \category,
    \item The simplicial set $\Exit(Y,Q)$ is an \category.
    \end{enumerate}
\end{lem}
\begin{proof}
Let us prove (1).
Let $\varepsilon>0$.
There is an isomorphism of simplicial sets 
\[
\Exit(Z\times C_{\varepsilon}(Y),P)\simeq \Sing(Z)\times \Exit(C_{\varepsilon}(Y),P) \ .
\]
From Lurie's \cref{Lurie_Thm}, the left-hand side is an \category.
Hence, so is $\Exit(C_{\varepsilon}(Y),P)$ in virtue of \cref{infty_cat_and_products}.
Let us prove (2).
Since $(Z\times C(Y),P)$ is conical, so is the open subset complement to the stratum lying above the initial element of $P$.
Thus, $(Z\times Y \times (0,\varepsilon),Q)$ is conical.
Hence, the left-hand side of the following isomorphism of simplicial sets 
\[
\Exit(Z\times Y \times (0,\varepsilon),Q)\simeq \Sing(Z\times (0,\varepsilon))\times \Exit(Y,Q)
\]
is an \category.
Item (2) then follows from \cref{infty_cat_and_products}.
\end{proof}

\begin{prop} \label{prop:locally_contractible_strata}
	Let $(X,P)$ be a conically stratified space.
	The following conditions are equivalent:
	\begin{enumerate}\itemsep=0.2cm
	\item $(X,P)$ is locally weakly contractible.
    \item The strata of $(X,P)$ are locally weakly contractible.
	\end{enumerate}
\end{prop}
\begin{proof}
	\cref{fully_faith_XS} shows that $(1)$ implies $(2)$.
	Assume now that $(2)$ holds.
	Let $x \in X$ be a point.
	Since $X$ is conical, we can suppose that $X=Z \times C(Y)$  where $(Y,Q)$ is a stratified space with $P=Q^{\vartriangleleft}$.
    Let $\cW_x$ be a fundamental system of weakly contractible open neighborhoods of $x$ in $Z$.
	Then $\{ W \times C_\varepsilon(Y) \}_{W \in \mathcal W_x, \varepsilon>0}$ is a fundamental system of open neighborhoods of $x$ inside $X$.
	For $W\in  \mathcal W_x$ and $\varepsilon>0$, we  have an isomorphism of simplicial sets
	\[ \Exit(W \times C_\varepsilon(Y), P) \simeq   \Sing(W)\times \Exit(C_{\varepsilon}(Y),P) \]
	where each factor of the right hand side is an \category \ in virtue of \cref{Exit_chart_is_category}.
	Since $W$ is weakly contractible, we obtain an equivalence of $\infty$-categories 
	\[ \Exit(W \times C_\varepsilon(Y), P) \simeq  \Exit(C_{\varepsilon}(Y),P) \]
	sending $x$ to $\ast \in C(Y)$.
	We conclude using \cref{Exit_cone_initial_object}.
\end{proof}

\begin{rem}
	Let $(X,P)$ be a conically stratified space whose strata are CW-complexes.
	Then \cref{prop:locally_contractible_strata} states that $(X,P)$ is locally weakly contractible if and only if its strata are locally contractible.
\end{rem}

\subsection{Finiteness conditions on Exit Paths}

\begin{defin}\label{categorically_compact}
	For a stratified space $(X,P)$, we say that:
	\begin{enumerate}\itemsep=0.2cm
		\item $(X,P)$ is \emph{categorically compact} if $\Exit(X,P)$ is an $\infty$-category and a compact object in $\Cat_\infty$ ;

		\item $(X,P)$ is \emph{locally categorically compact} if $X$ admits a fundamental system of open subsets $U$ such that $(U,P)$ is categorically compact ;
		
		\item $(X,P)$ is \emph{of finite stratified type} if the poset $P$ is finite and for every $p \in P$, the homotopy type $\Pi_\infty(X_p)$  is a compact object in $\cS$.
	\end{enumerate}
\end{defin}

\begin{rem}
	In \cite{Beyond_conicality} it is established the existence of a large class of stratified spaces for which $\Exit(X,P)$ is not an $\infty$-category, but that nevertheless admit a stratified homotopy type $\Pi_\infty(X,P) \in \Cat_\infty$ making the exodromy correspondence \cref{thm:exodromy} true.
	It would be more sensible to formulate the above finiteness conditions in terms of this abstract stratified homotopy type.
	Nevertheless, this paper is first and foremost concerned with the conical situation, and in this case the subtlety between these two definitions disappear.
\end{rem}

When the stratification is trivial, there is an abundance of examples :

\begin{eg}\label{eg:Stein_are_compact}
Let $X$ be a smooth Stein complex space.
		Then classical results of Remmert \cite{Remmert_Stein_embeddings}, Bishop \cite{Bishop_Stein_embeddings} and Narasimhan \cite{Narasimhan_Stein_embeddings} show that $X$ can be realized as a closed subvariety of $\mathbb C^n$ for $n \gg 0$.
		In particular, Andreotti-Frankel's theorem \cite{Andreotti_Frankel} shows that $X$ has the homotopy type of a finite CW complex, and in particular $\Pi_\infty(X)$ is compact.
\end{eg}

\begin{eg}\label{eg:algebraic_varieties_are_compact}
Let  $X$ be a smooth algebraic variety.
		If $X$ is affine, then its analytification $X\an$ is smooth and Stein and therefore $\Pi_\infty(X\an)$ is compact by the previous point.
		In general, $X$ admits a finite cover by affine smooth open subvarieties, whose intersections are again smooth and affine.
		Thus, $\Pi_\infty(X\an)$ can be realized as a finite colimit of compact objects in $\cS$, and henceforth that $\Pi_\infty(X\an)$ is compact itself.
\end{eg}	
	
\begin{eg}\label{eg:compact_topological_manifolds_are_compact}
Let $X$ be a \emph{compact} topological manifold.
		Then the work of Kirby and Siebenmann \cite[Theorem III]{Kirby_Siebenmann_Hauptvermutung} implies that $X$ has the homotopy type of a \emph{finite} CW complex (although $X$ might not be triangulable itself).
		In particular, $\Pi_\infty(X)$ is a compact object in $\cS$.
\end{eg}	

The works of Volpe \cite{Volpe_Verdier_duality} and Nocera and Volpe \cite{Nocera_Volpe_Whitney_stratifications} provide many stratified examples of interest in singularity theory :

\begin{eg}[\cite{Volpe_Verdier_duality, Nocera_Volpe_Whitney_stratifications}]
	Whitney stratified spaces are locally categorically compact.
	Compact Whitney stratified spaces are categorically compact.
\end{eg}

\begin{eg}[\cite{Beyond_conicality}]
	It follows from \cite[Theorems 0.4.2 \& 0.4.3]{Beyond_conicality} that every compact real analytic manifold equipped with a conical locally finite stratification by subanalytic subsets is categorically compact in the above sense.
	The same holds for real algebraic varieties equipped with conical stratifications by Zariski locally closed subsets.
	In fact the conicality assumption can be dropped, at the cost of replacing $\Exit(X,P)$ with the stratified homotopy type $\Pi_\infty(X,P)$ constructed in \emph{loc.\ cit.}
\end{eg}

In the original version of this paper, we formulated the following conjecture:

\begin{conj}
	Let $(X, P)$ be a conically smooth stratified space (see \cite[\S3.2]{Ayala_Francis_Tanaka_Local_structures} for the definition of conical smoothness). If $(X, P)$ is of finite stratified type, then $(X,P)$ is also categorically compact.
\end{conj}

\noindent In the recent work of M.\ Volpe \cite{Volpe_Finiteness}, this has been answered positively.
In fact he shows that every conically stratified space whose strata and whose homotopy links are compact is categorically compact in the above sense.

\medskip

The goal of what follows is to prove the following concrete reformulation of local categorical compactness in the conical situation with locally weakly contractible strata.

\begin{prop}\label{local_cat_compact_Prop_def}
	Let  $(X,P)$ be a conically stratified space with locally weakly contractible strata.
	Then the following are equivalent :
	\begin{enumerate}\itemsep=0.2cm
	\item $(X,P)$ is locally categorically compact.
	\item For every $p \in P$ and every $x \in X_p$, there exists a conical chart of the form $Z \times C(Y)$ containing $x$ such that $(Y,P_{>p})$ is categorically compact and $Z$ is weakly contractible and locally weakly contractible.
	\end{enumerate}	
\end{prop}

To this end, some preparation is needed.

\begin{construction}\label{Exit_cone_exit_initial_object}
Let $(Y,Q)$ be a stratified space and put $P\coloneqq Q^{\lhd}$.
Let $\varepsilon \in (0,+\infty]$. 
Assume that $\Exit(Y,Q)$ and $\Exit(C_{\varepsilon}(Y),P)$ are \categories.
    Let $[0,\varepsilon)\to \Delta^1$ be the stratification with $\{0\}$ as closed stratum and $(0,\varepsilon)$ as open stratum.
    In particular, $([0,\varepsilon), \Delta^1)$ is conically stratified and the induced functor 
 \[
 \Exit([0,\varepsilon), \Delta^1)\to \Delta^1
 \]
 is an equivalence.
    Consider the commutative square of continuous maps 
	\[
		\begin{tikzcd}
		\text{$[0,\varepsilon )$}  \times Y  \arrow{r}  \arrow{d}& C_{\varepsilon}(Y)\arrow{d} \\
	\Delta^1 \times Q	\arrow{r} & P 
	\end{tikzcd}
	\]
where the top arrow sends $\{0\}\times Y$  to the vertex of $C_{\varepsilon}(Y)$ and induces an isomorphism above $Q\subset P$.
Looking at exit-paths thus yields a functor between \categories 
 \[
 \Delta^1\times \Exit(Y,Q)  \to \Exit(C_{\varepsilon}(Y),P)
 \]
sending $\Delta^0\times \Exit(Y,Q)$ to the vertex of $C_{\varepsilon}(Y)$ and inducing an equivalence on the full subcategories of objects lying over $Q\subset P$.
This in turn yields a functor over $P$
 \[
 \Exit(Y,Q)^{\lhd}  \to \Exit(C_{\varepsilon}(Y),P)
 \]
sending the initial object to the vertex of $C_{\varepsilon}(Y)$ and inducing an equivalence on the full subcategories of objects lying over $Q\subset P$.
\end{construction}

\begin{lem}\label{lem_comparison_exit}
Let $(Y,Q)$ be a stratified space and put $P\coloneqq Q^{\lhd}$.
Let $\varepsilon \in (0,+\infty]$.
Assume that $\Exit(Y,Q)$ and $\Exit(C_{\varepsilon}(Y),P)$ are \categories.
Then, the functor over $P$
\begin{equation}\label{construction_comparison_exit}
 \Exit(Y,Q)^{\lhd}  \to \Exit(C_{\varepsilon}(Y),P)
\end{equation}
 is an equivalence of \categories.
\end{lem}
\begin{proof}
The essential surjectivity of (\ref{construction_comparison_exit}) follows from the surjectivity of $[0,\varepsilon) \times Y  \to C_{\varepsilon}(Y)$.
We are thus left to show that (\ref{construction_comparison_exit}) is fully faithful.
As noted in \cref{Exit_cone_exit_initial_object}, the functor (\ref{construction_comparison_exit}) induces an equivalence on the full subcategories of objects lying over $Q\subset P$.
Let $x,y\in \Exit(Y,Q)^{\lhd}$ such that $x$ or $y$ lies over the initial object of $P$.
Then,  $x$ is the initial object of $ \Exit(Y,Q)^{\lhd}$.
Thus  $x$ is mapped to the vertex of $C_{\varepsilon}(Y)$.
Since the vertex of $C_{\varepsilon}(Y)$ is   initial in $\Exit(C_{\varepsilon}(Y),P)$ in virtue of \cref{Exit_cone_initial_object}, fully faithfulness follows.
\end{proof}

\begin{proof}[Proof of \cref{local_cat_compact_Prop_def}]
Let $p\in P$ and let $x\in X_p$ be a point. 
At the cost of shrinking $X$, we can suppose that $p$ is the minimal element of $P$, that is $P=P_{\geq p}$.
Let $U \coloneqq Z \times C(Y)$ be a conical chart at $x$  where $Z$ is contractible weakly contractible and where $(Y,Q)$ is a stratified  space with $Q=P_{>p}$.
From  \cref{Exit_chart_is_category}-(2),  the simplicial set $\Exit(Y,Q)$ is an \category.
We are going to show that $\Exit(Y,Q)$  is a compact object in $\Cat_{\infty}$.
Let $V \subset U$ be an open neighbourhood of $x$ such that $(V,P)$ is categorically compact.
Let $Z' \subset Z$ contractible weakly contractible and $\varepsilon >0$ such that $U' \coloneqq Z' \times C_{\varepsilon}(Y)\subset V$.
The inclusions $U' \subset V \subset U$ yield functors
\[
\Exit(Z' \times C_{\varepsilon}(Y),P) \to \Exit(V,P)  \to \Exit(Z \times C(Y),P) \ .
\]
Since $Z$ and $Z'$ are contractible,  \cref{lem:categorical_homotopy_invariance} implies that the above chain of simplicial sets is categorically equivalent to 
\[
\Exit(C_{\varepsilon}(Y),P)  \to \Exit(V,P)  \to\Exit(C(Y),P)  \ .
\]
Since $\Exit(C(Y),P)$,  $\Exit(C_{\varepsilon}(Y),P) $ and $\Exit(Y,Q)$ are \categories  \ by \cref{Exit_chart_is_category}, 
\cref{lem_comparison_exit} implies that the above chain of functors is equivalent to 
\[
 \Exit(Y,Q)^{\lhd}   \to\Exit(V,P) \to \Exit(Y,Q)^{\lhd}  \ .
\]
Hence,  $\Exit(Y,Q)^{\lhd}$  is a retract of a compact object of $\Cat_{\infty}$.
Thus  $\Exit(Y,Q)^{\lhd}$ is compact.
Hence $\Exit(Y,Q)$ is compact as a consequence of \cite[A.3.20]{Beyond_conicality}.
\end{proof}

\subsection{Local finality at strata}


\begin{defin}\label{Excellent_at_S_open}
	Let $(X,P)$ be a conically stratified space and let $S \subset P$ be a subposet.
	We say that an open neighbourhood $U$ of $X_S$ inside $X$ is \emph{final at $S$} if the functor
		\[ \Exit(X_S,S) \to \Exit(U,P) \]
		is a final.
\end{defin}

\begin{defin}\label{Excellent_at_S_stratified space}
		Let $(X,P)$ be a conically stratified space and let $S \subset P$ be a subposet.
	We say that $(X,P)$ is \emph{final at $S$} if the collection of final at $S$ open neighbourhoods of $X_S$ inside $X$ forms a fundamental system of neighbourhoods of $X_S$ inside $X$. 
\end{defin}

\begin{defin}
	Let $(X,P)$ be a conically stratified space and let $S \subset P$ be a subposet.
	We say that $(X,P)$ is \emph{locally final at $S$} if every point $x \in X_S$ admits a fundamental system of open neighbourhoods $U$ such that the stratified space $(U,P)$ is final at $S$.
\end{defin}

\begin{prop} \label{prop:conical_implies_local_excellency}
	 A conically stratified space $(X,P)$  is locally final at every stratum.
\end{prop}
\begin{proof}
	Let $a\in P$ and let $x\in X_a$.
	We can suppose that $P=P_{\geqslant a}$.
	Since $(X,P)$ is conically stratified, $x$ admits a fundamental system of open neighbourhoods of the form $Z \times C(Y)$  where $(Y,Q)$ is a stratified space, where $Z$ is an open set of $X_a$ containing $x$ and where $P=Q^{\vartriangleleft}$.
	To prove \cref{prop:conical_implies_local_excellency}, it is thus enough to prove that the conically stratified space $(Z\times C(Y),P)$ is final at $a$.
	In that case, $\{ Z \times C_\varepsilon(Y) \}_{\varepsilon>0}$ is a fundamental system of open neighborhoods of $Z$.
	On the other hand, there is an isomorphism of simplicial sets
	\[ \Exit(Z \times C_\varepsilon(Y), P) \simeq \Sing(Z)\times \Exit(C_{\varepsilon}(Y),P) \]
	Hence, the functor
	\[ \Exit(Z\times \ast,\{a\}) \to  \Exit(Z \times C_\varepsilon(Y),P) \]
	reads as a product
	\begin{equation}\label{final_functor}
		\Sing(Z)\times  \ast \to \Sing(Z)\times  \Exit(C_\varepsilon(Y),P) 
	\end{equation}
	From \cref{Exit_cone_initial_object}, the functor $\ast \to  \Exit(C_\varepsilon(Y),P)$  is final.
	Thus, (\ref{final_functor}) is final  as a consequence of \cite[4.1.1.13]{HTT}.
\end{proof}

\subsection{Hypersheaves, hyperconstancy and hyperconstructibility}

\begin{defin}
     For a topological space  $X$, we denote by
\[
\HSh(X) \subseteq \PSh(X)
\]     
the full subcategory spanned by hypersheaves on $X$.
    For $\cE\in \PrL$, we put
\[ 
\HSh(X;\cE) \coloneqq \HSh(X)\otimes \cE\simeq \FunR(\HSh(X)\op,\cE)  \ .
\]
\end{defin}

\begin{rem}
      The canonical inclusion $\HSh(X)\subset \PSh(X)$  admits a left adjoint $(-)^{\hyp} : \PSh(X)\to \HSh(X)$ referred to as the hypersheafification functor. 
      Tensoring with $\cE$ thus induces a functor 
\[ 
(-)^{\hyp} \colon \PSh(X)\to \HSh(X;\cE) 
\]
admitting a fully-faithful right adjoint.
Hence, we can think of objects of $\HSh(X;\cE)$ as $\cE$-valued preseheaves satisfying hyperdescent.
\end{rem}

    Let $f \colon X\to Y$ be a morphism of topological spaces and let $\cE\in \PrL$.
    Let
\[ 
f^\ast \colon \PSh(Y;\cE) \leftrightarrows \PSh(X;\cE) \colon f_\ast 
\] 
be the canonical adjunction.
     The functor $f_*$ preserves hypersheaves and thus restricts to a functor
\[ 
f_* \colon \HSh(X;\cE) \to \HSh(Y;\cE)  \ .
\]
     Furthermore, $f_*$  admits
\[ 
f^{\ast,\hyp} \coloneqq (-)^{\hyp} \circ f^* : \HSh(Y;\cE) \to \HSh(X;\cE) \ . 
\]
as left adjoint.

\medskip

In the next definition, $\Gamma_X : X\to \ast$ denotes the tautological morphism.

\begin{defin}
     Fix $F\in \HSh(X;\cE)$.
     We say that $F$ is  \textit{hyperconstant} if $F$ lies in the essential image of $\Gamma_X^{\ast,\hyp} \colon \cE \to \HSh(X;\cE)$.
     We say that $F$ is  \textit{locally hyperconstant} if there exists a cover of $X$ by open subsets $ i \colon U\hookrightarrow X$ such that $i^{\ast,\hyp}(F)$ is hyperconstant. 
	We denote by 
	\[
	\mathrm{LC}(X;\cE) \subset \HSh(X;\cE)
       \]	
	 the full subcategory of  $\HSh(X;\cE)$ spanned by locally hyperconstant hypersheaves.
\end{defin}

Following \cite{Lejay_Constructible}, we now introduce the main player of this paper:

\begin{defin}
	Let $(X,P)$ be a  stratified space and let $\cE\in \PrL$.
	An hypersheaf $F \in \HSh(X;\cE)$ is \textit{hyperconstructible} if for every $p \in P$, the hypersheaf $i_p^{\ast,\hyp}(F)$ is locally hyperconstant on $X_p$, where $i_p \colon X_p \hookrightarrow X$ is the canonical inclusion.
	Let 
	\[
	\ConsPhyp(X;\cE) \subset \HSh(X;\cE)
	\] 
	 be the full subcategory of  $\HSh(X;\cE)$ spanned by hyperconstructible hypersheaves on $(X,P)$.
\end{defin}

\begin{rem}
      When $\cE=\cS$, we will write $\ConsPhyp(X)$  for $\ConsPhyp(X;\cS)$.
\end{rem}

\begin{rem}
	Given a morphism of stratified spaces $ f \colon (X,P) \to (Y,Q) $ and $\cE\in \PrL$, the hypersheaf pullback $f^{\ast,\hyp} : \HSh(Y;\cE) \to \HSh(X;\cE)$ restricts to a functor
	\begin{equation*}
	\ConsQhyp(Y,\cE)\to \ConsPhyp(X,\cE) \ .	\end{equation*}
\end{rem}

\subsection{Change of coefficients and hyperconstructibility}\label{change_coeff_general}
    Let $ X $ be a topological space. 
    Let $ L \colon \cE \to \cD $ be a morphism in $ \PrL $ and let $ R  : \cD \to \cE $ its right adjoint.
	We denote by 
	\begin{equation*}
		L^{\hyp} \coloneqq \id \otimes L \colon \HSh(X;\cE) \to \HSh(X;\cD)
	\end{equation*}
	the induced  morphism in $ \PrL $.
	Concretely, $L^{\hyp}$  is obtained via the composition of  
\[
L \circ - : \PSh(X, \cE)  \to \PSh(X, \cD) 
\]
followed by the hypersheafification functor 
\[
(-)^{\hyp } : \PSh(X, \cD)   \to \HSh(X;\cD) \ .
\]
     Via the identification $\HSh(X;\cE) \simeq \FunR(\HSh(X)\op,\cE) $ supplied by \cite[Proposition 4.8.1.17]{Lurie_Higher_algebra}, the right adjoint $R^{\hyp}$ of $L^{\hyp}$  is given by  composing with $R$.

\begin{rem}\label{hyp_coeff_change_inverse_image}
	The formation of $ L^{\hyp} $ tautologically commutes with hypersheaf pullback. 
\end{rem}

As an immediate consequence of \cref{hyp_coeff_change_inverse_image}, we have the following 

\begin{lem}\label{hyp_change_coeff_preserves_Conshyp}
    Let $(X,P)$ be a stratified space and let $L \colon \cE \to \cD$ be a morphism in $\PrL$.
    Then, $L^{\hyp } : \HSh(X;\cE) \to  \HSh(X;\cD)$ restricts to a functor 
\[ L^{\hyp } \colon \ConsPhyp(X;\cE) \to \ConsPhyp(X;\cD) \ . \]
\end{lem}

\begin{rem}\label{no_right_adjoint}
	Note that $L^{\hyp }$ may not have a right adjoint in general since $R^{\hyp}$ may not preserve hyperconstructibility.
	We will show that when Exodromy holds, it does.
	See \cref{change_coef_Exodromy}.
\end{rem}  
  
\begin{lem}\label{g_preserves_LC}
	Let $X$ be a locally weakly contractible topological space.
    Let $L \colon \cE \to \cD$ be a morphism in $\PrL$ with right adjoint $R : \cD \to \cE$.
    Then, the commutative diagram
    \[ \begin{tikzcd}
    	\HSh(X;\cD) \arrow{r}{R^{\hyp}} \arrow{d}{\Gamma_{X,\ast}} & \HSh(X;\cE)\arrow{d}{\Gamma_{X,\ast}}  \\
		\cD	\arrow{r}{R} & \cE 
	\end{tikzcd} \]
	is vertically left adjointable.	  
	That is, the Beck-Chevalley transformation
     \[ \Gamma_X^{\ast,\hyp} \circ R\to  R^{\hyp}\circ \Gamma_X^{\ast,\hyp} \ \]
	is an equivalence. 
	In particular, $R^{\hyp} : \HSh(X;\cD) \to \HSh(X;\cE)$ restricts to a  functor
	\[ R^{\hyp}  \colon \LChyp(X,\cD)\to \LChyp(X,\cE)\ \]
\end{lem}

\begin{proof}
	Let $d\in \cD$ be an object.
	We have to show that 
	\[\Gamma_X^{\ast,\hyp} R(d)\to R^{\hyp}\circ \Gamma_X^{\ast,\hyp}(d) \]
	is an equivalence of hypersheaves on $X$.
	It is enough to check that the above morphism is an equivalence above any weakly contractible open subset $U$ of $X$.
	From \cite[Theorem 2.13]{HPT}, we have 
	\[(\Gamma_X^{\ast,\hyp} \circ R(d) )(U) \coloneqq\big(\Gamma_X^{-1}( R(d) )\big)^{\hyp}(U)  \simeq (\Gamma_X^{-1} \circ R(d) )(U)\simeq R(d) \]
	From \cite[Theorem 2.13]{HPT} again, we have 
	\[(R^{\hyp}\circ \Gamma_X^{\ast,\hyp}(d))(U)=R^{\hyp}\big( (\Gamma_X^{\ast,\hyp}(d))(U)\big)\simeq R(d)\]
	and \cref{g_preserves_LC} follows.
	\personal{There is no reason for $g\circ -$ to preserve hyperconstructible since there is no reason for $g\circ -$ to commute with hyperrestriction in general. 
	This comes from the fact that the colimit going in the definition of restriction is not constant unless one uses exodromy.}
\end{proof}

\begin{notation}\label{closed_immersion_upper_shriek}
     Let $\cE$ be a presentable stable $\infty$-category.
     For  a closed immersion $i : Y \to X$ between topological spaces, the functor 
\[
i_\ast : \HSh(Y;\cE) \to\HSh(X;\cE) 
\]
commutes with colimits, and thus admits a right adjoint
\[
i^{!,\hyp} : \HSh(X;\cE) \to\HSh(Y;\cE)  \  .
\]
\end{notation}

\begin{notation}
     Let $\cE$ be a presentable stable $\infty$-category.
     For  an open immersion $j : U \to X$ between topological spaces, denote by 
\[
j_! : \HSh(U,\cE) \to \HSh(X,\cE) 
\]
the left adjoint to $j^{\ast, \hyp} :  \HSh(X,\cE) \to \HSh(U,\cE)$.
\end{notation}

\begin{lem}\label{hyp_coeff_change_open_lower_shriek}
     Let $L \colon \cE \to \cD$ be a morphism in $\PrL$ with $\cE,\cD$ stable and let  $R : \cD \to \cE$ be its right adjoint.   
      For every open immersion $j : U\to X$ between topological spaces, the square
\[ 
\begin{tikzcd}
		\HSh(U;\cE)  \arrow{r}{j_!} \arrow{d}{L^{\hyp}} & \HSh(X;\cE)\arrow{d}{L^{\hyp}} \\
		\HSh(U;\cD)\arrow{r}{j_!}& \HSh(X;\cD) 
\end{tikzcd} \]
commutes.
\end{lem}
\begin{proof}
      By passing to right adjoints, it is enough to show that $R^{\hyp}\circ j^{\ast, \hyp}  =j^{\ast, \hyp}  \circ R^{\hyp}   $.
      Since $j$ is an open immersion, $j^{\ast, \hyp}$ identifies canonically with the presheaf inverse image $j^{*}$.
      The conclusion thus follows from the equality $R\circ j^{\ast}  =j^{\ast}  \circ R$.
\end{proof}

\begin{notation}\label{locally_closed_lower_shriek}
      Let $\cE$ be a presentable stable $\infty$-category.
      For  a locally closed immersion $i : Y \to X$ between topological spaces presented as the composition of a closed immersion $\iota : Y\to U$ followed by an open immersion $j : U\to X$, we put 
\[
i^{!,\hyp} \coloneqq \iota^{!,\hyp} \circ j^{\ast, \hyp}   : \HSh(X;\cE) \to\HSh(Y;\cE)  \  .
\]
Observe that $i^{!,\hyp}$  admits $j_! \circ \iota_\ast : \HSh(Y;\cE) \to \HSh(X;\cE)$ as left adjoint.
\end{notation}

\begin{eg}\label{i_S_shrieck}
	Let $(X,P)$ be a stratified space. 
	Let $S \subset P$ be a locally closed subset.
	Put
	\[ \geq S \coloneqq\{p\in P | \exists s\in S \text{ with } p\geq s\} \]
	The set $\geq S$ is open in $P$ and $S$ is closed in $\geq S$.
	Thus, the inclusion $i_S \colon X_S \to X$ factors as
	\[ \begin{tikzcd}[column sep = small]
		X_S \arrow{r}{\iota_S} & X_{\geqslant S} \arrow{r}{i_{\geqslant S}} & X \ ,
	\end{tikzcd} \]
	where $\iota_S$ is a closed immersion and $i_{\geqslant S}$ is an open immersion.
	Let $\cE$ be a presentable  $\infty$-category.
	Following \cref{locally_closed_lower_shriek}, we thus have a right adjoint
	\[ 
	i_S^{!, \hyp} \coloneqq \iota_S^{!,\hyp} \circ i_{\geqslant S}^{\ast,\hyp} \ \colon \HSh(X;\cE) \to \HSh(X_S;\cE)   \ .
	\]
	admitting admits $i_{\geq S,!} \circ \iota_{S,\ast} : \HSh(X_S ;\cE) \to \HSh(X;\cE)$ as left adjoint.
\end{eg}

\begin{warning}
The functor $i_S^{!, \hyp}$ may not preserve $P$-hyperconstructibility without any further assumption on $(X,P)$.
\end{warning}

\begin{lem}\label{hyp_and_closed_direct_image}
Let $i : Y \to X$ be a closed immersion between topological spaces.
Let $\cE$ be a presentable $\infty$-category.
Then, the commutative square
\[ 
\begin{tikzcd}
		 \PSh(X;\cE)  &\HSh(X;\cE) \arrow{l}\\
		\PSh(Y;\cE)\arrow{u}{i_\ast}&  \HSh(Y;\cE)\arrow{l}\arrow{u}{i_\ast}
\end{tikzcd} 
\]
is horizontally left adjointable.
	That is, the Beck-Chevalley transformation
	\[ \gamma \colon \hyp\circ i_\ast \to i_\ast \circ \hyp \]
	is an equivalence.
\end{lem}
\begin{proof}
Since $i : Y \to X$ be a closed immersion, the functor $i_\ast$ commutes with colimits.
By universal property of the tensor product of presentable $\infty$-categories, we are thus left to treat the case where $\cE=\cS$.
Observe that for every $x\in X$, the transformation $x^{\ast,\hyp }(\gamma)$ is an equivalence.
Since the hypersheaf stalk functors are jointly conservative on hypersheaves with values in $\cS$, the conclusion follows.
\end{proof}

\begin{lem}\label{right_adjoint_and_!_closed_immersion}
     Let $i : Y \to X$ be a locally closed immersion between topological spaces presented as the composition of a closed immersion $\iota : Y\to U$ followed by an open immersion $j : U\to X$.
     Let $L: \cE\to \cD$ be a morphism in $\PrL$ with $\cE,\cD$ stable and let  $R :\cD\to \cE$ be its right adjoint.
Then, the squares
\[ 
\begin{tikzcd}
		\HSh(Y;\cE)  \arrow{r}{L^{\hyp}} \arrow{d}{j_! \circ \iota_\ast} & \HSh(Y;\cD)\arrow{d}{j_! \circ \iota_\ast} \\
		\HSh(X;\cE)\arrow{r}{L^{\hyp}}& \HSh(X;\cD)  
		\end{tikzcd} 
		\text{ and }
		\begin{tikzcd}
		\HSh(Y;\cE)  &\HSh(Y;\cD)\arrow{l}[swap]{R^{\hyp}}\\
		\HSh(X;\cE)\arrow{u}{i^{!,\hyp}}&  \HSh(X;\cD)\arrow{l}[swap]{R^{\hyp}}\arrow{u}{i^{!,\hyp}}
\end{tikzcd}  
\]
commute.
\end{lem}
\begin{proof}
The commutativity of the right square is deduced from that of the left square by passing to right adjoints.
Since $f^{\hyp}= \hyp \circ f$ with $f$ trivially commuting to $\iota_\ast$, the commutativity of the left square follows from \cref{hyp_and_closed_direct_image} and \cref{hyp_coeff_change_open_lower_shriek}. 
\end{proof}

\section{The categorical framework} \label{sec:categorical_framework}

Let $\cX$ be an $\infty$-category and let $\mathsf A \colon \cX \to \Cat_\infty$ be a functor with values in small $\infty$-categories.
Throughout this section we either assume $\cX$ to be a small $\infty$-category with a terminal object $\mathbf 1_\cX$ or to be presentable.
In the latter case, we assume $\mathsf A$ to be an accessible functor.
We let
\[ \pi_{\mathsf A} \colon \cA \to \cX\op \]
be the associated \emph{cartesian} fibration.
Writing $\mathbf 1_\cX$ for the terminal object of $\cX$, \cite[Corollary 3.3.4.3]{HTT} provides a canonical \emph{localization} functor
\[ \lambda_{\mathsf A} \colon \cA \to \mathsf A(\mathbf 1_\cX) . \]

\begin{eg}\label{eg:0}
Let $X$ be a topological space.
		Take $\cX \coloneqq \mathrm{Open}(X)$ and $\mathsf A \coloneqq \Pi_\infty$  the functor $\mathsf{Open}(X) \to \cS$ sending an open set $U$ to its homotopy type $\Pi_\infty(U)$.
		We will also take $\cX \coloneqq \PSh(X)$ and $A \colon \PSh(X)\to \cS$ the colimit-preserving functor obtained as left Kan extension of $\Pi_\infty$ along the Yoneda embedding.
		From \cite[A.3.10]{Lurie_Higher_algebra}, $\mathsf A$ carries $\infty$-connective morphisms in $\PSh(X)$ to equivalences in $\cS$.
		If we finally put $\cX \coloneqq \HSh(X)$, the functor $\mathsf A$ thus factors as a colimit-preserving functor issued from $\HSh(X)$.
\end{eg}

\begin{eg}\label{eg:I}
Let $(X,P)$ be a conically stratified topological space.
		Take $\cX \coloneqq \mathrm{Open}(X)$ and $\mathsf A \coloneqq  \Pi_\infty$  the functor $\mathsf{Open}(X) \to \Cat_{\infty}$ sending an open set $U$ to $\Pi_\infty(U,P|_U)$.
		We will also take $\cX \coloneqq \PSh(X)$ and $\mathsf A \colon \PSh(X)\to \cS$ the colimit-preserving functor obtained as left Kan extension of $\Pi_\infty$ along the Yoneda embedding.
		Similarly as in \cref{eg:0}, $\mathsf A$ carries $\infty$-connective morphisms in $\PSh(X)$ to equivalences in $\Cat_{\infty}$.
		If we finally put $\cX \coloneqq \HSh(X)$, the functor $\mathsf A$ thus factors as a colimit-preserving functor issued from $\HSh(X)$.		
		We will abuse notations by using the notation $\Pi_\infty$ to also refer to the functors $\mathsf A$ and its factorization through $\HSh(X)$.
\end{eg}

\begin{rem}
	\cref{eg:0} is a special case of \cref{eg:I} in  the same way monodromy is a special case of exodromy.
	Still, the proof of the latter is a reduction to the trivial stratification case.
	It will therefore be useful to analyze \cref{eg:0} on its own.
\end{rem}

\subsection{Size issues}

The main point of this paper is to show that the exodromy adjunction is realized by push-pull along the correspondence $\pi_{\mathsf A} \times \lambda_{\mathsf A} \colon \cA \to \cX\op \times \mathsf A(\mathbf 1_\cX)$.
When $\cX$ is a presentable $\infty$-category, there may be size-theoretical issues in defining the right Kan extension along $\pi_{\mathsf A}$ and the left Kan extension along $\lambda_{\mathsf A}$.
The following lemmas show  this is not an actual problem:

\begin{lem}\label{lem:RKE}
	For every $x \in \cX$, the canonical functor
	\[ \mathsf A(x) \to \cA \times_{\cX\op} (\cX\op)_{x/} \]
	is limit-final.
	In particular, for every presentable $\infty$-category $\cE$, the pullback functor
	\[ \pi_{\mathsf A}^\ast \colon \Fun(\cX\op, \cE) \to \Fun(\cA, \cE) \]
	admits a right adjoint, given by right Kan extension along $\pi_{\mathsf A}$.
	\personal{Notice this applies also when $\cX$ is not presentable.}
\end{lem}

\begin{proof}
	The second half of the statement is a direct consequence of the first half and the fact that $\mathsf A(x)$ is a small $\infty$-category, which implies that the right Kan extension along $p$ is indeed well-defined.
	To prove the first half, we use Quillen's theorem A \cite[Theorem 4.1.3.1]{HTT}.
	Write $\cD \coloneqq \cA \times_{\cX\op} (\cX\op)_{x/}$.
	We can represent an object $\mathbf a \in \cD$ as a pair $\mathbf a = (a,f)$ where $a$ belongs to $\cA$ and $f \colon x \to \pi_{\mathsf A}(a)$ is a morphism in $\cX\op$.
	We have to prove that for every object $\mathbf a \in \cD$, the $\infty$-category
	\[ \mathsf A(x) \times_{\cD} \cD_{/ \mathbf a} \]
	is weakly contractible.
	By definition, an object of $\mathsf A(x) \times_{\cD} \cD_{/ \mathbf a}$ is a lift of $f$ with target $a$.
	In particular, $\mathsf A(x) \times_{\cD} \cD_{/ \mathbf a}$ identifies with $\cA_{/a} \times_{\cX\op_{/\pi_{\mathsf A}(a)}} \{f\}$.
	From \cite[2.4.1.9]{HTT}, the latter category admits a cartesian lift of $f$ as final object.
	It is thus weakly contractible.
\end{proof}

It is slightly more technical to deal with the left Kan extension along $\lambda_{\mathsf A}$.
It is useful to introduce the following definition:

\begin{defin}\label{def:cofiltered}
	Let $\kappa$ be a regular cardinal.
	We say that a functor $f \colon \cC \to \cD$ is $\kappa$-cofiltered if for every essentially $\kappa$-small $\infty$-category $I$, the functor $f$ has the right lifting property against the inclusion $I \hookrightarrow I^{\lhd}$.
	When $\kappa = \omega$, we simply say that $f$ is cofiltered instead of $\omega$-cofiltered.
\end{defin}

The proof of the following lemma is straightforward and it is left to the reader:

\begin{lem} \label{lem:filtered_functor}
	Let $\kappa$ be a regular cardinal.
	Then:
	\begin{enumerate}\itemsep=0.2cm
		\item the collection of $\kappa$-filtered functors is closed under composition and pullback in $\Cat_\infty$;
		
		\item the functor $\Gamma_\cC \colon \cC \to *$ is $\kappa$-filtered if and only if $\cC$ is $\kappa$-filtered.
	\end{enumerate}
\end{lem}

\begin{lem}\label{lem:LKE}
	Choose a regular cardinal $\kappa \gg 0$ such that $\cX$ is $\kappa$-presentable and $\mathsf A$ is $\kappa$-accessible.
	Let
	\[ \cA^{(\kappa)} \coloneqq (\cX^\kappa)\op \times_{\cX\op} \cA . \]
	Then for every $\alpha \in \mathsf A(\mathbf 1_\cX)$ the canonical functor
	\[ \cA^{(\kappa)} \times_{\mathsf A(\mathbf 1_\cX)} \mathsf A(\mathbf 1_\cX)_{/\alpha} \to \cA \times_{\mathsf A(\mathbf 1_\cX)} \mathsf A(\mathbf 1_\cX)_{/\alpha} \]
	is colimit-final.
	In particular, for every presentable $\infty$-category $\cE$, the pullback functor
	\[ \lambda_{\mathsf A}^\ast \colon \Fun(\mathsf A(\mathbf 1_\cX), \cE) \to \Fun(\cA, \cE)  \]
	admits a left adjoint, given by left Kan extension along $\lambda_{\mathsf A}$.
\end{lem}

\begin{proof}
	The second half of the statement is a direct consequence of the first half and the fact that $\cA^{(\kappa)} \times_{\mathsf A(\mathbf 1_\cX)} \mathsf A(\mathbf 1_\cX)_{/\alpha}$ is a small $\infty$-category, which implies that the left Kan extension along $\lambda_{\mathsf A}$ is indeed well-defined.
	To prove the first half, we use Quillen's theorem A.
	Write for simplicity
	\[ \cA_{/\alpha} \coloneqq \cA \times_{\mathsf A(\mathbf 1_\cX)} \mathsf A(\mathbf 1_\cX)_{/\alpha} \quad \text{and} \quad \cA^{(\kappa)}_{/\alpha} \coloneqq \cA^{(\kappa)} \times_{\mathsf A(\mathbf 1_\cX)} \mathsf A(\mathbf 1_\cX)_{/\alpha} . \]
	Fix an object $\mathbf \beta \in \cA_{/\alpha}$, which we represent as a pair $\mathbf \beta = (\beta,f)$, where $\beta \in \cA$ and $f \colon \lambda_{\mathsf A}(\beta) \to \alpha$ is a morphism in $\mathsf A(\mathbf 1_\cX)$.
	Write $\cA_{\mathbf \beta /\!\!/\alpha} \coloneqq (\cA_{/\alpha})_{\mathbf \beta /}$.
	Then, we have to prove that $\infty$-category
	\[ \cA^{(\kappa)}_{\mathbf \beta /\!\!/\alpha} \coloneqq \cA^{(\kappa)}_{/\alpha} \times_{\cA_{/\alpha}} \cA_{\mathbf \beta /\!\!/\alpha} \]
	is weakly contractible.
	We claim that it is cofiltered.
	To see this, start by writing $x \coloneqq p(\beta)$, so that $\beta$ can be seen as an element of $\mathsf A(x)$, and $\mathbf \beta$ as an element of $\mathsf A(x)_{/\alpha} \coloneqq \mathsf A(x) \times_{\mathsf A(\mathbf 1_\cX)} \mathsf A(\mathbf 1_\cX)_{/\alpha}$.
	We can thus set $\mathsf A(x)_{\mathbf \beta /\!\!/\alpha} \coloneqq (\mathsf A(x)_{/\alpha})_{\mathbf \beta/}$.
	Observe now that since $\pi_{\mathsf A}$ is a cartesian fibration, there is an induced commutative diagram
	\[ \begin{tikzcd}
		\cA^{(\kappa)}_{\mathbf \beta/\!\!/\alpha} \arrow{d} \arrow{r} & \cA^{(\kappa)}_{\beta /} \arrow{d} \\
		\mathsf A(x)_{\mathbf \beta/\!\!/\alpha} \arrow{r} & \mathsf A(x)_{\beta /} \ .
	\end{tikzcd} \]
	Inspection immediately reveals that it is a pullback square.
	Moreover, $\mathsf A(x)_{\mathbf \beta/\!\!/\alpha}$ has an initial object, and it is in particular cofiltered.
	Thus, \cref{lem:filtered_functor} reduces us to check that the map $\cA^{(\kappa)}_{\beta/} \to \mathsf A(x)_{\beta/}$ is cofiltered.
	Unraveling the definitions, we see that it is enough to check that for every finite category $I$ the canonical map
	\[ \colim_{y \in \cX^\kappa_{/x}} \Fun(I, \mathsf A(y)_{\beta/}) \to \Fun(I, \mathsf A(x)_{\beta/}) \]
	is an equivalence.
	This, however, is guaranteed from the assumption that $\cX$ is $\kappa$-filtered and that the functor $\mathsf A$ is $\kappa$-accessible.
\end{proof}

Fix a presentable $\infty$-category $\cE$.
\Cref{lem:RKE} provides  a  functor
\[ \Psi_{\cX, \mathsf A}^{\cE} \coloneqq \pi_{\mathsf A, \ast} \circ \, \lambda_{\mathsf A}^\ast \colon \Fun(\mathsf A(\mathbf 1_\cX), \cE) \to \Fun(\cX\op, \cE) . \]
From \cref{lem:RKE}, this functor is concretely described by the formula
\begin{equation}\label{eq:Psi_formula}
	\Psi_{\cX,\mathsf A}^{\cE}(F)(x) \simeq \lim_{\mathsf A(x)} F |_{\mathsf A(x)} \ ,
\end{equation}
where the restriction is performed along the map $\mathsf A(x) \to \mathsf A(\mathbf 1_\cX)$ induced by the canonical morphism $x \to \mathbf 1_\cX$.
On the other hand, \cref{lem:LKE} shows the existence of a second  functor
\[ \Phi_{\cX,\mathsf A}^\cE \coloneqq \lambda_{A,!} \circ \pi^\ast_{\cA} \colon \Fun(\cX\op, \cE) \to \Fun(\mathsf A(\mathbf 1_\cX),\cE) . \]
By construction, $\Phi_{\cX,\mathsf A}^\cE$ is left adjoint to $\Psi_{\cX,\mathsf A}^\cE$.
To increase readability, we will drop the superscript $\cE$ when there is no place for confusion.

\subsection{Change of coefficients}

Let $\cE$ be a presentable $\infty$-category.
For every presentable $\infty$-category $\cE^{\prime}$  and for every $\cC\in \Cat_{\infty}$, there is a canonical equivalence
\[ \Fun(\cC, \cE)\otimes \cE^{\prime} \simeq \Fun(\cC, \cE \otimes \cE^{\prime} ) \ . \]
This gives  an evident functoriality in $\cE^{\prime} $ with respect to morphisms in $\PrL$.
To see that the exodromy adjunction is compatible with this functoriality, consider first the following lemma:

\begin{lem}\label{lem:Phi_with_coefficients}
	Let $\cX$ be an $\infty$-category with a terminal object $1_\cX$.
	Let $\mathsf A \colon \cX \to \Cat_\infty$ be a functor.
	For every presentable $\infty$-categories $\cE$ and $\cE^{\prime}$, there is a canonical equivalence
	\[ \Phi_{\cX,\mathsf A}^{\cE\otimes \cE^{\prime}} \simeq \Phi_{\cX,\mathsf A}^{\cE} \otimes \id_{\cE^{\prime}} \ . \]
\end{lem}

\begin{proof}
	Since $\Phi_{\cX,\mathsf A}^{\cE\otimes \cE^{\prime}}$ commutes with colimits, it is enough to observe that for every $F \in \Fun(\cX\op, \cE)$ and $E^{\prime} \in \cE^{\prime} $, one has
	\[ \Phi_{\cX,\mathsf A}^{\cE}( F \otimes E^{\prime}  ) \coloneqq \lambda_{\mathsf A,!}( \pi_{\mathsf A}^\ast( F \otimes E^{\prime}  ) ) \simeq \lambda_{\mathsf A,!}( \pi_{\mathsf A}^\ast(F) \otimes E^{\prime}  ) \simeq \lambda_{\mathsf A,!}(\pi_{\mathsf A}^\ast(F)) \otimes E^{\prime}  \ , \]
	where the last equivalence follows from the fact that the external tensor product $\otimes \colon \cE\times \cE^{\prime}\to \cE\otimes \cE^{\prime}$ commutes with the colimits computing the left Kan extension in the first variable.
\end{proof}

\begin{cor} \label{cor:exodromy_change_of_coefficients}
Let $\cX$ be an $\infty$-category with a terminal object $1_\cX$.
Let $\mathsf A \colon \cX \to \Cat_\infty$ be a functor.
	Let $L \colon \cE \to \cE'$ be a morphism in $\PrL$.
	Then the diagram
	\[ \begin{tikzcd}
			\Fun(\cX\op, \cE) \arrow{d}[swap]{L\circ -} \arrow{r}{\Phi_{\cX,\mathsf A}^\cE } &\Fun(\mathsf A(\mathbf 1_\cX),\cE)\arrow{d}{L \circ -} \\
			\Fun(\cX\op, \cE') \arrow{r}{\Phi_{\cX,\mathsf A}^{\cE'} } &\Fun(\mathsf A(\mathbf 1_\cX),\cE')
	\end{tikzcd} \]
	is canonically commutative.
\end{cor}

\begin{proof}
	This follows from \cref{lem:Phi_with_coefficients} and the observation that under the identification $\Fun(\cC, \cE) \simeq \Fun(\cC, \cS) \otimes \cE$ for $\cC\in \Cat_{\infty}$, the functor $L \circ -$ corresponds to $\id_{\Fun(\cC, \cS)} \otimes L$.
\end{proof}

\subsection{The sheaf condition}

Assume now that $\cX$ is presentable and that the functor $\mathsf A$ commutes with colimits.
Fix a presentable $\infty$-category $\cE$.
Then the adjunction $\Phi_{\cX,\mathsf A} \dashv \Psi_{\cX,\mathsf A}$ can be refined thanks to the following observation:

\begin{lem} \label{lem:sheaf_condition}
	Assume that $\mathsf A \colon \cX \to \Cat_\infty$ commutes with colimits.
	Then $\Psi_{\cX,\mathsf A}$ factors through the full subcategory $\cX \otimes \cE \simeq \FunR(\cX\op, \cE)$ of $\Fun(\cX\op,\cE)$.
\end{lem}

\begin{proof}
	Fix a functor $F \colon \mathsf A(\mathbf 1_\cX) \to \cE$ and let $I \to \cX\op$ be a diagram, noted $i \mapsto x_i$, and let $x$ denote the limit of this diagram.
	By assumption, the canonical map
	\[ \colim_{i \in I\op} \mathsf A(x_i) \to \mathsf A(x) \]
	is an equivalence, the colimit being computed in $\Cat_\infty$.
	It follows that the canonical map
	\[ \Fun(\mathsf A(x), \cE) \to \lim_{i \in I} \Fun(\mathsf A(x_i), \cE) \]
	is an equivalence as well.
	Applying \cite[\S8.2]{Porta_Yu_Higher_analytic_stacks_2014}, we deduce a canonical equivalence
	\[ \lim_{\mathsf A(x)} F |_{\mathsf A(x)} \simeq \lim_{i \in I} \lim_{\mathsf A(x_i)} F |_{\mathsf A(x_i)} \ . \]
	Thus, $\Psi_{\cX,\mathsf A}(F)$ commutes with limits in $\cX\op$.
\end{proof}

\begin{cor} \label{cor:exodromy_adjunction}
	Assume that $\mathsf A \colon \cX \to \Cat_\infty$ commutes with colimits.
	Then for any presentable $\infty$-category $\cE$, the functors $\Phi_{\cX,\mathsf A}$ and $\Psi_{\cX,\mathsf A}$ induce an adjunction
	\[ \Phi_{\cX,\mathsf A} \colon \cX \otimes \cE \leftrightarrows \Fun(\mathsf A(\mathbf 1_\cX), \cE) \colon \Psi_{\cX,\mathsf A} \ . \]
\end{cor}

\begin{eg}\label{eg:II}
	Let us place ourselves again in the context of \cref{eg:I}.
	Let $(X,P)$ be a conically stratified space.
	To lighten the notation, we write
	\[ \Phi_{X,P}^{\hyp} \coloneqq \Phi_{\HSh(X),\Pi_\infty} \quad \text{and} \quad \Psi_{X,P}^{\hyp} \coloneqq \Psi_{\HSh(X),\Pi_\infty} \ . \]
	When the stratification is trivial we further simplify these notations by removing the subscript $P$.
	When $X$ and $\cE$ are clear from the context we also remove the corresponding decoration.
	Given $F \in \Fun(\Pi_\infty(X,P), \cE)$, \cref{lem:sheaf_condition} shows that the functor $\Psi(F)$ belongs to $\FunR(\HSh(X)\op;\cE) \simeq \HSh(X;\cE)$.
	The previous discussion shows that for an open subset $U$ of $X$, $\Psi(F)(U)$ is canonically given by the formula
	\begin{equation}\label{formula_for_psi}
		\Psi(F)(U) \simeq \lim_{\Pi_\infty(U, P|_U)} F |_{\Pi_\infty(U,P|_U)} . 
	\end{equation}	
	Finally, we have obvious variants
	\begin{gather*}
		\Phi_{X,P}^{\mathrm{psh}} \coloneqq \Phi_{\PSh(X),\Pi_\infty} \ , \qquad \Psi_{X,P}^{\mathrm{psh}} \coloneqq \Psi_{\PSh(X),\Pi_\infty} \ , \\
		\Phi_{X,P}^{\cE} \coloneqq \Phi_{\mathrm{Open}(X),\Pi_\infty} \ , \qquad \Psi_{X,P} \coloneqq \Psi_{\mathrm{Open}(X),\Pi_\infty} \ ,
	\end{gather*}
	obtained replacing $\HSh(X)$ by $\PSh(X)$ and $\mathrm{Open}(X)$ respectively.
\end{eg}

\subsection{Functoriality} \label{subsec:functoriality}

Let $f \colon \cX \to \cY$ be a functor.
We assume that $\cX$ and $\cY$ are either small with final objects, or presentable (importantly, we allow $\cX$ to be of the first kind and $\cY$ to be of the second).
Furthermore, we assume that $f(\mathbf 1_\cX) \simeq \mathbf 1_\cY$, and if both $\cX$ and $\cY$ are presentable then we additionally assume that $f$ is a left adjoint.
Let
\[ \mathsf B \colon \cY \to \Cat_\infty \quad \text{and} \quad \mathsf A \colon \cX \to \Cat_\infty \]
be two functors and let
\[ \gamma \colon \mathsf B \circ f \to \mathsf A \]
be a natural transformation.
We can summarize this information in the following diagram:
\begin{equation}\label{eq:functoriality}
	\begin{tikzcd}
		{} & \mathsf B(\mathbf 1_\cY) \arrow{rr}{\gamma_{\mathbf 1}} & & \mathsf A(\mathbf 1_\cX) \\
		\cB \arrow{ur}{\lambda_{\mathsf B}} \arrow{d}{\pi_{\mathsf B}} & & \cB_f \arrow{ul}[swap]{\lambda_f} \arrow{ll}[swap]{q} \arrow{d}{\pi_f} \arrow{r}{p} & \cA \arrow{dl}{\pi_{\mathsf A}} 	\arrow{u}[swap]{\lambda_{\mathsf A}} \\
		\cY\op & & \cX\op \arrow{ll}[swap]{f} \ ,
	\end{tikzcd}
\end{equation}
where $\cB_f$ and the maps $\pi_f$ and $q$ are defined by declaring that the bottom left square is a pullback, $p$ is induced by the natural transformation $\gamma$ and $\gamma_{\mathbf 1} \coloneqq \gamma_{\mathbf 1_\cX}$ is the value of $\gamma$ on the final object of $\cX$.

\begin{lem} \label{lem:functoriality}
	For every presentable $\infty$-category $\cE$, the diagram
	\[ \begin{tikzcd}
		\Fun(\cY\op, \cE) \arrow{r}{f^\ast} \arrow{d}{\pi_{\mathsf B}^\ast} & \Fun(\cX\op, \cE) \arrow{d}{\pi_{f}^\ast} \\
		\Fun(\cB, \cE) \arrow{r}{q^\ast} & \Fun(\cB_f, \cE)
	\end{tikzcd} \]
	is vertically right adjointable.
\end{lem}

\begin{proof}
	The existence of the right adjoints to $\pi_{\mathsf A}^\ast$ and $\pi_{\mathsf B}^\ast$  follows from \cref{lem:RKE}.
	In order to check that the Beck-Chevalley transformation
	\[ f^\ast \circ \pi_{\mathsf B,\ast} \to \pi_{f,\ast} \circ q^\ast \]
	is an equivalence, it is enough to fix $F \in \Fun(\cB, \cE)$ and $x \in \cX\op$.
	Then \cref{lem:RKE}  identifies $f^\ast(\pi_{\mathsf B,\ast}(F))(x)$ and $\pi_{f,\ast}(q^\ast(F))(x)$ with
	\[ \lim_{\mathsf B(f(x))} F |_{\mathsf B(f(x))} . \]
	The conclusion follows.
\end{proof}

\begin{notation}
	Assume that  $\cX$ and $\cY$ are presentable and that  $f \colon \cX \to \cY$ commutes with colimits.
	For any presentable $\infty$-category $\cE$ set $f_\cE \coloneqq f \otimes \id_\cE \colon \cX \otimes \cE \to \cY \otimes \cE$.
	We let $g_\cE$ denote the right adjoint to $f_\cE$.
\end{notation}

\begin{cor} \label{cor:functoriality}
	Fix a presentable $\infty$-category $\cE$.
	\begin{enumerate}\itemsep=0.2cm
		\item Assume that $\cX$ and $\cY$ are small.
		Then the diagrams
		\[ \begin{tikzcd}[column sep = 0.1pt]
			{} & \Fun(\mathsf B(\mathbf 1_\cY), \cE) \arrow{dl}[swap]{\Psi_{\cY,\mathsf B}} \arrow{dr}{\Psi_{\cX,\mathsf B \circ f}} \\
			\Fun(\cY\op, \cE) \arrow{rr}{f^\ast} & & \Fun(\cX\op, \cE)
		\end{tikzcd} \quad \text{and} \quad \begin{tikzcd}[column sep = 0.1pt]
			\Fun(\cX\op, \cE) \arrow{rr}{f_!} \arrow{dr}[swap]{\Phi_{\cX,\mathsf B \circ f}} & & \Fun(\cY\op, \cE) \arrow{dl}{\Phi_{\cY, \mathsf B}} \\
			{} & \Fun(\mathsf B(\mathbf 1_\cY),\cE)
		\end{tikzcd} \]
		are canonically commutative.
		
		\item Assume that $\cX$ and $\cY$ are presentable and that $\mathsf B$  and $f$ commute with colimits.
		Then the diagrams
		\[ \begin{tikzcd}[column sep = small]
			{} & \Fun(\mathsf B(\mathbf 1_\cY), \cE) \arrow{dl}[swap]{\Psi_{\cY,\mathsf B}} \arrow{dr}{\Psi_{\cX,\mathsf B \circ f}} \\
			\cY \otimes \cE \arrow{rr}{g_\cE} & & \cX \otimes \cE
		\end{tikzcd} \quad \text{and} \quad \begin{tikzcd}[column sep = small]
			\cX \otimes \cE \arrow{rr}{f_\cE} \arrow{dr}[swap]{\Phi_{\cX,\mathsf B \circ f}} & & \cY \otimes \cE \arrow{dl}{\Phi_{\cY, \mathsf B}} \\
			{} & \Fun(\mathsf B(\mathbf 1_\cY),\cE)
		\end{tikzcd} \]
		are canonically commutative.
	\end{enumerate}
\end{cor}

\begin{proof}
	It is enough to prove the commutativity of the left triangle of $(1)$.
	Breaking $\Psi_{\cY,\mathsf B}$ into its components, we have to show that the squares of the following diagram commute:
	\[ \begin{tikzcd}
		\Fun(\mathsf B(\mathbf 1_\cY),\cE) \arrow{r}{\lambda_{\mathsf B}^\ast} \arrow[equal]{d} & \Fun(\cB, \cE) \arrow{r}{\pi_{\mathsf B, \ast}} \arrow{d}{q^\ast} & \Fun(\cY\op, \cE) \arrow{d}{f^\ast} \\
		\Fun(\mathsf B(\mathbf 1_\cY), \cE) \arrow{r}{\lambda_{f}^\ast} & \Fun(\cB_f, \cE) \arrow{r}{\pi_{f, \ast}} & \Fun(\cX\op, \cE) \ .
	\end{tikzcd} \]
	For the left square, this just follows obviously from the definition of the functors.
	For the right square, this is a consequence of \cref{lem:functoriality}.
\end{proof}

\begin{eg} \label{eg:functoriality}
	Let $(X,P)$ be a conically stratified topological space.
	Take $\cX \coloneqq \mathrm{Open}(X)$, $\cY \coloneqq \PSh(X)$ and take $f$ to be the Yoneda embedding.
	We take $\mathsf A \coloneqq \Pi_\infty$, the exit paths $\infty$-functor and we take $\mathsf B$ to be the left Kan extension of $\mathsf A$ along $f$.
	Then
	\[ f^\ast \colon \FunR(\PSh(X)\op, \cE) \to \Fun(\mathrm{Open}(X)\op, \cE) \]
	is an equivalence as a consequence of \cite[Theorem 5.1.5.6]{HTT}. 
	Thus \cref{lem:sheaf_condition} $(1)$ and \cref{cor:functoriality} yield a natural equivalence
	\[ \Psi_{X,P} \simeq \Psi_{X,P}^{\mathrm{psh}} \ . \]
	\Cref{cor:functoriality} also identifies $\Phi_{X,P}^{\mathrm{psh}}$ with the left Kan extension of $\Phi_{X,P}$ along the Yoneda embedding.
\end{eg}

\begin{eg} \label{eg:functoriality_bis}		
	Let $(X,P)$ be a conically stratified topological space. 
	Take $\cX \coloneqq \PSh(X)$, $\cY \coloneqq \HSh(X)$ and take $f \coloneqq (-)^{\hyp}$ to be the hypersheafification functor.
		In this case, \cref{cor:functoriality} yields  natural equivalences
		\[ \Psi_{X,P}^{\mathrm{psh}} \simeq \Psi_{X,P}^{\hyp} \qquad \text{and} \qquad \Phi_{X,P}^{\mathrm{psh}} \simeq \Phi_{X,P}^{\hyp} \circ (-)^{\hyp} \ . \]
		In particular, we see that for every $F \in \PSh(X;\cE)$, the functor $\Phi_{X,P}$ takes the natural transformation $F \to F^{\hyp}$ to an equivalence.
\end{eg}

\begin{construction} \label{construction:comparison_morphisms}
	Associated to the top right square in the diagram \eqref{eq:functoriality} there is, for every presentable $\infty$-category $\cE$, a Beck-Chevalley transformation
	\[ \mathsf{BC} \colon \lambda_{f,!} \circ p^\ast \to \gamma_{\mathbf 1}^\ast \circ \lambda_{\mathsf A, !} \ , \]
	which induces the following transformation
	\[ \phi_{\mathsf A, \mathsf B, f, \gamma} \colon \Phi_{\cY,\mathsf B}^\cE \circ f_\cE \simeq \Phi_{\cX,\mathsf B \circ f}^\cE \simeq \lambda_{f,!} \circ \pi_f^\ast \simeq \lambda_{f,!} p^\ast \pi_{\mathsf A}^\ast \xrightarrow{\mathsf{BC}} \gamma_{\mathbf 1}^\ast \circ \lambda_{A,!} \circ \pi_A^\ast \simeq \gamma_{\mathbf 1}^\ast \circ \Phi_{\cX,\mathsf A}^\cE \ . \]
	In turn, associated to $\phi_{\mathsf A, \mathsf B, f, \gamma}$ there is a natural exchange transformation
	\[ \psi_{\mathsf A, \mathsf B, f, \gamma} \colon f_\cE \circ \Psi_{\cX,\mathsf A}^\cE \xrightarrow{\eta_\cY} \Psi_{\cY,\mathsf B}^\cE \Phi_{\cY,\mathsf B}^\cE f_\cE \Psi_{\cX,\mathsf A}^\cE \xrightarrow{\phi_{\mathsf A, \mathsf B, f, \gamma}} \Psi_{\cY,\mathsf B}^\cE \gamma_{\mathbf 1}^\ast \Phi_{\cX,\mathsf A}^\cE \Psi_{\cX,\mathsf A}^\cE \xrightarrow{\varepsilon_\cX} \Psi_{\cY,\mathsf B}^\cE \circ \gamma_{\mathbf 1}^\ast \ , \]
	where $\eta_{\cY}$ denotes the unit of the adjunction $\Phi_{\cY,\mathsf B}^\cE \dashv \Psi_{\cY,\mathsf B}^\cE$ and $\varepsilon_{\cX}$ denotes the counit of the adjunction $\Phi_{\cX,\mathsf A}^\cE \dashv \Psi_{\cX,\mathsf A}^\cE$.
	In the rest of the paper we will often write $\phi_f$ and $\psi_f$ instead of $\phi_{\mathsf A, \mathsf B, f, \gamma}$ and $\psi_{\mathsf A, \mathsf B, f, \gamma}$, respectively.
\end{construction}

\begin{rem} \label{rem:phi_equiv_iff_psi_equiv}
	The triangular identities imply that the exchange transformation associated to $\psi_{f}$ is once again $\phi_{f}$.
	In particular, if one knows that both $\Phi_{\cY,\mathsf B}^\cE \dashv \Psi_{\cY,\mathsf B}^\cE$ and $\Phi_{\cX,\mathsf A}^\cE \dashv \Psi_{\cX,\mathsf A}^\cE$ are equivalences, then $\phi_{f}$ is an equivalence if and only if $\psi_{f}$ is one.
	Furthermore, triangular identities also imply that the square
	\[ \begin{tikzcd}
		f \arrow{r}{\eta_\cY f} \arrow{d}{f \eta_\cX} & \Psi_{\cY, \mathsf B} \Phi_{\cY, \mathsf B} f \arrow{d}{\Psi_\cY \phi_f} \\
		f \Psi_{\cX, \mathsf A} \Phi_{\cX, \mathsf A} \arrow{r}{\psi_f \Phi_{\cX,\mathsf A}} & \Psi_{\cY,\mathsf B} \gamma_{\mathbf 1}^\ast \Phi_{\cX,\mathsf B}
	\end{tikzcd} \]
	is canonically commutative.
\end{rem}
		
\begin{rem} \label{rem:phi_equiv_iff_psi_equiv_bis}
Suppose that the natural transformation $\psi_f$ is an equivalence, so that it renders the square
		\[ \begin{tikzcd}
			\Fun(\mathsf A(\mathbf 1_\cX), \cE) \arrow{r}{\Psi_{\cX,\mathsf A}} \arrow{d}{\gamma_{\mathbf 1}^\ast} & \cX \otimes \cE \arrow{d}{f_\cE} \\
			\Fun(\mathsf B(\mathbf 1_\cY), \cE) \arrow{r}{\Psi_{\cY, \mathsf B}} & \cY \otimes \cE
		\end{tikzcd} \]
		commutative.
		Then this square is horizontally right adjointable \emph{if and only if} the natural transformation $\phi_f$ is an equivalence as well.
		Thus, if both $\phi_f$ and $\psi_f$ are equivalences, it follows formally that $f_\cE$ takes the unit of the adjunction $\Phi_{\cX,\mathsf A} \dashv \Psi_{\cX,\mathsf A}$ to the unit of the adjunction $\Phi_{\cY,\mathsf B} \dashv \Psi_{\cY,\mathsf B}$.
		Similarly, $\gamma_{\mathbf 1}^\ast$ takes the counit of the adjunction $\Phi_{\cX,\mathsf A} \dashv \Psi_{\cX,\mathsf A}$ to the counit of the adjunction $\Phi_{\cY,\mathsf B} \dashv \Psi_{\cY,\mathsf B}$.
\end{rem}

\begin{cor} \label{cor:Psi_open_restriction}
	Assume that $\cX$ and $\cY$ are presentable, that $\mathsf B$ commutes with colimits, that $f \colon \cX \to \cY$ has a left adjoint $h \colon \cY \to \cX$ and that the induced transformation
	\[ \mathsf B \xrightarrow{\eta} \mathsf B \circ f \circ h \xrightarrow{\gamma} A \circ h \]
	is an equivalence, where $\eta$ is the unit of the adjunction $h \dashv f$.
	Then the natural transformation
	\[ \psi_{f} \colon f_\cE \circ \Psi_{\cX,\mathsf A} \to \Psi_{\cY,\mathsf B} \circ \gamma_{\mathbf 1}^\ast \]
	is an equivalence.
\end{cor}

\begin{proof}
	Indeed, since $f$ has a left adjoint $h$, the same goes for $f_\cE$ and moreover under the equivalences $\cX \otimes \cE \simeq \FunR(\cX\op, \cE)$ and $\cY \otimes \cE \simeq \FunL(\cY\op,\cE)$, the functor $f_\cE$ corresponds to $h^\ast$.
	Unwinding the definitions and using \cref{lem:RKE}, we see that for every $F \in \Fun(\mathsf A(\mathbf 1_\cX), \cE)$ and every $y \in \cY\op$, the map $\psi_{f}(F)$ at $y$ is given by
	\[ \lim_{\mathsf A(h(y))} F |_{\mathsf A(h(y))} \to \lim_{\mathsf B(y)} (F \circ u)|_{\mathsf B(y)} \ . \]
	Thus, the conclusion follows from the fact that the functor $\mathsf B(y) \to \mathsf A(h(y))$ is an equivalence.
\end{proof}

\begin{eg} \label{eg:functoriality_topological_spaces}
	Let $f \colon (Y,Q) \to (X,P)$ be a morphism of conically stratified spaces.
	Take $\cX \coloneqq \HSh(X)$ and $\cY \coloneqq \HSh(Y)$ and $f^{\ast,\hyp} \colon \HSh(X) \to \HSh(Y)$ the induced geometric morphism.
	We let $\mathsf A$ and $\mathsf B$ be the exit paths functors
	\[ \Pi_\infty \colon \HSh(X) \to \Cat_\infty \qquad \text{and} \qquad \Pi_\infty \colon \HSh(Y) \to \Cat_\infty , \]
	respectively.
	The functoriality of the exit path construction yields a natural transformation $\gamma \colon \mathsf B \circ f^{\ast,\hyp} \to \mathsf A$.
	In this setup, we denote by
	\[ \phi_f^{\hyp} \colon \Phi_Y \circ f^{\ast,\hyp} \to \Pi_\infty(f)^\ast \circ \Phi_X \quad \textrm{and} \quad \psi_f^{\hyp} \colon f^{\ast,\hyp} \circ \Psi_X \to \Psi_Y \circ \Pi_\infty(f)^\ast \]
	the natural transformations of \cref{construction:comparison_morphisms}.
	Taking $\cX = \PSh(X)$ and $\cY = \PSh(Y)$ or $\cX = \mathrm{Open}(X)$ and $\cY = \mathrm{Open}(Y)$ we obtain obvious variants of this construction, that we denote respectively by $\phi_f^{\mathrm{psh}}$, $\psi_f^{\mathrm{psh}}$, $\phi_f$ and $\psi_f$.
\end{eg}

\begin{eg} \label{eg:functoriality_topological_spaces_bis}
In the setting of \cref{eg:functoriality_topological_spaces}, assume furthermore that the underlying morphism of topological spaces $f \colon Y \to X$ is an open immersion.
		Then $f^{\ast, \hyp} \colon \HSh(X) \to \HSh(Y)$ admits a fully faithful \emph{left} adjoint, denoted $f_!^{\hyp}$.
		Furthermore, if $U$ is an open subset of $Y$, seen as an object in $\HSh(Y)$, then $f_!^{\hyp}(U) \in \HSh(X)$ represents again $U$, seen as an open in $X$.
		It follows that the assumptions of \cref{cor:Psi_open_restriction} are satisfied and therefore that $\psi_f^{\hyp}$ is an equivalence in this case.
\end{eg}

\subsection{A glimpse of higher-functoriality}

The natural transformations $\phi_f$ and $\psi_f$ of \cref{construction:comparison_morphisms} enjoy themselves a form of higher functoriality.
We fix two composable morphisms $f \colon \cX \to \cY$ and $h \colon \cY \to \cZ$ in $\Cat_\infty$.
We further fix three functors
\[ \mathsf A \colon \cX \to \Cat_\infty , \qquad \mathsf B \colon \cY \to \Cat_\infty , \qquad \mathsf C \colon \cZ \to \Cat_\infty \ , \]
together with natural transformations
\[ \gamma \colon \mathsf B \circ f \to \mathsf A \qquad \text{and} \qquad \delta \colon \mathsf C \circ h \to \mathsf B \ . \]
As in the previous section, we assume that $\cX$, $\cY$, $\cZ$ are either small with a final objects or presentable.
As for $f$ and $h$, we assume that they preserve final objects and, if both their source and target are presentable, that they commute with colimits.

\medskip

Fix a presentable $\infty$-category $\cE$.
In this situation, \cref{construction:comparison_morphisms} provides three natural transformations
\[ \phi_h \colon \Phi_{\cZ,\mathsf C} \circ h_\cE \to \delta_{\mathbf 1}^\ast \circ \Phi_{\cY, \mathsf B} \ , \qquad \phi_f \colon \Phi_{\cY,\mathsf B} \circ f_\cE \to \gamma_{\mathbf 1}^\ast \circ \Phi_{\cX,\mathsf A} \]
as well as
\[ \phi_{h\!f} \colon \Phi_{\cZ,\mathsf C} \circ h_\cE \circ f_\cE \to \delta_{\mathbf 1}^\ast \circ \gamma_{\mathbf 1}^\ast \circ \Phi_{\cX,\mathsf A} \ . \]
We have:

\begin{prop}\label{2-functoriality_phi}
	The triangle
	\begin{equation}\label{eq:2-functoriality}
		\begin{tikzcd}[column sep = 30pt]
			\Phi_{\cZ,\mathsf C} \circ h_\cE \circ f_\cE \arrow{r}{\phi_h \circ f_\cE} \arrow{dr}[swap]{\phi_{h\!f}} & \delta_{\mathbf 1}^\ast \circ \Phi_{\cY,\mathsf B} \circ f_\cE \arrow{d}{\phi_{f}} \\
			{} & \delta_{\mathbf 1}^\ast \circ \gamma_{\mathbf 1}^\ast \circ \Phi_{\cX,\mathsf A}
		\end{tikzcd}
	\end{equation}
	is canonically commutative.
\end{prop}

\begin{proof}
	Consider the following diagram
	\[ \begin{tikzcd}[column sep = 20pt]
		{} & & & \mathsf A(\mathbf 1_\cX) \\
		{} & & \mathsf B(\mathbf 1_\cY) \arrow{ur}{\gamma_{\mathbf 1}} & & & & \cA \arrow[bend right = 10pt]{ulll}[swap]{\lambda_{\mathsf A}} \arrow[bend left = 30pt]{dddll}{\pi_{\mathsf A}} \\
		{} & \mathsf C(\mathbf 1_\cZ) \arrow{ur}{\delta_{\mathbf 1}} & & \cB \arrow{ul}[swap]{\lambda_{\mathsf B}} \arrow[bend left = 15pt]{ddl}[near end]{\pi_{\mathsf B}} & & \cB_{f} \arrow[bend left = 15pt]{ddl}{\pi_{f}} \arrow[bend right = 10pt]{ulll}[swap]{\lambda_{f}} \arrow{ll}[swap]{r} \arrow{ur}{u} \\
		\cC \arrow{d}{\pi_{\mathsf C}} \arrow{ur}{\lambda_{\mathsf C}} & & \cC_h \arrow{ul}[swap]{\lambda_h} \arrow{ur}[near start]{p} \arrow{ll}[swap]{q} \arrow{d}{\pi_h} & & \cC_{h\!f} \arrow[bend right = 10pt, crossing over]{ulll}[swap,near end]{\lambda_{h\!f}} \arrow{d}{\pi_{h\!f}} \arrow[crossing over]{ll}[swap]{t} \arrow{ur}{s} \\
		\cZ\op & & \cY\op \arrow{ll}[swap]{h} & & \cX\op \arrow{ll}[swap]{f}
	\end{tikzcd} \]
	Using \cref{cor:functoriality} twice, we find canonical equivalences
	\[ \alpha \colon \Phi_{\cZ,\mathsf C} \circ h_\cE \circ f_\cE \simeq \lambda_{h\!f,!} \circ \pi_{h\!f}^\ast , \qquad \beta \colon \Phi_{\cY,\mathsf B} \circ f_\cE \simeq \lambda_{f,!} \circ \pi_f^\ast \ . \]
	Write $\sigma_1$ for the square whose edges are $(\lambda_f, \delta_{\mathbf 1}, \lambda_{h\!f}, s)$ and $\sigma_2$ for the square with edges $(\lambda_{\mathsf A}, \gamma_{\mathbf 1}, \lambda_f, u)$.
	We also let $\sigma_{12}$ denote the rectangle obtained juxtaposing $\sigma_1$ and $\sigma_2$.
	Write
	\[ \mathsf{BC}(\sigma_1) \colon \lambda_{h\!f,!} \circ s^\ast \to \delta_{\mathbf 1}^\ast \circ \lambda_{f,!} \]
	for the Beck-Chevalley transformation associated to the square $\sigma_1$, and define similarly $\mathsf{BC}(\sigma_2)$ and $\mathsf{BC}(\sigma_{12})$.
	Unwinding \cref{construction:comparison_morphisms} and under the equivalences $\alpha$ and $\beta$ above, we are called to prove that the diagram
	\[ \begin{tikzcd}
		\lambda_{h\!f,!} \circ \pi_{h\!f}^\ast \arrow{r}{\mathsf{BC}(\sigma_1)} \arrow{dr}[swap]{\mathsf{BC}(\sigma_{12})} & \lambda_{f,!} \circ \pi_f^\ast \arrow{d}{\mathsf{BC}(\sigma_2)} \\
		& \lambda_{\mathsf A,!} \circ \pi_{\mathsf A}^\ast
	\end{tikzcd} \]
	is canonically commutative.
	Since $\sigma_{12}$ is the composite square of $\sigma_1$ and $\sigma_2$, this follows from the well-known componibility property of Beck-Chevalley transformations.
\end{proof}

\begin{rem}
	The above proposition is just a shadow of the actual functorial dependence of $\Phi_{\cX,\mathsf A}$ on the pair $(\cX,\mathsf A)$.
	To make this construction into an actual functor requires, as clearly shown by the above proof, to use the $(\infty,2)$-functoriality property of the Beck-Chevalley transformation.
	Since this is not needed for the current paper, we leave the details for the interested reader.
\end{rem}

\begin{cor}\label{2-functoriality_psi}
	In the same setting of \cref{2-functoriality_phi}, the triangle
	\[ \begin{tikzcd}[column sep = 30pt]
		h_\cE \circ f_\cE \circ \Psi_{\cX,\mathsf A} \arrow{r}{h_\cE \circ \psi_f} \arrow{dr}[swap]{\psi_{h\!f}} & h_\cE \circ \Psi_{\cY,\mathsf B} \circ \gamma_{\mathbf 1}^\ast \arrow{d}{\psi_h} \\
		& \Psi_{\cZ,\mathsf C} \circ \delta_{\mathbf 1}^\ast \circ \gamma_{\mathbf 1}^\ast
	\end{tikzcd} \]
	is canonically commutative.
\end{cor}

\begin{proof}
	We consider the case where $\cX$, $\cY$ and $\cZ$ are presentable, as every other case is dealt with similarly.
	We start with the following diagram:
	\[ \begin{tikzcd}
		\cX \otimes \cE \arrow{d}{\Phi_{\cX,\mathsf A}} \arrow{r}{f_\cE} & \cY \otimes \cE \arrow{r}{h_\cE} \arrow{d}{\Phi_{\cY,\mathsf B}} & \cZ \otimes \cE \arrow{d}{\Phi_{\cZ,\mathsf C}} \\
		\Fun(\mathsf A(\mathbf 1_\cX), \cE) \arrow{r}{\gamma_{\mathbf 1}^\ast} & \Fun(\mathsf B(\mathbf 1_\cY), \cE) \arrow{r}{\delta_{\mathbf 1}^\ast} & \Fun(\mathsf C(\mathbf 1_\cZ), \cE)
	\end{tikzcd} \]
	It is made $2$-commutative by the transformations $\phi_f$, $\phi_h$.
	\Cref{2-functoriality_phi} identifies their composite $\phi_f \circ (\phi_h \circ f_\cE)$ with $\phi_{h\!f}$.
	Since $\psi_f = \mathsf{BC}(\phi_f)$, $\psi_h = \mathsf{BC}(\phi_h)$ and $\psi_{h\!f} = \mathsf{BC}(\phi_{h\!f})$, the conclusion follows directly from the componibility property of the Beck-Chevalley transformations.
\end{proof}

\subsection{Monodromy, revisited}

Let us explain how to recover the monodromy equivalence as a particular instance of \cref{cor:exodromy_adjunction}.
Let $X$ be a  topological space and let $\cE$ be a presentable $\infty$-category.
From \cite[Construction 2.1]{HPT}, we have a functor
\[ \Pi_\infty^\cE \colon \HSh(X;\cE) \to \cE \ . \]
When $\cE = \cS$ is the $\infty$-category of spaces, this functor -- denoted simply by $\Pi_\infty$ -- is the left Kan extension of the homotopy type functor, sending an open $U \in \mathrm{Open}(X)$ to its homotopy type $\Pi_\infty(U)$.
For a general $\cE$, we have a canonical identification $\Pi_\infty^\cE \simeq \Pi_\infty \otimes \id_{\cE}$.
As observed in \emph{loc.\ cit.}, $\Pi_\infty^\cE$ has a right adjoint
\[ \Pi^\infty_\cE \colon \cE \to \HSh(X;\cE) \ . \]
Assume now that $X$ is locally weakly contractible. 
Then, \cite[Proposition 2.5]{HPT} implies that  $\Pi^\infty_\cE$ is canonically identified with the constant hypersheaf functor $\Gamma_X^{\ast,\hyp}$.

\medskip

If $\cE = \cS$, the functor $\Pi_\infty$ can be canonically lifted to a functor
\[ \widetilde{\Pi}_\infty \colon \HSh(X) \to \cS_{/\Pi_\infty(X)} \ , \]
which in turn admits a right adjoint $\widetilde{\Pi}^\infty$ sending $f \colon K \to \Pi_\infty(X)$ to the hypersheaf defined by
\begin{equation}\label{description_pi_tilde}
\widetilde{\Pi}^\infty(K)(U) \coloneqq \Map_{/\Pi_\infty(X)}(\Pi_\infty(U), K) \ .
\end{equation}
Put $\widetilde{\Pi}_\infty^\cE\coloneqq\widetilde{\Pi}_\infty \otimes \id_{\cE} $ and denote by $\widetilde{\Pi}^\infty_\cE$
the right adjoint of $\widetilde{\Pi}_\infty^\cE$.

\begin{lem}\label{factor_through_LC}
Let $X$ be a weakly contractible and locally weakly contractible topological space.
Let $\cE$ be a presentable $\infty$-category.
Then, $\widetilde{\Pi}^\infty_\cE$ factors through the full subcategory $\LChyp(X;\cE)$.
\end{lem}
\begin{proof}
	Since $X$ is weakly contractible, $\Pi_\infty(X) \simeq *$, and therefore $\widetilde{\Pi}_\infty$ identifies with $\Pi_\infty$.
	By tensoring with $\cE$, we deduce that $\widetilde{\Pi}_\infty^\cE$ identifies with $\Pi_\infty^\cE$.
	Hence, $\widetilde{\Pi}^\infty_\cE$ identifies with $\Pi^\infty_\cE $.
	Since $X$ is locally weakly contractible, \cite[Proposition 2.5]{HPT} supplies a canonical identification of $\Pi^\infty_\cE$ with $\Gamma_X^{\ast,\hyp} \colon \cE \to \HSh(X;\cE)$.
	The conclusion follows.
\end{proof}

\begin{lem}\label{cor:Phi_identification}
	Let $X$ be a locally weakly contractible topological space and let $\cE$ be a presentable $\infty$-category.
	Then $\Phi_X^{\hyp,\cE} $ and $\widetilde{\Pi}_\infty^\cE$ are canonically equivalent as functors from $\HSh(X;\cE)$ to $\Fun(\Pi_\infty(X),\cE)$.
	In particular, $\Psi_X^{\hyp,\cE} $ and $\widetilde{\Pi}^\infty_\cE$ are canonically equivalent.
\end{lem}
\begin{proof}
It is enough to show that $\Phi_X^{\hyp,\cE} $ and $\widetilde{\Pi}_\infty^\cE$ are canonically equivalent.
From \cref{lem:Phi_with_coefficients}, we have $\Phi_X^{\hyp,\cE}  \simeq \Phi_X^{\hyp} \otimes \id_\cE$.
Since $\widetilde{\Pi}_\infty^\cE=\widetilde{\Pi}_\infty \otimes \id_{\cE} $, we are left to show that $\Phi_X^{\hyp} $ and $\widetilde{\Pi}_\infty$ are canonically equivalent.
To do this, it is enough to show that $\Psi_X^{\hyp} $ and $\widetilde{\Pi}^\infty$ are canonically equivalent.
Under the straightening equivalence $\cS_{/\Pi_\infty(X)} \simeq \Fun(\Pi_\infty(X), \cS)$, the formula (\ref{description_pi_tilde}) for $\widetilde{\Pi}^\infty$ identifies canonically with the formula  (\ref{eq:Psi_formula}) for $\Psi^{\hyp}_X$.
Thus $\Psi_X^{\hyp} $ and $\widetilde{\Pi}^\infty$ are canonically equivalent.
\end{proof}

\begin{cor} \label{cor:Psi_produces_locally_constant}
	Let $X$ be a locally weakly contractible topological space.
	Then for every presentable $\infty$-category $\cE$ the functor $\Psi_X^{\hyp,\cE}  \colon \Fun(\Pi_\infty(X), \cE) \to \HSh(X;\cE)$ factors through the full subcategory $\LChyp(X;\cE)$.
\end{cor}
\begin{proof}
	Let $F \colon \Pi_\infty(X) \to \cE$ be a functor.
	Since $X$ is locally weakly contractible, it is enough to prove that for every weakly contractible open subset $U$ of $X$ the restriction $\Psi_X^{\hyp,\cE} (F) |_{U}$ is locally hyperconstant.
	Using \cref{cor:Psi_open_restriction} in the form of \cref{eg:functoriality_topological_spaces_bis}, we find a canonical equivalence
	\[ \Psi_X^{\hyp,\cE} (F) |_U \simeq \Psi_U^{\hyp,\cE} ( F |_{\Pi_\infty(U)} ) . \]
	Hence, we are left to suppose that $X$ is weakly contractible and locally weakly contractible.
	In that case, \cref{cor:Psi_produces_locally_constant} follows from the combination of \cref{cor:Phi_identification} and \cref{factor_through_LC}. 
\end{proof}

In particular, \cref{cor:Psi_produces_locally_constant} implies that the adjunction $\Phi_X^{\hyp} \dashv \Psi_X^{\hyp}$ restricts to an adjunction
\begin{equation} \label{eq:monodromy_revisited}
	\Phi_X^{\hyp}\colon \LChyp(X;\cE) \leftrightarrows \Fun(\Pi_\infty(X),\cE) \colon \Psi_X^{\hyp} \ .
\end{equation}
At this point, we have:

\begin{cor}\label{cor:monodromy_revisited}
	Let $X$ be a locally weakly contractible topological space.
	Let $\cE$ be a presentable $\infty$-category.
	Then the adjunction \eqref{eq:monodromy_revisited} is an equivalence.
\end{cor}

\begin{proof}
	Combine \cref{cor:Phi_identification} with the monodromy equivalence \cite[Corollary 3.9]{HPT}.
\end{proof}

\section{Exodromy adjunction for constructible hypersheaves}

Fix a conically stratified space $(X,P)$ and a presentable $\infty$-category $\cE$.
We specialize the discussion of \cref{sec:categorical_framework} to the context of \cref{eg:I}.
In other words, we take $\cX \coloneqq \HSh(X)$ and consider the functor
\[ \Pi_\infty \colon \HSh(X) \to \Cat_\infty \]
left Kan extended from the functor sending an open $U$ of $X$ to the $\infty$-category of exit paths $\Pi_\infty(U, P|_U)$.
Van Kampen's theorem for exit paths \cite[Theorem A.7.1]{Lurie_Higher_algebra} shows that $\Pi_\infty$ is a colimit preserving functor.
Thus, \cref{cor:exodromy_adjunction} provides for every presentable $\infty$-category $\cE$ an adjunction
\[ \Phi_{X,P}^{\hyp} \colon \HSh(X;\cE) \leftrightarrows \Fun(\Pi_\infty(X,P), \cE) \colon \Psi_{X,P}^{\hyp} \ , \]
which we refer to as the exodromy adjunction.
When the stratification on $X$ is trivial, Corollaries \ref{cor:Psi_produces_locally_constant} and \ref{cor:monodromy_revisited} show that this adjunction restricts to an equivalence between $\Fun(\Pi_\infty(X,P),\cE)$ and the $\infty$-category of locally hyperconstant hypersheaves on $X$.
Our goal in this section is to show that  the functor $\Psi_{X,P}^{\hyp}$ factors through the full subcategory $\Cons_P^{\hyp}(X;\cE)$ of constructible hypersheaves on $X$, and therefore that the exodromy adjunction restricts to an adjunction
\[ \Phi_{X,P}^{\hyp} \colon \Cons_P^{\hyp}(X;\cE) \leftrightarrows \Fun(\Pi_\infty(X,P),\cE) \colon \Psi_{X,P}^{\hyp} \ . \]
Since the case of trivial stratification is already known, the main point is to show that the functor $\Psi_{X,P}^{\hyp}$ is compatible with the restriction to strata.

\subsection{Criteria for $\psi_f^{\hyp}$ to be an equivalence}

Let $f \colon (Y,Q) \to (X,P)$ be a morphism of conically stratified spaces.
\cref{construction:comparison_morphisms}  associated to $f$ a natural transformation
\[ \psi_f^{\hyp} \colon f^{\ast,\hyp} \circ \Psi_{X,P}^{\hyp} \to \Psi_{Y,Q}^{\hyp} \circ \Pi_\infty(f)^\ast \ . \]
Our  goal is to prove that $\psi_f^{\hyp}$ is an equivalence when $f$ is the inclusion of a single stratum.
For this, we start laying out general locality properties of $\psi_f^{\hyp}$.

\begin{notation}
	Let $f \colon (Y,Q) \to (X,P)$ be a morphism of conically stratified spaces.
	For every open $U \subseteq X$ we let $Y_U \coloneqq U \times_X Y$ and we denote by $f_U \colon (Y_U,Q) \to (U,P)$ the induced morphism.
\end{notation}

The formation of the hypersheaf associated to a functor from the Exit Paths is local.
This is justified by the following

\begin{lem}\label{commutation_open_immersion}
Let $(X,P)$ be a conically stratified space.
Let $j : U \to X$ be the inclusion of an open subset of $X$. 
Then, the natural transformation
\[ \psi_j^{\hyp} \colon j^{\ast,\hyp} \circ \Psi_{X,P}^{\hyp} \to \Psi_{U,P}^{\hyp} \circ \Pi_\infty(j)^\ast \  \]
is an equivalence.
\end{lem}
\begin{proof}
Since $j : U \to X$ is the inclusion of an open subset of $X$, the functor $j^{\ast,\hyp} : \Sh^{\hyp}(X; \cE)\to \Sh^{\hyp}(U;\cE)$ is computed as the presheaf pullback $j^{-1}$.
On the other hand, $j^{-1}$ admits a fully-faithul left adjoint $j_!$.
Hence, $j_{\sharp}^{\hyp} \coloneqq \hyp \circ j_!$ is left adjoint to $j^{\ast,\hyp}$.
Since the hypersheafification commutes with $j^{-1}$, the unit transformation $\id \to j^{\ast,\hyp}\circ j_{\sharp}^{\hyp}$ is an equivalence.
Then, \cref{commutation_open_immersion} follows from \cref{cor:Psi_open_restriction}.
\end{proof}

\begin{lem}\label{two_psif_bis}
Let $f \colon (Y,Q) \to (X,P)$ be a morphism of conically stratified spaces.
Let $i : U \to X$ be the inclusion of an open subset of $X$. 
Let $j : f^{-1}(U) \to Y$ be the induced inclusion. 
Then, for every $F \in \Fun(\Pi_\infty(X,P),\cE)$, the transformation $j^{\ast,\hyp}(\psi_f^{\hyp}(F))$ is canonically equivalent to
$\psi_{f_U}^{\hyp}(\Pi_\infty(i)^\ast (F))$.
\end{lem}
\begin{proof}
Apply \cref{2-functoriality_psi} and \cref{commutation_open_immersion} twice.
\end{proof}

\begin{cor}\label{lem:localization}
	Let $f \colon (Y,Q) \to (X,P)$ be a morphism of conically stratified spaces.
	Let $\cB$ be a basis for the topology of $X$.
	Assume that for every $U\in \cB$, the transformation 	
		\[ \psi_{f_{U}}^{\hyp} \colon f_{U}^{\ast, \hyp} \circ \Psi_{U,P}^{\hyp} \to \Psi_{Y_U,Q}^{\hyp} \circ \Pi_\infty(f_{U})^\ast \]
		is an equivalence.
		Then, $\psi_f^{\hyp}$ is an equivalence.
\end{cor}
\begin{proof}
Fix $F \in \Fun(\Pi_\infty(X,P),\cE)$.
We have to show that  $\psi_f^{\hyp}(F)$ is an equivalence of hypersheaves on $Y$.
This statement is local on $Y$.
In particular, for $U \in \cB$, if $i : U\to X$ and $j  : f^{-1}(U)\to Y$ are the inclusions, it is enough to show that $j^{\ast,\hyp}(\psi_f^{\hyp}(F))$ is an equivalence.
Then, \cref{lem:localization} follows from \cref{two_psif_bis}.
\end{proof}

\begin{lem}\label{two_psif}
Let $f \colon (Y,Q) \to (X,P)$ be a morphism of conically stratified spaces. 
Then, $\psi_f^{\hyp}$ is canonically equivalent to
$\hyp(\psi_f^{\mathrm{psh}})$.
\end{lem}
\begin{proof}
Apply \cref{2-functoriality_psi} and \cref{eg:functoriality_bis} twice.
\end{proof}

\begin{cor}\label{criterion_psif}
Let $f \colon (Y,Q) \to (X,P)$ be a morphism of conically stratified spaces.
Fix $F \in \Fun(\Pi_\infty(X,P),\cE)$. 
Assume that the collection $\cB = \cB(f,F)$ of open subsets $V \subseteq Y$ for which the transformation
		\[ \psi^{\mathrm{psh}}_{f}(F)(V) \colon f\inv\big(\Psi_{X,P}^{\mathrm{psh}}(F)\big)(V) \to \Psi_{Y,Q}^{\mathrm{psh}}\big( \Pi_\infty(f)^\ast(F) \big)(V) \]
		is an equivalence is a basis for the topology of $Y$.
		Then, $\psi_f^{\hyp}(F)$ is an equivalence.		
\end{cor}
\begin{proof}
Let $j \colon \cB \hookrightarrow \mathrm{Open}(Y)$ be the canonical inclusion.
Consider the commutative square
\begin{equation}\label{eq:diagpsijj}
	\begin{tikzcd}
	f\inv\big(\Psi_{X,P}^{\mathrm{psh}}(F)\big)\arrow{rr}{\psi^{\mathrm{psh}}_{f}(F)} \arrow{d} && \Psi_{Y,Q}^{\mathrm{psh}}\big( \Pi_\infty(f)^\ast(F) \big) \arrow{d}\\
	j_*j\inv	f\inv\big(\Psi_{X,P}^{\mathrm{psh}}(F)\big) \arrow{rr}{ j_*j\inv\psi^{\mathrm{psh}}_{f}(F)} &&   j_*j\inv\Psi_{Y,Q}^{\mathrm{psh}}\big( \Pi_\infty(f)^\ast(F) \big) 
	\end{tikzcd} 
\end{equation}
From \cref{eg:functoriality}, the presheaf $\Psi_{Y,Q}^{\mathrm{psh}}\big( \Pi_\infty(f)^\ast(F) \big)$ is a hypersheaf on $Y$.
From \cite[Proposition 1.13]{HPT}, the adjunction  $j\inv \dashv  j_*$ induces an equivalence between $\Sh^{\hyp}(Y;\cE)$ and the full subcategory $\Sh^{\hyp}(\cB;\cE)$ of $\PSh(\cB;\cE)$ spanned by the functors $F :  \cB\op\to \cE$ such that $j_*(F)$ is a hypersheaf on $Y$.
Hence, the right vertical arrow of \cref{eq:diagpsijj} is an equivalence.
Since $j\inv\psi_{f}^{\mathrm{psh}}(F)$ is an equivalence by assumption, the bottom arrow is an equivalence as well.
Thus, $	j_*j\inv	f\inv\big(\Psi_{X,P}^{\mathrm{psh}}(F)\big) $ is a hypersheaf on $Y$.
Then, \cite[Remark 1.14]{HPT} implies that the left vertical arrow exhibits its target as the hypersheafification of its source.
Hence, 
$$
 j_*j\inv\psi_{f}^{\mathrm{psh}}(F) \simeq  \hyp(\psi_{f}^{\mathrm{psh}}(F)) \simeq\psi_{f}^{\hyp}(F)
$$
where the last equivalence follows from \cref{two_psif}.
\cref{criterion_psif} thus follows.
\end{proof}

\begin{cor} \label{cor:functoriality_monodromy}

	Let $f \colon Y \to X$ be a morphism between locally weakly contractible topological spaces.
	Then for every presentable $\infty$-category $\cE$, both transformations
	\[ \psi_f^{\hyp} \colon f^{\ast,\hyp} \circ \Psi_X^{\hyp} \to \Psi_Y^{\hyp} \circ \Pi_\infty(f)^\ast \quad \text{and} \quad \phi_f^{\hyp} \colon \Phi_Y^{\hyp} \circ f^{\ast, \hyp} \to \Pi_\infty(f)^\ast \circ \Phi_X^{\hyp} \]
	are equivalences once restricted to locally hyperconstant hypersheaves.
\end{cor}

\begin{proof}
	Combining \cref{rem:phi_equiv_iff_psi_equiv} and \cref{cor:monodromy_revisited} we see that it is enough to show that $\psi_f$ is an equivalence.
	Using \cref{lem:localization}, we can reduce ourselves to the case where $X$ is weakly contractible, so that the conclusion follows from \cref{cor:Phi_identification} and \cite[Proposition 2.5]{HPT}.
\end{proof}

\subsection{Restriction to strata for $\Psi$}

\begin{prop}\label{prop:Psi_restriction_stratum}
	Let $(X,P)$ be a conically stratified space.
	Let $a \in P$ and let $i_a \colon X_a \to X$ be the inclusion of the associated stratum.
	Then the natural transformation
	\[ \psi_a^{\hyp} \colon i_a^{\ast,\hyp} \circ \Psi_{X,P}^{\hyp} \to \Psi_{X_a}^{\hyp} \circ \Pi_\infty(i_a)^\ast \]
	is an equivalence.
\end{prop}

\begin{proof}
	Recall from \cref{Excellent_at_S_stratified space} the notion of excellency at strata.
	Given an open subset $U$ of $X_a$, write $\cB^{\exc}(U)$ for the poset of open neighborhoods $V$ of $U$ in $X$ such that $(V,P)$ is final at the $a$-stratum, that is for which the functor
	\[ \Pi_\infty(U) \simeq \Pi_\infty(U,\{a\}) \to \Pi_\infty(V,P) \]
	is final.
	Say that $U$ is \emph{nice} if $\cB^{\exc}(U)$ is a fundamental system of open neighbourhoods of $U$ in $X$ and write $\cB$ for the poset of nice open subsets of $X_a$.
	\Cref{prop:conical_implies_local_excellency} guarantees that $\cB$ is a basis of $X_a$.
	Fix $F \in \Fun(\Pi_\infty(X,P),\cE)$.
	From \cref{criterion_psif}, we are left to show that for every
	$U\in \cB$, the morphism
	\[ \psi^{\mathrm{psh}}_{i_a}(F)(U) \colon i_a\inv\big(\Psi_{X,P}^{\mathrm{psh}}(F)\big)(U) \to \Psi_{X_a,\{a\}}^{\mathrm{psh}}\big( \Pi_\infty(i_a)^\ast(F) \big)(U) \]
	is an equivalence.
	On the one hand
	\[ i_a\inv( \Psi_{X,P}^{\mathrm{psh}}(F))(U) \simeq \colim_{V\in \cB^{\exc}(U)} \Psi_{X,P}^{\mathrm{psh}}(F)(V) \simeq \colim_{V\in \cB^{\exc}(U)} \lim_{\Pi_\infty(V,P)} F |_{\Pi_\infty(V,P)} \ . \]
	Since for every $V\in \cB^{\exc}(U)$, the functor $\Pi_\infty(U)\to \Pi_\infty(V,P)$ is final, we deduce 
	\[ i_a\inv( \Psi_{X,P}^{\mathrm{psh}}(F))(U) \simeq \colim_{V\in \cB^{\exc}(U)} \lim_{\Pi_\infty(U)} F |_{\Pi_\infty(U)} \simeq  \lim_{\Pi_\infty(U)}  F |_{\Pi_\infty(U)}\ . \]
	On the other hand,
	\[ \Psi_{X_a}^{\hyp} \big(\Pi_\infty(i_a)^\ast(F)\big)(U) \simeq \lim_{\Pi_\infty(U)} F|_{\Pi_\infty(U)} \ . \]
	This concludes the proof of \cref{prop:Psi_restriction_stratum}.
\end{proof}

\begin{cor}\label{cor:Psi_produces_constructible}
	Let $(X,P)$ be a conically stratified space with locally weakly contractible strata.
	Then the functor $\Psi_{X,P}^{\hyp} \colon \Fun(\Pi_\infty(X,P), \cE) \to \HSh(X;\cE)$ factors through the full subcategory $\Cons_P^{\hyp}(X;\cE)$.
\end{cor}

\begin{proof}
	This is an immediate consequence of \cref{prop:Psi_restriction_stratum} and \cref{cor:Psi_produces_locally_constant}.
\end{proof}

\section{The exodromy equivalence}

\subsection{A criterion for $\phi_f$ to be an equivalence}

Let $f \colon (Y, Q) \to (X,P)$ be a morphism of conically stratified spaces and let $\cE$ be a presentable $\infty$-category.
As usual $\cE$ will be fixed throughout this section, and we therefore suppress it from the notations.
Our first goal is to provide a general method to establish when the natural transformation
\[ \phi_f^{\hyp} \colon \Phi_{Y,Q}^{\hyp} \circ f^{\ast,\hyp} \to \Pi_\infty(f)^\ast \circ \Phi_{X,P}^{\hyp} \]
is an equivalence.

\begin{lem}\label{relations_phif}
Let $f \colon (Y, Q) \to (X,P)$ be a morphism of conically stratified spaces.
Then the following statements hold; 
	\begin{enumerate}\itemsep=0.2cm
\item the transformations $\phi_f^{\mathrm{psh}}$ and 
$\phi_f^{\hyp}\circ \hyp$ are canonically equivalent.
\item for every diagram $F : I \to \PSh(X;\cE)$ with $I$ a small $\infty$-category, the canonical morphism 
$$
\phi_f^{\mathrm{psh}}(\colim_{i\in I} j(F_i))\to  \colim_{i\in I} \phi_f(F_i)
$$ 
is an equivalence, where $j : \PSh(X;\cE)\to \Fun(\PSh(X;\cE)\op;\cE)$ is the Yoneda embedding.
	\end{enumerate}
\end{lem}
\begin{proof}
For the first point, apply \cref{eg:functoriality_bis} and \cref{2-functoriality_phi} twice.
For the second point, we have 
$$
\phi_f^{\mathrm{psh}}(\colim_{i\in I} j(F_i))\simeq \colim_{i\in I}  \phi_f^{\mathrm{psh}}(j(F_i)) \simeq \colim_{i\in I} \phi_f(F_i)
$$ 
where the first equivalence follows from the fact that $\phi_f^{\mathrm{psh}}$ is built out of left adjoint functors and where the second equivalence follows from the observation that $\Phi_{X,P}^{\mathrm{psh}}$ identifies with the left Kan extension of $\Phi_{X,P}$ along the Yoneda embedding.
\end{proof}
\begin{cor} \label{lem:phi_equivalence_reduction}
	If $\phi_f$ is an equivalence, the same goes for $\phi_f^{\mathrm{psh}}$ and $\phi_f^{\hyp}$.
\end{cor}
\begin{proof}
The first statement comes from \cref{relations_phif}-$(2)$ and the fact that every presheaf is a small colimit of representable objects.
The second statements then follows from the observation from \cref{relations_phif}-$(1)$.
\end{proof}

The advantage of $\phi_f$ is that it can be explicitly computed in terms of open subsets of $Y$, rather than having to deal with all hypersheaves on $Y$.
To fully exploit this, let us specialize the main construction of \cref{subsec:functoriality} to the current setting.

\begin{notation}\label{notation:key_correspondence}
	For a conically stratified space $(X,P)$, let
	\[ \pi_X \colon \mathrm E_X \to \mathrm{Open}(X)\op \]
	be the cartesian fibration classifying the exit paths $\infty$-functor $\Pi_\infty \colon \mathrm{Open}(X) \to \Cat_\infty$.
	We let
	\[ \lambda_X \colon \mathrm E_X \to \Pi_\infty(X,P) \]
	be the canonical localization functor.
	Given a morphism $f \colon (Y,Q) \to (X,P)$ of conically stratified spaces, we consider the following diagram:
	\begin{equation}\label{eq:functoriality_topological_setting}
		\begin{tikzcd}[column sep = small]
			{} & \Pi_\infty(Y,Q) \arrow{rr}{\Pi_\infty(f)} & & \Pi_\infty(X,P) \\
			\mathrm E_Y \arrow{ur}{\lambda_Y} \arrow{d}{\pi_Y} & & \mathrm E_f \arrow{ll}[swap]{q} \arrow{r}{p} \arrow{d}{\pi_f} \arrow{ul}[swap]{\lambda_f} & \mathrm E_X \arrow{dl}{\pi_X} \arrow{u}{\lambda_X} \\
			\mathrm{Open}(Y)\op & & \mathrm{Open}(X)\op \arrow{ll}[swap]{f\inv}
		\end{tikzcd}
	\end{equation}
\end{notation}

With these notations, \cref{eg:functoriality} shows that $\Psi_{X,P}^{\mathrm{psh}} \simeq \pi_{X,\ast} \circ \lambda_X^\ast$.
In turn, we obtain:

\begin{cor} \label{cor:phi_equivalence_reduction}
	\hfill
	\begin{enumerate}\itemsep=0.2cm
		\item There is a canonical identification $\Phi_{X,P}^{\mathrm{psh}} \simeq \lambda_{X,!} \circ \pi_X^\ast$.
		
		\item The Beck-Chevalley transformation
		\[ (f\inv)^\ast \circ \pi_{Y,\ast} \to \pi_{f,\ast} \circ q^\ast \]
		is an equivalence.
		
		\item If the Beck-Chevalley transformation
		\[ \lambda_{f,!} \circ p^\ast \to \Pi_\infty(f)^\ast \circ \lambda_{X,!} \]
		is an equivalence, then so is $\phi_f^{\hyp}$.
	\end{enumerate}
\end{cor}

\begin{proof}
	Point (1) follows from the identification $\Psi_{X,P}^{\mathrm{psh}} \simeq \pi_{X,\ast} \circ \lambda_X^\ast$ passing to left adjoints.
	Point (2) is just a reformulation of \cref{lem:functoriality} in this specific situation.
	As for point (3), observe that by construction $\phi_f$ is the composition
	\[ \lambda_{Y,!} \circ \pi_Y^\ast \circ (f\inv)_! \simeq \lambda_{f,!} \circ \pi_{f}^\ast \simeq \lambda_{f,!} \circ p^\ast \circ \pi_X^\ast \to \Pi_\infty(f)^\ast \circ \lambda_{X,!} \circ \pi_X^\ast \ . \]
	Therefore, if the Beck-Chevalley transformation of the statement is an equivalence, the same goes for $\phi_f$.
	Thus, the conclusion follows from \cref{lem:phi_equivalence_reduction}.
\end{proof}

\subsection{Restriction to a closed union of strata for $\Phi$}

The following lemma is a straightforward consequence of Quillen's theorem A and its proof is left to the reader:

\begin{lem} \label{lem:verifying_left_adjointability}
	Let
	\[ \begin{tikzcd}
		\cC_0 \arrow{d}{g} \arrow{r}{i} & \cC \arrow{d}{f} \\
		\cD_0 \arrow{r}{j} & \cD
	\end{tikzcd} \]
	be a pullback square in $\Cat_\infty$.
	Assume that:
	\begin{enumerate}\itemsep=0.2cm
		\item the functors $i$ and $j$ are fully faithful;
		\item for $d \in \cD$, $d_0 \in \cD_0$, if $\Map_{\cD}(d,j(d_0)) \ne \emptyset$, then $d$ belongs to the essential image of $j$.
	\end{enumerate}
	Then, for any $\cE$ presentable $\infty$-category the Beck-Chevalley transformation $ g_! \circ i^* \to j^* \circ f_!$ is an equivalence.
\end{lem}

\personal{
\begin{proof}
	For every $d_0 \in \cD_0$, it is enough to show that  the morphism $\phi \colon \cC_0 \times_{\cD_0} (\cD_0)_{/d_0} \to \cC \times_{\cD} \cD_{/j(d_0)}$ is an equivalence, and hence cofinal. 
	Because $i$ and $j$ are fully faithful, the same goes for $\phi$.
	We only need to check that it is essentially surjective.
	So let $(c, \alpha) \in \cC \times_\cD \cD_{/j(d_0)}$.
	Here $c \in \cC$ and $\alpha \colon f(c) \to j(d_0)$.
	Condition (3) implies that $f(c) \in \cD_0$ and therefore condition (1) implies that $c \in \cC_0$.
	Therefore $(c, \alpha)$ actually defines an element in $\cC_0 \times_{\cD_0} (\cD_0)_{/d_0}$.
	The proof is complete.
\end{proof}
}
In the following corollaries, recall that the notations of \cref{notation_defin} are in use.
\begin{cor} \label{cor:Beck-Chevalley_beta_closed_strata}
	Let $(X,P)$ be a conically stratified space and let $S \subset P$ be a closed subset.
	The upper right square of the diagram \eqref{eq:functoriality_topological_setting}
	\[ \begin{tikzcd}
		\mathrm E_{i_S} \arrow{r}{p} \arrow{d}{\lambda_{i_S}} & \mathrm E_X \arrow{d}{\lambda_X} \\
		\Pi_\infty(X_S,S) \arrow{r}{\Pi_\infty(i_S)} & \Pi_\infty(X,P)
	\end{tikzcd} \]
	induces the Beck-Chevalley transformation
	\[ \beta \colon \lambda_{i_S !} \circ p^* \to \Pi_\infty(i_S)^* \circ \lambda_{X!} , \]
	which is an equivalence.
\end{cor}

\begin{proof}
	We check that the conditions of  \cref{lem:verifying_left_adjointability} are satisfied.
	First of all we observe that for every open subset $U$ of $X$, one has $U \cap X_S \simeq U_S$, where we consider $U$ equipped with the induced stratification $P|_U$.
	From \cref{fully_faith_XS}, the functor
	\[ \Pi_\infty(i_S|_{U_S}) \colon \Pi_\infty(U_S, S|_U) \to \Pi_\infty(U, P|_U) \]
	is fully faithful.
	It follows that both $\Pi_\infty(i_S)$ and $p$ are fully faithful functors.
	That the above square is a pullback is then obvious.
	Finally, the condition (3) from \cref{lem:verifying_left_adjointability} holds since $S\subset P$ is closed.
\end{proof}

Combining Corollaries \ref{cor:phi_equivalence_reduction}-(3) and \ref{cor:Beck-Chevalley_beta_closed_strata} we immediately obtain the following:

\begin{cor}\label{cor:Phi_restriction_closed_strata}
	Let $(X,P)$ be a conically stratified space and let $S \subset P$ be a closed subset.
	Then the natural transformation
	\[ \phi_S^{\hyp} \colon \Phi_{X_S,S}^{\hyp} \circ i_S^{\ast,\hyp} \to \Pi_\infty(i_S)^\ast \circ \Phi_{X,P}^{\hyp} \]
	is an equivalence.
\end{cor}

\subsection{Restriction to an open union of strata for $\Phi$}

Dealing with the inclusion of an open union of strata is more complicated.
We start collecting some general $\infty$-categorical facts:

\begin{defin}
	We say that a functor $f \colon \cC \to \cD$ is weakly cofiltered if for every object $d \in \cD$ the $\infty$-category $\cC_{d/} \coloneqq \cC \times_{\cD} \cD_{d/}$ is cofiltered.
\end{defin}

\begin{rem}
	If $f \colon \cC \to \cD$ is cofiltered in the sense of \cref{def:cofiltered} it is also weakly cofiltered: indeed, \cref{lem:filtered_functor} shows that for every $d \in \cD$ the induced map $\cC_{d/} \to \cD_{d/}$ is cofiltered and, on the other hand, $\cD_{d/}$ is obviously cofiltered since it has an initial object.
	On the other hand, Quillen's theorem A implies that if $f \colon \cC \to \cD$ is weakly cofiltered, then it is also colimit-final.
\end{rem}

\begin{lem}\label{lem:fiberwise_cofinality_criterion}
	Let $\cX$ be an $\infty$-category. 
	Let $\mathsf A, \mathsf B \colon \cX \to \Cat_\infty$ be functors and let
	\[ \pi_{\mathsf A} \colon \cA \to \cX\op \ , \qquad \pi_{\mathsf B} \colon \cB \to \cX\op \]
	be the associated cartesian fibrations.
	Let $f \colon \mathsf A \to \mathsf B$ be a natural transformation.
	If for every $x \in \cX$ the functor $f_x \colon \mathsf A(x) \to \mathsf B(x)$ is weakly cofiltered, then the induced functor $\cA \to \cB$ is colimit-final.
\end{lem}

\begin{proof}
	Let $b \in \cB$ be an object.
	In virtue of Quillen's theorem A, we have to prove that the $\infty$-category
	\[ \cA_{b/} \coloneqq \cA \times_{\cB} \cB_{b/} \]
	is weakly contractible.
	We claim that it is cofiltered.
	Set $x \coloneqq \pi_{\mathsf B}(b)$, so that we can review $b$ as an element in $\cB_x \simeq \mathsf B(x)$.
	Write $\mathsf A(x)_{b/} \coloneqq \mathsf A(x) \times_{\mathsf B(x)} \mathsf B(x)_{b/}$ and let
	\[ j \colon \mathsf A(x)_{b/} \to \cA_{b/} \]
	be the natural functor.
	Let now $F \colon I \to \cA_{b/}$ be a finite diagram.
	Composing $F$ with the canonical projection $\cA_{b/} \to \cX_{x/}$ and using the fact that both $\pi_{\mathsf A}$ and $\pi_{\mathsf B}$ are cartesian fibrations, we deduce the existence of a functor
	\[ F' \colon I \to \mathsf A(x)_{b/} \coloneqq \mathsf A(x) \times_{\mathsf B(x)} \mathsf B(x)_{b/} \]
	together with a natural transformation $\gamma \colon j \circ F' \to F$.
	By assumption the functor $f_x \colon \mathsf A(x) \to \mathsf B(x)$ is weakly cofiltered, and therefore $\mathsf A(x)_{b/}$ is cofiltered.
	It follows that $F'$ admits an extension $\widetilde{F}' \colon I^\lhd \to \mathsf A(x)_{b/}$.
	At this point, the natural transformation $\gamma$ allows us to prolong $\widetilde{F}'$ into an extension $\widetilde{F} \colon I^\lhd \to \cA_{b/}$ of $F$.
	The conclusion follows.
\end{proof}

Let $(X,P)$ be a conically stratified space.
Fix an object $x \in \Pi_\infty(X,P)$.
Define $\Pi_\infty(X,P)_{/x}^\simeq$ as the full subcategory of $\Pi_\infty(X,P)_{/x}$ spanned by its final objects.
For every functor $\cC \to \Pi_\infty(X,P)$ we set
\[ \cC_{/x} \coloneqq \cC \times_{\Pi_\infty(X,P)} \Pi_\infty(X,P)_{/x} \quad \text{and} \quad \cC_{/x}^\simeq \coloneqq \cC \times_{\Pi_\infty(X,P)} \Pi_\infty(X,P)_{/x}^\simeq \ . \]
The following lemma is the key geometrical argument of this section (and of the whole paper):

\begin{lem}\label{technical_lemma_for_exit}
	Let $U\subset X$ be an open subset.
	Then, the functor
		\[ \Pi_\infty(U,P)_{/x}^\simeq \to \Pi_\infty(U,P)_{/x} \]
	is weakly cofiltered.
	In particular, it is cofinal.
\end{lem}

Before giving the proof, let us describe the topological idea.
Unraveling the definitions, we see that we have to show the following statement: given a morphism $\gamma \colon y \to x$ where $y \in U$, the space of factorizations of $\gamma$ as
\[ \begin{tikzcd}[column sep=small]
	y \arrow{r}{\gamma'} & y' \arrow{r}{\gamma''} & x \ ,
\end{tikzcd} \]
where $\gamma''$ lies entirely in the stratum of $x$ is cofiltered (and in particular weakly contractible).
It is easy to see that it is non-empty: representing $\gamma$ as a continuous morphism $\gamma \colon [0,1] \to X$, the exit path condition ensures that $\gamma((0,1])$ entirely belongs to the stratum of $x = \gamma(1)$.
Since $\gamma$ is continuous, $\gamma\inv(U)$ is an open subset of $[0,1]$ and $0 \in \gamma\inv(U)$, so that $\gamma\inv(U) \cap (0,1] \ne \emptyset$.
Any point $t$ in this intersection allows to write $\gamma$ as the concatenation of two paths, $\gamma |_{[0,t]}$ and $\gamma |_{[t,1]}$, which gives the factorization we were seeking for.
The proof below is an implementation of this idea, made more sophisticated by the fact that we not only check non-emptyness but the property of being cofiltered.

\begin{proof}[Proof of \cref{technical_lemma_for_exit}]
	Let $\mathbf \gamma  \in \Pi_\infty(U,P)_{/x}$ and write $\Pi_\infty(U,P)_{\mathbf \gamma /\!\!/x}\coloneqq (\Pi_\infty(U,P)_{/x})_{\mathbf \gamma/}$.
	We show that the $\infty$-category
	\[ \Pi_\infty(U,P)_{\mathbf \gamma/\!\!/x}^\simeq \coloneqq \Pi_\infty(U,P)_{/x}^\simeq \times_{\Pi_\infty(U,P)_{/x}} \Pi_\infty(U,P)_{\mathbf \gamma/\!\!/x} \]
    is cofiltered.
	Let therefore $I$ be a finite category and consider the lifting problem
	\[ \begin{tikzcd}
		I \arrow{r}{g} \arrow[hook]{d}{j} & \Pi_\infty(U,P)_{\mathbf \gamma /\!\!/ x}^\simeq \\
		I^\lhd \arrow[dashed]{ur}[swap]{\overline{g}}
	\end{tikzcd} \]
	Since $\Pi_\infty(U,P)_{\mathbf \gamma/\!\!/x}^\simeq$ is fully faithful inside $\Pi_\infty(U,P)_{\mathbf \gamma /\!\!/x}$, we can rewrite the above lifting problem as
	\[ \begin{tikzcd}
		I \arrow{r}{g} \arrow[hook]{d}{j} & \Pi_\infty(U,P)_{\mathbf \gamma/\!\!/x} \\
		I^\lhd \arrow[dashed]{ur}[swap]{\overline{g}}
	\end{tikzcd} \]
	where $g$ sends an object $i\in I$ to a morphism $\mathbf \gamma \to \mathbf \delta_i$ in $\Pi_\infty(U,P)_{/x}$ with $\mathbf \delta_i \in \Pi_\infty(U,P)_{/x}^\simeq $ and where $\overline{g}$ is required to send the vertex $v_1$ of $I^\lhd$ to a morphism $\mathbf \gamma \to \mathbf \delta$ in $\Pi_\infty(U,P)_{/x}$ with $\mathbf \delta \in \Pi_\infty(U,P)_{/x}^\simeq$.
	The following picture summarizes the situation in the case where $I$ consists of two objects:
	\begin{center}
		\begin{tikzpicture}
			\node (A) at (-2,3) {} ;
			\node (B) at (2,0) {} ;
			\node (C1) at (0.2,2.1) {} ;
			\node (C2) at (4,3.8) {} ;
			\node (C3) at (6,2.6) {} ;
			\node (C4) at (6,0.2) {} ;
			\node (C5) at (3.5,0.8) {} ;
			\draw (-2,3) .. controls (C1) .. node[pos=0.65,label=240:{\tiny $y$}] (y) {} (2,0) ;
			\draw[rounded corners] (-2,3) .. controls (C2) and (5.8,3.1) .. (6,2.6) .. controls (C4) and (C5) .. (2,0) ;
			
			\node (D1) at (-3.7,2.8) {} ;
			\node (D2) at (-4,2) {} ;
			\node (D3) at (-2,1.5) {} ;
			\node (D4) at (0,-1) {} ;
			\draw[rounded corners] (-2,3) .. controls (D1) .. (-3.8,2.3) .. controls (D3) .. (0,-1) .. controls (0.6,-0.6) .. (2,0) ;
			
			\filldraw[black] (y) circle (1pt) ;
			\draw[dashed] (y) circle (43pt) ;
			
			\node[label={[label distance=0.2pt]60:{\tiny{$x$} }}] (x) at (4,2.5) {} ;
			\filldraw[black] (x) circle (1pt) ;
			
			\draw[->] (y) .. controls (1,2) and (3,2) .. node[pos=0.2,label={[label distance=0.1pt]60:{\tiny{$\overline{y}$}}}] (ybar) {} node[pos=0.6,sloped,above] {\tiny$\gamma$} (x) ;
			
			\node[label={[label distance=0.2pt]180:{\tiny{$y_1$}}}] (y1) at (0.7,2.6) {} ;
			\filldraw[black] (y1) circle (1pt) ;
			
			\node[label={[label distance=0.2pt]275:{\tiny{$y_2$}}}] (y2) at (1.5,1.2) {} ;
			\filldraw[black] (y2) circle (1pt) ;
			
			\filldraw[black] (ybar) circle (1pt) ;
			
			\draw[->] (y) -- (y1) ;
			
			\draw[->] (y) -- (y2) ;
			
			\draw[->] (y1) .. controls (2.1,3) .. node[midway,sloped,above] {\tiny{$\delta_1$}} (x) ;
			
			\draw[->] (y2) .. controls (3,1.6) .. node[midway,sloped,below] {\tiny{$\delta_2$}} (x) ;
			
			\draw[->,dotted] (ybar) -- (y1) ;
			
			\draw[->,dotted] (ybar) -- (y2) ;
			
			\node at (-1.2,1) {\small{$U$}} ;
		\end{tikzpicture}
	\end{center}
	The stratum through $y$ is represented by the line through $y$, and everything on its right represents the stratum containing $x$.
	The goal is to establish the existence of the dotted arrows as well as of the homotopies making the above diagram commutative.
	If $t \in (0,1)$ is a point for which $\overline{y} = \gamma(t)$, then taking $\delta \coloneqq \gamma |_{[t,1]}$ solves the above lifting problem (notice that in this case, the morphism $\gamma \to \delta$ would correspond to $\gamma |_{[0,t]}$).
	
	\medskip
	
	Formally, this is equivalent to a lifting problem
	\[ \begin{tikzcd}
		I^\lhd \arrow{r}{g'} \arrow[hook]{d}{j^\lhd} & \Pi_\infty(U,P)_{/x} \\
		(I^\lhd)^\lhd \arrow[dashed]{ur}[swap]{\overline{g}'} & \phantom{(\mathrm E_X)_{/x}^\simeq} 
	\end{tikzcd} \]
	where $g'$ sends the vertex $v_0$ of the exterior cone $(I^\lhd)^\lhd$ to $\mathbf \gamma$ and the other objects  in $\Pi_\infty(U,P)_{/x}^\simeq$, and where $\overline{g}'$ is required to send  $v_1$  in $\Pi_\infty(U,P)_{/x}^\simeq$. 
	Unraveling the definitions and denoting by $v_2$ the final object of $((I^\lhd)^\lhd)^\rhd$, we can further reduce this lifting problem to the following one
	\[ \begin{tikzcd}
		(I^\lhd)^\rhd \arrow[hook]{d}{j^\lhd} \arrow{r}{h} & \Pi_\infty(X,P) \\
		((I^\lhd)^\lhd)^\rhd \arrow[dashed]{ur}[swap]{\overline{h}} & \phantom{\Pi_\infty(X,P)} \ ,
	\end{tikzcd} \]
	where we ask for $\overline{h}$ to represent a continuous morphism
	\[ \overline{\mathrm h} \colon | ((I^\lhd)^\lhd)^\rhd | \to X , \]
	with the following properties:
	\begin{enumerate}\itemsep=0.2cm
		\item the morphism $\mathrm h$ takes the segment $[v_0,v_2]$ to the exit path $\gamma$ underlying $\mathbf \gamma$. 
		
		\item the morphism $\mathrm h$ takes the complement of $v_0$ to the same stratum of $x$;
		
		\item the morphism $\mathrm h$ takes the double cone $(I^\lhd)^\lhd$ inside the open $U$.
		
	\end{enumerate}
	Observe that there is a canonical equivalence
	\[ ((I^\lhd)^\lhd)^\rhd \simeq (\Delta^1 \star I) \star \Delta^0 \simeq \Delta^1 \star (I \star \Delta^0) \simeq ((I^\rhd)^\lhd)^\lhd \ , \]
	where $\star$ denotes the join operation.
	Set $J \coloneqq I^\rhd$.
	By definition of $\Pi_\infty(X,P)$, we can represent $h$ by an explicit continuous morphism
	\[ \mathrm h \colon | J^\lhd | \to X , \]
	with the property that it takes the segment $[v_0,v_2]$   to $\gamma$, that the complement of $v_0$ is sent in the same stratum as $x$, and that the cone over $I$ with vertex $v_0$ is sent in $U$.
	Since $\mathrm h$ is continuous, the preimage of $U$ is an open of  $| J^\lhd |$ containing the cone over $I$ with vertex $v_0$.
	In particular, $\mathrm h\inv(U) \cap (v_0,v_2] \ne \emptyset$.
	Let $w \in \mathrm h\inv(U) \cap (v_0,v_2]$ be any point.
	Since $\gamma$ is an exit path, $\mathrm h([w,v_2])$ is contained in the same stratum of $x$.
	Then, the inclusion of the exterior cone
	\[ i \coloneqq |j^\lhd| \colon | J^\lhd | \hookrightarrow | (J^\lhd)^\lhd | \]
	admits a retraction $r_w$ satisfying the following conditions:
	\begin{enumerate}[(i)]\itemsep=0.2cm
		\item $r_w(v_1) = w$;
		
		\item $r_w$ sends $|(I^\lhd)^\lhd|$ inside $\mathrm h^{-1}(U)$;
		
		\item the vertex $v_0$ is the unique point $s \in | (J^\lhd)^\lhd |$ such that $r_w(s) = v_0$.
	\end{enumerate}
	Define
	\[ \overline{\mathrm h} \coloneqq \mathrm h \circ r . \]
	It is then straightforward to verify that $\overline{\mathrm h}$ satisfies the conditions (1),(2) and (3) listed above.
\end{proof}

\begin{cor}\label{cofinality_trivial_fibration}
Let $(X,P)$ be a  conically stratified space.
Let $S\subset P$ be any subset.
Put $U\coloneqq X_S$.
Let $V\subset X$ be any open subset and consider the commutative square 
$$
\begin{tikzcd}
\Pi_\infty(U\cap V,P)  \arrow{r} \arrow{d}  &  \Pi_\infty(U,P)  \arrow{d}     \\
\Pi_\infty(V,P)  \arrow{r}   &  \Pi_\infty(X,P)
	\end{tikzcd} 
$$
Let $x\in U$. 
Then, the functor 
$$
\Pi_\infty(U\cap V,P)  \times_{ \Pi_\infty(U,P) }  \Pi_\infty(U,P)_{/x}\to  \Pi_\infty(V,P)  \times_{ \Pi_\infty(X,P) }  \Pi_\infty(X,P)_{/x}
$$
is cofinal.
\end{cor}
\begin{proof}
If $x$ lies in the stratum $0\in S$, we have 
$$
\Pi_\infty(X,P)_{/x}^{\simeq}=\Exit(X_0,\ast)_{/x}=\Pi_\infty(U,P)_{/x}^{\simeq}
$$
Thus, 
$$
\Pi_\infty(V,P)  \times_{ \Pi_\infty(X,P) }  \Pi_\infty(X,P)_{/x}^{\simeq}  =\Pi_\infty(U\cap V,P)  \times_{ \Pi_\infty(U,P) }  \Pi_\infty(U,P)_{/x}^{\simeq} 
$$
Hence, there is a commutative diagram 
$$
\begin{tikzcd}
\Pi_\infty(U\cap V,P)  \times_{ \Pi_\infty(U,P) }  \Pi_\infty(U,P)_{/x}^{\simeq} \arrow{r} \arrow[equal]{d}  &  \Pi_\infty(U\cap V,P)  \times_{ \Pi_\infty(U,P) }  \Pi_\infty(U,P)_{/x} \arrow{d}     \\
\Pi_\infty(V,P)  \times_{ \Pi_\infty(X,P) }  \Pi_\infty(X,P)_{/x}^{\simeq}   \arrow{r}   &  \Pi_\infty(V,P)  \times_{ \Pi_\infty(X,P) } \Pi_\infty(X,P)_{/x}
	\end{tikzcd} 
$$
From \cref{technical_lemma_for_exit}, the top and bottom arrows of the above diagram are cofinal.
Hence, the right vertical arrow is cofinal in virtue of \cite[4.1.1.3]{HTT}.
\end{proof}

\begin{cor}\label{technical_lemma}
	The functor
	$(\mathrm E_X)_{/x}^\simeq \to (\mathrm E_X)_{/x}$
	is cofinal.
\end{cor}
\begin{proof}
	We apply \cref{lem:fiberwise_cofinality_criterion} taking $\cX = \mathrm{Open}(X)$, and $\mathsf A$ and $\mathsf B$ to be the functors given by
	\[ \mathsf A(U) \coloneqq \Pi_\infty(U,P)_{/x}^\simeq \quad \text{and} \quad \mathsf B(U) \coloneqq \Pi_\infty(U,P)_{/x} \ . \]
	By construction, the cartesian fibration classified by $\mathsf B$ is given by the natural projection $(\mathrm E_X)_{/x} \to \mathrm{Open}(X)\op$, and similarly the cartesian fibration classified by $\mathsf A$ is given by $(\mathrm E_X)_{/x}^\simeq \to \mathrm{Open}(X)\op$.
	Thus, \cref{lem:fiberwise_cofinality_criterion} shows that it is enough to show that for every open $U$ of $X$, the functor
	\[ \Pi_\infty(U,P)_{/x}^\simeq \to \Pi_\infty(U,P)_{/x} \]
	is weakly cofiltered.
	This is proved in \cref{technical_lemma_for_exit}.
	\end{proof}

\begin{cor} \label{cor:Beck-Chevalley_beta_open}
	Let $(X,P)$ be a conically stratified space.
	Let $S \subset P$ be an open subset.
	The upper right square of the diagram \eqref{eq:functoriality_topological_setting}
	\begin{equation} \label{eq:Beck-Chevalley_beta_open_strata}
		\begin{tikzcd}
			\mathrm E_{i_S} \arrow{r}{p} \arrow{d}{\lambda_{i_S}} & \mathrm E_X \arrow{d}{\lambda_X} \\
			\Pi_\infty(X_S,S) \arrow{r}{\Pi_\infty(i_S)} & \Pi_\infty(X,P)
		\end{tikzcd}
	\end{equation}
	induces the Beck-Chevalley transformation
	\[ \beta \colon \lambda_{i_S !} \circ p^* \to \Pi_\infty(i_S)^* \circ \lambda_{X!} \]
	which is an equivalence.
\end{cor}

\begin{proof}
	As in the proof of \cref{cor:Beck-Chevalley_beta_closed_strata}, $\Pi_\infty(i_S)$ and $p$ are fully faithful and the above square is a pullback.
	Let now $x \in \Pi_\infty(X_S, S)$ be an object.
	Committing a slight abuse of notation, we still denote by $x$ its image in $\Pi_\infty(X,P)$ via the functor $\Pi_\infty(i_S)$.
	Put
	\[ (\mathrm E_{i_S})_{/x} \coloneqq \mathrm E_{i_S} \times_{\Pi_\infty(X_S, S)} \Pi_\infty(X_S, S)_{/x} \]
	and consider the commutative diagram
	\[ \begin{tikzcd}
		(\mathrm E_{i_S})_{/x} \arrow{r} \arrow{d} & \mathrm (\mathrm E_X)_{/x}  \arrow{d} \\
		\Pi_\infty(X_S, S)_{/x} \arrow{r} & \Pi_\infty(X,P)_{/x}.
	\end{tikzcd} \]
	Unraveling the definitions, we reduce ourselves to check that  $(\mathrm E_{i_S})_{/x} \to (\mathrm E_X)_{/x}$ is cofinal.
	Observe that since $p$ is fully faithful, the same goes for this map.
	By definition, $(\mathrm E_{i_S})_{/x}$ is the full subcategory of triples $(U,y,\gamma)$ in $(\mathrm E_X)_{/x}$ such that  $y$ lies in $X_S$.
	Using the same notations of \cref{technical_lemma}, we find the following commutative diagram of $\infty$-categories over $\Pi_\infty(X, P)_{/x}$ :
	\[ \begin{tikzcd}
		(\mathrm E_{i_S})_{/x}^\simeq \arrow{r} \arrow{d} & \mathrm (\mathrm E_X)_{/x}^\simeq \arrow{d} \\
		(\mathrm E_{i_S})_{/x} \arrow{r} & (\mathrm E_X)_{/x} .
	\end{tikzcd} \]
	Since the vertical functors and the bottom horizontal one are fully faithful, the same goes for the top horizontal one.
	Furthermore, if $(U,y,\gamma) \in (\mathrm E_X)_{/x}^\simeq$ then $\gamma \colon y \to x$ is an equivalence in $\Pi_\infty(X,P)$.
	Since $x$ belongs to $X_S$, the same goes for $y$.
	In other words, $(U,y,\gamma)$ belongs to $(\mathrm E_{i_S})_{/x}^\simeq$.
	This shows that the top horizontal arrow is also essentially surjective, hence an equivalence.
	As a consequence, to prove that the bottom horizontal functor is cofinal it is enough to prove that both vertical functors are cofinal.
	
	We now show that it is enough to prove that the right vertical functor is cofinal.
	By Quillen's Theorem A and the above discussion, it is enough to show that for $\mathbf y = (U,y,\gamma) \in (\mathrm E_{i_S})_{/x}$, the fully faithful functor $((\mathrm E_{i_S})_{/x})_{\mathbf y/} \to ((\mathrm E_X)_{/x})_{\mathbf y/}$ is an equivalence.
	Let $\mathbf z \coloneqq (V,z,\delta) \in (\mathrm E_X)_{/x}$ and let  $\mathbf y \to \mathbf z$ be a morphism in $(\mathrm E_X)_{/x}$. 
	Observe that this morphism produces an exit path $y\to z$. 
	Since  $S$ is open, this implies that $z$ lies in $X_S$. Thus, we have $\mathbf z \in (\mathrm E_{i_S})_{/x}$.
	Hence,  the  fully faithful functor $((\mathrm E_{i_S})_{/x})_{\mathbf y/} \to ((\mathrm E_X)_{/x})_{\mathbf y/}$  is an equivalence and the sought-after reduction is complete.
	\personal{Indeed, if
		\[ \begin{tikzcd}[ampersand replacement = \&]
			\cA \arrow{r} \arrow{d} \& \cB \arrow{d} \\
			\cC \arrow{r} \& \cD
		\end{tikzcd} \]
		is a pullback of $\infty$-categories and the bottom horizontal functor is fully faithful and for every $c \in \cC$ the canonical map $\cC_{c/} \to \cD_{c/}$ is an equivalence, then cofinality of $\cB \to \cD$ immediately implies cofinality for $\cA \to \cC$, thanks to Quillen's theorem A.}
	We then conclude the proof of \cref{cor:Beck-Chevalley_beta_open} using \cref{technical_lemma}.
\end{proof}

Combining Corollaries~\ref{cor:phi_equivalence_reduction}-(3) and \ref{cor:Beck-Chevalley_beta_open} we immediately obtain the following:

\begin{cor}\label{cor:Phi_restriction_open_strata}
	Let $(X,P)$ be a conically stratified space.
	Let $S \subset P$ be an open  subset.
	Then the natural transformation
	\[ \phi_S^{\hyp} \colon \Phi_{X_S,S}^{\hyp} \circ i_S^{\ast,\hyp} \to \Pi_\infty(i_S)^\ast \circ \Phi_{X,P}^{\hyp} \]
	is an equivalence.
\end{cor}	

\begin{cor}\label{cor:Phi_restriction_single_stratum}
	Let $(X,P)$ be a conically stratified space.
	For every  $a \in P$ , the natural transformation
	\[ \phi_a^{\hyp} \colon \Phi_{X_a}^{\hyp} \circ i_a^{\ast,\hyp} \to \Pi_\infty(i_a)^\ast \circ \Phi_{X,P}^{\hyp} \]
	is an equivalence.
\end{cor}

\begin{proof}
	It follows by first  applying first \cref{cor:Phi_restriction_open_strata}  to $(X,P)$ and the subset $P_{\geqslant a} \subseteq P$ and subsequently \cref{cor:Phi_restriction_closed_strata} to $(X_{\geqslant a}, P_{\geqslant a})$ and the subset $\{a\} \subseteq P_{\geqslant a}$.
\end{proof}

\subsection{The equivalence}

We are now ready for the main theorem:

\begin{thm} \label{thm:exodromy}
	Let $(X,P)$ be a conically stratified space and let $\cE$ be a presentable $\infty$-category.
	Assume that:
	\begin{enumerate}\itemsep=0.2cm
		\item for every $a \in P$, the associated stratum $X_a$ is locally weakly contractible;
		
		\item the hyper-restrictions
		\[ \big\{ i_a^{\ast,\hyp} \colon \HSh(X;\cE) \to \HSh(X_a;\cE) \big\}_{a \in P} \]
		are jointly conservative.
	\end{enumerate}
	Then the exodromy adjunction
	\[ \Phi_{X,P}^{\hyp} \colon \Cons_P^{\hyp}(X;\cE) \leftrightarrows \Fun(\Pi_\infty(X,P),\cE) \colon \Psi_{X,P}^{\hyp} \]
	is an equivalence.
\end{thm}

\begin{proof}
	We have to prove that for every functor $F \colon \Pi_\infty(X,P) \to \cE$ and every hyperconstructible hypersheaf $G \in \Cons_P^{\hyp}(X;\cE)$, the unit and counit
	\[ \eta_G \colon G \to \Psi_{X,P}^{\hyp}(\Phi_{X,P}^{\hyp}(G)) \quad \text{and} \quad \varepsilon_F \colon \Phi_{X,P}^{\hyp}(\Psi_{X,P}^{\hyp}(F)) \to F \]
	are equivalences.
	By assumption, it is enough to check that for every $a \in P$, the (hyper-)restrictions $i_a^{\ast,\hyp}(\eta_G)$ and $\Pi_\infty(i_a)^{\ast}(\varepsilon_F)$ are equivalences.
	Consider the following square:
	\[ \begin{tikzcd}
		\HSh(X;\cE) \arrow{d}{i_a^{\ast,\hyp}} \arrow{r}{\Phi_{X,P}^{\hyp}} & \Fun\big(\Pi_\infty(X,P), \cE\big) \arrow{d}{\Pi_\infty(i_a)^\ast} \\
		\HSh(X_a;\cE) \arrow{r}{\Phi_{X_a}^{\hyp}} & \Fun\big(\Pi_\infty(X_a),\cE\big) \ .
	\end{tikzcd} \]
	\Cref{cor:Phi_restriction_single_stratum} shows that the natural transformation $\phi_a^{\hyp}$ makes this into a commutative diagram.
	Moreover, combining \cref{rem:phi_equiv_iff_psi_equiv_bis} and \cref{prop:Psi_restriction_stratum} we deduce that it is horizontally right adjointable.
	Therefore we obtain canonical identifications
	\[ i_a^{\ast,\hyp}(\eta_G) \simeq \eta_{i_a^{\ast,\hyp}(G)}  \qquad \text{and} \qquad \Pi_\infty(i_a)^{\ast}(\varepsilon_F) \simeq \varepsilon_{\Pi_\infty(i_a)^\ast(F)} \ , \]
	where the latter are the unit and the counit of the adjunction $\Phi_{X_a}^{\hyp} \dashv \Psi_{X_a}^{\hyp}$.
	Observe now that since $G$ is hyperconstructible, $i_a^{\ast,\hyp}(G)$ is locally constant.
	Thus the conclusion follows directly from \cref{cor:monodromy_revisited}.
\end{proof}

\begin{rem}\label{joint_conservativeness}
The joint conservativity of the hyper-restrictions to strata is satisfied in the following two cases of interest:
\begin{enumerate}\itemsep=0.2cm
\item the category $\cE$ is compactly generated. 
See \cite[Corollary 5.16]{HPT}.

\item the poset $P$ is noetherian and the category $\cE$ is stable and presentable. 
See \cite[Corollary 5.21]{HPT}.
\end{enumerate}
\end{rem}

Let us emphasize some immediate consequences of the above theorem.

\begin{cor}\label{cor:constructible_presentable}
	Let $(X,P)$ be a conically stratified space and let $\cE$ be a presentable $\infty$-category.
	If the assumptions of \cref{thm:exodromy} are satisfied, then:
	\begin{enumerate}\itemsep=0.2cm
		\item the $\infty$-category $\Cons_P^{\hyp}(X;\cE)$ is presentable;
		\item if $\cE$ is stable, then so is $\Cons_P^{\hyp}(X;\cE)$.
		
		\item if $\cE$ is an $\infty$-topos, then so is $\Cons_P^{\hyp}(X;\cE)$.
	\end{enumerate}
\end{cor}

The following corollary generalizes at the same time \cite[Lemma A.9.14]{Lurie_Higher_algebra} and \cite[Corollary 3.2]{HPT}

\begin{cor} \label{cor:constructible_limits_colimits}
	Let $(X,P)$ be a conically stratified space and let $\cE$ be a presentable $\infty$-category.
	If the assumptions of \cref{thm:exodromy} are satisfied, then the full subcategory $\Cons_P^{\hyp}(X;\cE) \hookrightarrow \HSh(X;\cE)$ is closed under small limits and small colimits.
\end{cor}

\begin{proof}
	For colimits it is enough to observe that if $a \in P$, then $i_a^{\ast,\hyp} \colon \HSh(X;\cE) \to \HSh(X_a;\cE)$ commutes with colimits.
	Thus, we are reduced to show that $\LChyp(X;\cE)$ is closed under colimits in $\HSh(X;\cE)$. 
	This was shown in \cite[Corollary 3.2]{HPT}.
	Concerning limits, it is enough to observe that $\Psi_{X,P}^{\hyp} \colon \Fun(\Pi_\infty(X,P), \cE) \to \HSh(X;\cE)$ is fully faithful and right adjoint to $\Phi_{X,P}^{\hyp}$.
	Thus $\Psi_{X,P}^{\hyp}$ commutes with small limits.
	On the other hand \cref{thm:exodromy} shows that the essential image of $\Psi_{X,P}^{\hyp}$ is $\Cons_P^{\hyp}(X;\cE)$, whence the conclusion.
\end{proof}

\begin{rem} \label{rem:constructibilization}
Combining Corollaries~\ref{cor:constructible_presentable} and \ref{cor:constructible_limits_colimits} it follows that if the assumptions of \cref{thm:exodromy} are satisfied then the inclusion $\iota : \Cons_P(X;\cE) \hookrightarrow \HSh(X;\cE)$ admits both a left adjoint $\mathrm L_P^{\hyp}$ and a right adjoint $\mathrm R_P^{\hyp}$.
\end{rem}		

\begin{rem} \label{rem:phi_does_not_see_the_constructibilization}
		The functor $\Psi_{X,P}^{\hyp} :\Fun(\Pi_\infty(X,P),\cE)\to \HSh(X;\cE)$ 
		factors as $\iota \circ\Psi_{X,P}^{\hyp}$.
		Hence, if the assumptions of \cref{thm:exodromy} are satisfied, it admits both $\Phi_{X,P}^{\hyp}$ 	 and $\Phi_{X,P}^{\hyp}|_{\Cons_P(X;\cE)}\circ \mathrm L_P^{\hyp}$ as left adjoints.
		Thus, for every $F\in  \HSh(X;\cE)$, the unit adjunction $F\to \iota\circ L_P^{\hyp}(F)$ induces an equivalence $\Phi_{X,P}^{\hyp}(F)\simeq \Phi_{X,P}^{\hyp}L_P^{\hyp}(F)$.
\end{rem}

\begin{cor} \label{cor:criterion_constructibility}
	Let $(X,P)$ be a conically stratified space and let $\cE$ be a presentable $\infty$-category.
	Under the assumptions of \cref{thm:exodromy}, for a hypersheaf $F \in \HSh(X;\cE)$ the following statements are equivalent:
	\begin{enumerate}\itemsep=0.2cm
		\item $F$ is hyperconstructible on $(X,P)$;
		
		\item for every open inclusion $U \subset V$ for which the induced map $\Pi_\infty(U,P) \to \Pi_\infty(V,P)$ is a categorical equivalence, the restriction map $F(V) \to F(U)$ is an equivalence;
		
		\item for every conical chart $Z \times C(Y)$ of $(X,P)$, $F$ satisfies the following two conditions:
		\begin{enumerate}[(i)]\itemsep=0.2cm
			\item for every open subset $W \subset Z$ and every $0 < \varepsilon < \varepsilon' \leqslant 1$, the restriction map
			\[ F( W \times C_{\varepsilon'}(Y)) \to F( W \times C_\varepsilon(Y) ) \]
			is an equivalence;
			
			\item for every inclusion $U \subseteq V$ of weakly contractible open subsets of $Z$, the restriction map
			\[ F(V \times C(Y)) \to F(U \times C(Y)) \]
			is an equivalence.
		\end{enumerate}
	\end{enumerate}
\end{cor}

\begin{proof}
	The implication (1) $\Rightarrow$ (2) is a direct consequence of \cref{thm:exodromy}.
	For the implication (2) $\Rightarrow$ (3), it is enough to observe that, on the one hand, the inclusion $W \times C_{\varepsilon}(Y) \hookrightarrow W \times C_\varepsilon(Y)$ is a stratified homotopy equivalence (and hence it satisfies the assumption of point (2)); and on the other hand, for an inclusion $U \subseteq V$ of weakly contractible open subsets of $Z$, the induced map $\Pi_\infty(U \times C(Y)) \to \Pi_\infty(V \times C(Y))$ is also a categorical equivalence, thanks to \cref{lem:categorical_homotopy_invariance}.
	We are therefore left to prove that (3) implies (1).
	The question is local, so we can replace $X$ by a conical chart of the form $Z \times C(Y)$.
	Letting $i \colon Z \to Z \times C(Y)$ be the natural inclusion, we only have to check that $i^{\ast,\hyp}(F)$ is locally hyperconstant.
	Consider first the presheaf-theoretic pullback $i\inv(F)$.
	For every open subset $W$ of $Z$, one has
	\[ i\inv(F)(W) \simeq \colim_{0 < \varepsilon \leqslant 1} F( W \times C_\varepsilon(Y) ) \ . \]
	Our assumption (i) guarantees that the transition maps in the above colimit diagram are equivalences, and hence that the canonical restriction map
	\[ F(W \times C(Y)) \to i\inv(F)(W) \]
	is an equivalence.
	In particular, it follows that $i\inv(F)$ is a hypersheaf, and therefore that it coincides with $i^{\ast,\hyp}(F)$.
	
	At this point, in order to prove that $i\inv(F) \simeq i^{\ast,\hyp}(F)$ is locally hyperconstant it is enough, in virtue of \cite[Proposition 3.1]{HPT}, to prove that for every inclusion $U \subset V$ of weakly contractible open subsets of $Z$, the restriction map
	\[ i\inv(F)(V) \to i\inv(F)(U) \]
	is an equivalence.
	As we showed above, this amounts to check that the restriction map
	\[ F(V \times C(Y)) \to F(U \times C(Y)) \]
	is an equivalence.
	Since this holds by assumption (ii), the conclusion follows.
\end{proof}

For later use, let us remark that in the previous proof we established the following fact:

\begin{cor}\label{cor:hyperpullback_computation}
	Let $Z$ be a locally weakly contractible topological space and let $(Y,Q)$ be a stratified space so that $(X,P) \coloneqq (Z \times C(Y), Q^{\vartriangleleft})$ is conically stratified.
	Let $\cE$ be a presentable $\infty$-category and assume that the assumptions of \cref{thm:exodromy} are satisfied by $(X,P)$ and $\cE$.
	Let $F \in \ConsPhyp(X;\cE)$  and let $U$ be an open subset of $Z$.
	Then the canonical map
	\[ F(U \times C(Y)) \to i^{\ast,\hyp}(F)(U) \]
	is an equivalence, where $i$ denotes the canonical inclusion $i \colon Z \hookrightarrow X$.
\end{cor}

\section{Consequences}

Having proven \cref{thm:exodromy}, we now explore some of its consequences.
In this section, we mainly focus on improved functoriality results for the exodromy equivalence and some structural results for the $\infty$-category of hyperconstructible hypersheaves.

\subsection{Structural results for hyperconstructible hypersheaves}

Let $(X,P)$ be a conically stratified with locally weakly contractible strata.
Applying \cref{cor:constructible_presentable}, we see that $\Cons_P^{\hyp}(X)$ is a presentable $\infty$-category.
In particular, for every $\cE \in \PrL$, the tensor product $\Cons_P^{\hyp}(X) \otimes \cE$ is well defined.
Our first task is to compare it with $\Cons_P^{\hyp}(X;\cE)$.

\medskip

Since $\Cons_P^{\hyp}(X)$ is presentable, it follows from \cref{cor:constructible_limits_colimits} that the inclusion $\Cons_P^{\hyp}(X) \hookrightarrow \HSh(X)$ is a morphism in $\PrL$ having a further left adjoint.
Therefore, tensoring with $\cE$ yields a fully faithful functor
\begin{equation}\label{eq:tensor_decomposition}
	\Cons_P^{\hyp}(X) \otimes \cE \longhookrightarrow \HSh(X) \otimes \cE \ .
\end{equation}
Let $a \in P$ and let $i_a \colon X_a \hookrightarrow X$ be the inclusion of the corresponding stratum.
Functoriality of the tensor product of presentable $\infty$-categories immediately implies the commutativity of the following diagram:
\[ \begin{tikzcd}
	\Cons_P^{\hyp}(X) \otimes \cE \arrow[hook]{r} \arrow{d}{i_a^{\ast,\hyp} \otimes \id_\cE} & \HSh(X) \otimes \cE \arrow{d}{i_a^{\ast,\hyp} \otimes \id_\cE} \\
	\LChyp(X_a) \otimes \cE \arrow[hook]{r} & \HSh(X_a) \otimes \cE \ .
\end{tikzcd} \]
Recall from \cite[Observation 3.11]{HPT} that the equivalence $\HSh(X_a) \otimes \cE \simeq \HSh(X_a;\cE)$ restricts to an equivalence $\LChyp(X_a) \otimes \cE \simeq \LChyp(X_a;\cE)$.
Thus it follows that under the equivalence $\HSh(X) \otimes \cE \simeq \HSh(X;\cE)$, the functor \eqref{eq:tensor_decomposition} factors through $\Cons_P^{\hyp}(X;\cE)$, yielding the following commutative diagram:
\begin{equation}\label{eq:tensor_decomposition_II}
	\begin{tikzcd}
		\Cons_P^{\hyp}(X) \otimes \cE \arrow{r} \arrow{d} & \HSh(X) \otimes \cE \arrow{d} \\
		\Cons_P^{\hyp}(X;\cE) \arrow{r} & \HSh(X;\cE) \ .
	\end{tikzcd}
\end{equation}

\begin{cor} \label{cor:tensor_decomposition}
	Let $(X,P)$ be a conically stratified space and let $\cE$ be a presentable $\infty$-category.
	If the assumptions of \cref{thm:exodromy} are satisfied, then the diagram
	\[ \begin{tikzcd}[column sep = large]
			\Cons_P^{\hyp}(X) \otimes \cE \arrow{r}{\Phi_{X,P}^{\hyp} \otimes \id_\cE} \arrow{d}{\boxtimes} & \Fun\big(\Pi_\infty(X,P), \cS\big) \otimes \cE \arrow{d}{\boxtimes} \\
			\Cons_P^{\hyp}(X;\cE) \arrow{r}{\Phi_{X,P}^{\hyp,\cE}} & \Fun\big(\Pi_\infty(X,P),\cE\big)
	\end{tikzcd} \]
	is canonically commutative.
	In particular, the left vertical functor is an equivalence.
\end{cor}

\begin{proof}
	The first half is a direct consequence of \cref{lem:Phi_with_coefficients} and the commutativity of the diagram \eqref{eq:tensor_decomposition_II}.
	The second half follows from the first one and the 2-out-of-3 property of the equivalences.
\end{proof}

\begin{cor}\label{cor:k_compact_generation}
	Let $(X,P)$ be a conically stratified space and let $\cE$ be a $\kappa$-presentable $\infty$-category.
	If the assumptions of \cref{thm:exodromy} are satisfied, then the $\infty$-category $\ConsPhyp(X;\cE)$ is $\kappa$-presentable.
\end{cor}

\begin{proof}
	Indeed, $\Fun(\Pi_\infty(X,P), \cS)$ is always compactly generated, and therefore \cite[Lemma 5.3.2.11]{Lurie_Higher_algebra} implies $\Fun(\Pi_\infty(X,P),\cS) \otimes \cE$ is $\kappa$-presentable as well.
\end{proof}

\begin{cor}[Categorical K\"unneth formula] \label{cor:categorical_Kunneth}
	Let $(X,P)$ and $(Y,Q)$ be conically stratified spaces with locally weakly contractible strata.
	Then there is a canonical equivalence
	\[ \Cons_{P \times Q}^{\hyp}(X\times Y) \simeq \Cons_P^{\hyp}(X) \otimes \Cons_Q^{\hyp}(Y) \ . \]
\end{cor}

\begin{proof}
	This follows from the fact that $\Fun(-,\cS)$ takes products of $\infty$-categories to tensor products of presentable $\infty$-categories, and the fact that $\Pi_\infty$ commutes with finite products of stratified spaces.
\end{proof}

\subsection{Exodromic morphisms}

Let $f \colon (Y,Q) \to (X,P)$ be a morphism of conically stratified spaces.
Let $\cE$ be a presentable $\infty$-category.
Although the natural transformations
\[ \psi_f^{\hyp} \colon f^{\ast,\hyp} \circ \Psi_{X,P}^{\hyp} \to \Psi_{Y,Q}^{\hyp} \circ \Pi_\infty(f)^\ast \quad \text{and} \quad \phi_f^{\hyp} \colon \Phi_{Y,Q}^{\hyp} \circ f^{\ast,\hyp} \to \Pi_\infty(f)^\ast \circ \Phi_{X,P}^{\hyp} \]
are always defined, they are not always equivalences.

\begin{defin}
	We say that the morphism $f \colon (Y,Q) \to (X,P)$ is:
	\begin{itemize}\itemsep=0.2cm
		\item \emph{right exodromic} if the natural transformation $\psi_f^{\hyp}$ is an equivalence;
		
		\item \emph{left exodromic} if the natural transformation $\phi_f^{\hyp}$ is an equivalence when evaluated on hyperconstructible hypersheaves on $(X,P)$;
		
		\item \emph{strongly left exodromic} if the natural transformation $\phi_f^{\hyp}$ is an equivalence;
		
		\item (\emph{strongly}) \emph{exodromic} if it is right and (strongly) left exodromic.
	\end{itemize}
\end{defin}

With this terminology and with the help of \cref{thm:exodromy}, \cref{rem:phi_equiv_iff_psi_equiv_bis} implies:

\begin{lem}\label{lem:left_right_exodromic}
	Let $f \colon (Y,Q) \to (X,P)$ be a morphism of conically stratified spaces with  locally weakly contractible strata.
	Let $\cE$ be a presentable $\infty$-category satisfying the assumptions of \cref{thm:exodromy}.
	Then $f$ is left exodromic if and only if it is right exodromic.
\end{lem}

\begin{prop}\label{cor:induced_stratification}
Let $f \colon (Y,Q) \to (X,P)$ be a morphism of conically stratified spaces with locally weakly contractible strata.
Let $\cE$ be a presentable $\infty$-category satisfying the assumptions of \cref{thm:exodromy}.
Then $f$ is left exodromic.
\end{prop}

\begin{proof}
	From \cref{lem:left_right_exodromic}, it is enough to prove that $\psi_f^{\hyp}$ is an equivalence.
	For this, it is enough to show that for every $a \in Q$, the hyper-restriction $i_a^{\ast,\hyp}(\psi_f^{\hyp})$ is an equivalence.
	From \cref{prop:Psi_restriction_stratum} and \cref{2-functoriality_psi}, we find for every $F : \Pi_{\infty}(X,P)\to \cE$ a canonical identification
		\[ i_a^{\ast,\hyp}(\psi_f^{\hyp}(F)) \simeq \psi_{f\circ i_a}^{\hyp}(F) \ . \]
	Let $b\in P$ be the image of $a$  and let $g : Y_a \to X_b$ be the morphism morphism induced by $f$.
	Then, $f\circ i_a = i_b\circ g$ so that a second application of \cref{prop:Psi_restriction_stratum} and \cref{2-functoriality_psi} gives a canonical identification 
\[ i_a^{\ast,\hyp}(\psi_f^{\hyp}(F)) \simeq \psi_{g}^{\hyp}(\Pi_{\infty}(i_b)^{\ast}(F)) \ . \]
Thus, the conclusion follows from \cref{cor:functoriality_monodromy}.
\end{proof}




\begin{rem}\label{rem:open_inclusion_strata_strongly_exodromic}
	In particular, if $f \colon (Y,Q) \to (X,P)$ is an open immersion, \cref{cor:induced_stratification} implies that $\phi_f^{\hyp}$ is an equivalence on hyperconstructible hypersheaves.
	When $f$ is the inclusion of an open union of strata of $(X,P)$, Corollaries~\ref{cor:Psi_open_restriction} and \ref{cor:Phi_restriction_open_strata} show that more is true: indeed, in this case $f$ is strongly exodromic.
\end{rem}

\begin{cor}\label{cor:Phi_formula}
	Let $(X,P)$ be a conically stratified space with locally weakly contractible strata.
	Let $\cE$ be a presentable $\infty$-category satisfying the assumptions of \cref{thm:exodromy}.
	Let $x \in X$ be a point and let $F \in \Cons_P^{\hyp}(X;\cE)$.
	We have:
	\[ \Phi_{X,P}^{\hyp}(F)(x) \simeq \colim_{x \in U} F(U) \ , \]
	where the colimit ranges over the open subsets of $X$ containing $x$.
\end{cor}

\begin{proof}
	Review $x$ as a morphism \smash{$x \colon * \to X$}.
	The induced stratification on $\ast$ is the trivial one, and in particular it is conical.
	Thus, \cref{cor:induced_stratification} implies that $x$ is exodromic, and therefore that the natural transformation $\phi_x^{\hyp}$ is an equivalence.
	The conclusion follows since the colimit in the statement is canonically identified with $x^{\ast,\hyp}(F)$.
\end{proof}
The next lemma asserts that the failure for the formula from \cref{cor:Phi_formula}  to hold exactly measures the defect for a hypersheaf to be hyperconstructible.
\begin{cor}\label{formula_for_phi_characterize_cons}
Let $(X,P)$ be a conically stratified space with locally weakly contractible strata. 
Let $\cE$ be a presentable $\infty$-category satisfying the assumptions of \cref{thm:exodromy}.
Let $F\in \HSh(X;\cE)$.
Then $F$ is hyperconstructible on $(X,P)$ if and only if for every $x\in X$, the canonical morphism
\[ \phi_x^{\hyp}(F) : \Phi_{X,P}^{\hyp}(F)(x) \to \colim_{x \in U} F(U) \  \]
is an equivalence.
\end{cor}
\begin{proof}
	The direct implication follows from \cref{cor:Phi_formula}.
	Assume now that $\phi_x^{\hyp}(F)$  is an equivalence for every $x\in X$.
	Let $L_P^{\hyp}$ be the left adjoint to the inclusion $\Cons_P(X;\cE) \hookrightarrow \HSh(X;\cE)$ as obtained in \cref{rem:constructibilization}.
	To prove \cref{formula_for_phi_characterize_cons}, it is enough to prove that the unit transformation $F\to L_P^{\hyp}(F)$ is an equivalence.
	Since both source and target are hypersheaves it is enough to show that for every $x\in X$, the induced morphism $F_x\to L_P^{\hyp}(F)_x$ is an equivalence.
	By assumption, $F_x$ identifies with $\Phi_{X,P}^{\hyp}(F)(x)$.
	On the other hand, \cref{cor:Phi_formula} combined with \cref{rem:phi_does_not_see_the_constructibilization} gives the following chain of equivalences
	$$
	L_P^{\hyp}(F)_x\simeq \Phi_{X,P}^{\hyp}(L_P^{\hyp}(F))(x)\simeq \Phi_{X,P}^{\hyp}(F)(x)
	$$
 \cref{formula_for_phi_characterize_cons} is thus proved.
\end{proof}

%
%
%

\begin{cor}\label{cor:localization}
	Let $(X,P)$ be a conically stratified space and let $(X,Q)$ be a conical refinement of $P$.
	Assume that the strata of $(X,P)$ and $(X,Q)$ are locally weakly contractible.
	Then the natural map
	\[ f \colon \Pi_\infty(X,Q) \to \Pi_\infty(X,P) \]
	is a localization.
\end{cor}

\begin{proof}
	We first reduce to the case where the stratification on $X$ is trivial.
	Let $W$ be the collection of $f$-local equivalences in $\Pi_\infty(X,Q)$.
	Since $f$ is essentially surjective, \cite[7.1.7, 7.1.11]{Cisinski_Higher_Category} shows that it is enough to prove that the functor
	\[ f^* \colon \Fun(\Pi_\infty(X,P), \cS) \to \Fun(\Pi_\infty(X,Q), \cS) \]
	is fully faithful and its essential image consists of $W$-local objects.
	Since by definition $f = \Pi_\infty(\id_X)$, \cref{cor:induced_stratification} shows that $\psi_f^{\hyp}$ makes the diagram
	\[ \begin{tikzcd}
		\Fun(\Pi_\infty(X,P),\cS) \arrow{r}{f^*} \arrow{d}{\Psi_{X,P}^{\hyp}} & \Fun(\Pi_\infty(X,Q),\cS) \arrow{d}{\Psi_{X,Q}^{\hyp}} \\
		\Cons_P^{\hyp}(X) \arrow{r}{\id_X^{\ast,\hyp}} & \Cons_Q^{\hyp}(X)
	\end{tikzcd} \]
	commutative, while \cref{thm:exodromy} shows that the vertical functors are equivalences.
	Considering $\Cons_Q^{\hyp}(X)$ and $\Cons_P^{\hyp}(X)$ as full subcategories of $\HSh(X)$, we see that $f^{\ast,\hyp}$ acts as the identity and it is therefore fully faithful.
	Thus, the commutativity of the above diagram implies that the functor $f^*$ is fully faithful as well.
	Since $f$ takes (by definition) arrows in $W$ to equivalences, we see that $f^*$ factors through the full subcategory of $W$-local objects.
	Thus, we are left to check the essential surjectivity.
	Let $F \colon \Pi_\infty(X,Q) \to \cS$ be a functor and assume that it is $W$-local.
	In virtue of \cref{thm:exodromy}, it is enough to prove that $\Psi_{X,Q}^{\hyp}(F)$ belongs to $\Cons_P^{\hyp}(X)$.
	Denote by $\varphi \colon Q \to P$ be the given morphism at the level of posets.
	Let $a \in P$ be a point and let $Q_a \coloneqq \varphi\inv(a)$.
	Denote by $j_a \colon X_{Q_a} \to X$ the inclusion of the corresponding union of strata.
	Applying \cref{prop:Psi_restriction_stratum} we see that the transformation
	\[ \psi_{j_a}^{\hyp} \colon j_a^{\ast,\hyp}\big( \Psi_{X,Q}^{\hyp}(F) \big) \to \Psi_{X_{Q_a}, Q_a}^{\hyp}\big( \Pi_\infty(j_a)^\ast(F) \big) \]
	is an equivalence.
	Thus, we can replace $X$ by $X_a$.
	Equivalently, we can assume from the very beginning that $P$ is the trivial stratification (and therefore that $X$ is locally weakly contractible).
	
	\medskip
	
	Equipping $*$ with the trivial stratification $S$, we can see the canonical map $\Gamma_X \colon X \to *$ as a map of stratified spaces $(X,Q) \to (*,S)$.
	Then \cref{cor:induced_stratification} implies that  $\psi_{\Gamma_X}^{\hyp}$ yields a canonical identification
	\[ \Gamma_X^{\ast,\hyp} \simeq \Psi_{X,Q}^{\hyp} \circ \Gamma_{\Pi_\infty(X,Q)}^\ast \ . \]
	Since $X$ is locally weakly contractible, \cite[\S~2.2]{HPT} shows that $\Gamma_X^{\ast,\hyp}$ has a left adjoint $\Gamma_{X,\sharp}^{\hyp}$.
	The above identification implies
	\[ \Gamma_{X,\sharp}^{\hyp} \simeq \Gamma_{\Pi_\infty(X,Q),!} \circ \Phi_{X,Q}^{\hyp} \ . \]
	Using \cite[Corollary 3.5]{HPT} we find a canonical identification $\Pi_\infty(X) \simeq \Gamma_{X,\sharp}^{\hyp}( \Gamma_X^{\ast,\hyp}(*) )$.
	Thus, the above identification yields:
	\[ \Pi_\infty(X) \simeq \Gamma_{X,\sharp}^{\hyp} \Gamma_X^{\ast,\hyp}(*) \simeq \colim_{\Pi_\infty(X,Q)} \Gamma_{\Pi_\infty(X,Q)}^\ast(*) \simeq \mathrm{Env}(\Pi_\infty(X,Q)) \ , \]
	whence the conclusion.
\end{proof}

We conclude this subsection with \cref{observation:extra_functorialities} showing that exodromy implies enhanced functorialities at the level of hyperconstructible hypersheaves:

\begin{observation}\label{observation:extra_functorialities}
	Let $f \colon (X,P) \to (Y,Q)$ be a  morphism between conically stratified spaces with locally weakly contractible strata, and let $\cE$ be a presentable $\infty$-category satisfying the assumption of \cref{thm:exodromy}.
	The functor $\Pi_\infty(f)^\ast$ has both a left adjoint given by the left Kan extension $\Pi_\infty(f)_!$ and a right adjoint given by the right Kan extension $\Pi_\infty(f)_\ast$.
	Thus, it follows that the functor
	\[ f^{\ast,\hyp} \colon \Cons_P^{\hyp}(X;\cE) \to \Cons_Q^{\hyp}(Y;\cE) \]
	also has a left and a right adjoint.
	We set
	\[ f_{\sharp}^{\Cons} \coloneqq \Psi_{X,P}^{\hyp} \circ \Pi_\infty(f)_! \circ \Phi_{X,P}^{\hyp} \quad \text{and} \quad f_\ast^{\Cons} \coloneqq \Psi_{X,P}^{\hyp} \circ \Pi_\infty(f)_\ast \circ \Phi_{X,P}^{\hyp} \ . \]
	Since the functor $f^{\ast, \hyp} \colon \HSh(X;\cE) \to \HSh(Y;\cE)$ has a right adjoint given by $f_\ast$, and the inclusions of $\Cons_P^{\hyp}(X;\cE)$ and $\Cons_Q^{\hyp}(Y;\cE)$ have right adjoints $\mathrm R_P^{\hyp}$ and $\mathrm R_Q^{\hyp}$ (see \cref{rem:constructibilization}), it follows from formal reasons that
	\begin{equation}\label{eq:pushforward}
		f_\ast^{\Cons} \simeq \mathrm R_P^{\hyp} \circ f_\ast \ .
	\end{equation}
	When $f^{\ast,\hyp}$ has itself a left adjoint\footnote{This is e.g.\ the case if $f$ is the projection $S \times X \to S$ for locally weakly contractible topological spaces, see \cite[\S~2.2]{HPT}.} $f_\sharp^{\hyp}$, it follows similarly that
	\[ f_\sharp^{\Cons} \simeq \mathrm L_P^{\hyp} \circ f_\sharp^{\hyp} \ , \]
	where $\mathrm L_P^{\hyp}$ denotes the left adjoint to the inclusion of $\ConsPhyp(X;\cE)$ in $\HSh(X;\cE)$ (see \cref{rem:constructibilization} again).
\end{observation}

\subsection{Monoidal structures}

Assume now that our $\infty$-category of coefficients $\cE$ carries a presentably symmetric monoidal structure.
In other words, we fix $\cE \in \CAlg(\PrLotimes)$.

\begin{recollection}
	\hfill
	\begin{enumerate}\itemsep=0.2cm
		\item For every topological space $X$, the $\infty$-category $\HSh(X;\cE)$ inherits a canonical symmetric monoidal structure, and every morphism $f \colon X \to Y$ induces a symmetric monoidal functor
		\[ f^{\ast,\hyp} \colon \HSh(Y;\cE) \to \HSh(X;\cE) \ . \]
		In particular, it follows that $\LChyp(X;\cE)$ is stable under tensor product, and consequently that for every stratified space $(X,P)$, $\ConsPhyp(X,P)$ acquires a symmetric monoidal structure.
		
		\item Similarly, $\Fun\big(\Pi_\infty(X,P), \cE \big)$ inherits a symmetric monoidal structure, where the tensor product is computed objectwise.
	\end{enumerate}
\end{recollection}

\begin{prop}\label{prop:exodromy_strong_monoidal}
	Let $(X,P)$ be a conically stratified space and let $\cE^\otimes$ be a symmetric monoidal $\infty$-category.
	Then the functor
	\[ \Phi_{X,P}^{\hyp,\cE} \colon \HSh(X;\cE) \to \Fun\big(\Pi_\infty(X,P), \cE\big) \ , \]
	has a natural lax symmetric monoidal structure.
	Under the assumptions of \cref{thm:exodromy}, $\Phi_{X,P}^{\hyp,\cE}$ restricts to a symmetric monoidal functor on the full subcategory $\ConsPhyp(X;\cE)$.
\end{prop}

\begin{proof}
	Thanks to \cref{lem:Phi_with_coefficients} and \cref{cor:tensor_decomposition}, we immediately reduce to the case where $\cE = \cS$.
	In this case, the monoidal structure induced on both $\ConsPhyp(X)$ and $\Fun\big(\Pi_\infty(X,P),\cS\big)$ is the cartesian one.
	Then $\Phi_{X,P}^{\hyp}$ has a canonical lax monoidal structure given by \cite[Proposition 2.4.1.7]{Lurie_Higher_algebra}.
	Since $\Phi_{X,P}^{\hyp}$ is an equivalence, it commutes with products and hence Corollary 2.4.1.8 in \emph{loc.\ cit.} guarantees that $\Phi_{X,P}^{\hyp}$ is strong monoidal.
	The conclusion follows.
\end{proof}

As a consequence of our criterion for constructibility (see \cref{cor:criterion_constructibility}), we obtain:

\begin{cor}\label{cor:internal_hom_constructible}
	Let $(X,P)$ be a conically stratified space and let $\cF, \cG \in \ConsPhyp(X;\cE)$.
	If the assumptions of \cref{thm:exodromy} are satisfied, then the internal hom $\cHom_X(\cF, \cG)$, computed inside $\HSh(X;\cE)$, belongs to $\ConsPhyp(X;\cE)$.
\end{cor}

\begin{proof}
	We apply \cref{cor:criterion_constructibility}.
	It is enough to prove that for every inclusion $i \colon U \subset V$ for which the induced map $\Pi_\infty(i) \colon \Pi_\infty(U,P) \to \Pi_\infty(V,P)$ is an equivalence, the restriction map
	\[ \cHom_X(\cF, \cG)(V) \to \cHom_X(\cF,\cG)(U) \]
	is an equivalence as well.
	By definition, this map can be rewritten as
	\begin{equation}\label{eq:constructible_hom}
		\Hom_{\ConsPhyp(V;\cE)}( \cF|_V, \cG|_V ) \to \Hom_{\ConsPhyp(U;\cE)}(\cF|_U, \cG|_U) \ .
	\end{equation}
	\Cref{thm:exodromy} allows to write the target of this morphism
	\[ \Hom_{\Fun(\Pi_\infty(U,P),\cE)}(\Phi_{U,P}^{\hyp}(\cF|_U), \Phi_{U,P}^{\hyp}(\cG|_U)) \ , \]
	and similarly for the source.
	Moreover, since the natural transformation
	\[ \phi^{\hyp}_i \colon \Phi_{U,P}^{\hyp} \circ i^{\ast,\hyp} \to \Pi_\infty(i)^\ast \circ \Phi_{V,P}^{\hyp} \]
	is an equivalence thanks to \cref{cor:induced_stratification}, the morphism \eqref{eq:constructible_hom} becomes canonically identified with the morphism induced by $\Pi_\infty(i)^\ast$
	\[  \Hom_{V}(\Phi_{V,P}^{\hyp}(\cF|_V), \Phi_{V,P}^{\hyp}(\cG|_V)) \to \Hom_{U}(\Pi_\infty(i)^\ast\Phi_{V,P}^{\hyp}(\cF|_V), \Pi_\infty(i)^\ast\Phi_{V,P}^{\hyp}(\cG|_V)) \ , \]
	where we wrote $\Hom_V$ to denote the $\cE$-enriched hom object computed in $\Fun(\Pi_\infty(V,P),\cE)$, and similarly for $\Hom_U$.
	Since $\Pi_\infty(i)^\ast$ is an equivalence, so is the above morphism, whence the conclusion.
\end{proof}

\subsection{Compact generation}
 
     Using the  functoriality of the exodromy equivalence, we are now able to prove the following result, strengthening \cref{cor:k_compact_generation}.

\begin{thm} \label{thm:compact_generation}
	Let $(X,P)$ be a conically stratified space with locally weakly contractible strata.
	Let $\cE\in \PrLomega$ and let $(e_{\alpha})_{\alpha \in A}$ be a set of compact generators for $\cE$.
	Then $(j_{x,\sharp}(e_{\alpha}))_{x\in X,\alpha \in A}$ is a set of compact generators for $\ConsPhyp(X;\cE)$.
\end{thm}

\begin{proof}
      For $x\in X$, \cref{cor:induced_stratification} shows that the triangle
		\begin{equation}\label{exodromic_morphism_square}
		\begin{tikzcd}
			\ConsPhyp(X;\cE) \arrow[rr,"\sim"] \arrow[rd, "j^{\ast, \hyp}_x"']&  & \Fun(\Pi_\infty(X,P), \cE) \arrow[ld, "j^{\ast}_x "] \\
			& \cE & 
		\end{tikzcd} 
	\end{equation} 
		commutes, where $j_x : \{x\}\to \Pi_\infty(X,P)$ is the inclusion.
		Passing to left adjoints, \cref{observation:extra_functorialities} guarantees that the functor
	$j_x^{\ast,\hyp} \colon \ConsPhyp(X;\cE) \to \cE$
	admits a left adjoint 
	\[ 
	j_{x,\sharp}^{\Cons} \colon \ \cE \to \ConsPhyp(X;\cE) 
	\] 
corresponding under the exodromy equivalence to the left Kan extension
\[
    j_{x,!} : \cE \to  \Fun(\Pi_\infty(X,P), \cE)  \ .
\]
	Hence, we are left to show that $(j_{x,!}(e_{\alpha}))_{x\in X,\alpha \in A}$ is a set of compact generators for $\Fun(\Pi_\infty(X,P), \cE)$.
	Via the equivalence 
	\[
	\Fun(\Pi_\infty(X,P), \cE)\simeq \Fun(\Pi_\infty(X,P), \cS) \otimes \cE
	\]
the functor 
\[
j_{x,!}(e_{\alpha}) : \Pi_\infty(X,P)\to \cE
\] 
corresponds to $j(x)\otimes e_{\alpha}$, where $j : \Pi_\infty(X,P)\op \to \Fun(\Pi_\infty(X,P), \cS)$
is the Yoneda embedding.
      Since $(j(x))_{x\in X}$  is a set of compact generators for $\Fun(\Pi_\infty(X,P), \cS)$, the conclusion follows from \cite[5.3.2.11]{Lurie_Higher_algebra}.
\end{proof}


\begin{cor}
    Let $X$ be a connected locally weakly contractible topological space and let $x \in X$ be a point.    
    For every $\mathbb E_\infty$-algebra $A$, the $\infty$-category $\LChyp(X;\Mod_A)$ is canonically equivalent to the $\infty$-category of right modules over the endomorphism ring of $j_{x,\sharp}(A)$.
\end{cor}

\begin{proof}
   By \cite{Schwede_Shipley} (see also \cite[Theorem 7.1.2.1]{Lurie_Higher_algebra}), it is enough to show that  $j_{x,\sharp}(A)[n]$, $n\in \bZ$ is a set of compact generators for  $\LChyp(X;\Mod_A)$.
   Since $A[n]$, $n\in \bZ$ is a set of compact generators for $\Mod_A$, the conclusion follows from \cref{thm:compact_generation}.
\end{proof}

     It turns out that it is possible to compute quite explicitly such endomorphism ring (we warmly thank Julian Holstein and Alexandru Oancea for discussions on this point) :

\begin{thm}\label{thm:LC_modules_over_chains}
	Let $X$ be a connected and locally weakly contractible topological space.
	Fix a point $x \in X$ and consider the $x$-based loop space $\Omega_x(X)$.
	Then for every $\mathbb E_\infty$-algebra $A$, the endomorphism ring of $j_{x,\sharp}(A)$ canonically coincides with the chain algebra $\mathrm C_\ast(\Omega_x(X);A)$.
	In particular, there is an equivalence
	\[ \LChyp(X;\Mod_A) \simeq \Mod_{\mathrm C_\ast(\Omega_x(X);A)} \ . \]
\end{thm}

\begin{rem}
	This theorem recovers \cite[Theorem 26]{Holstein_Morita_cohomology_I} and \cite[Theorem 12]{Holstein_Morita_cohomology_II}.
	Besides it generalizes it in several directions, as $A$ is now allowed to be an $\mathbb E_\infty$-algebra, and $X$ is now only required to be locally weakly contractible rather than admitting a bounded locally finite good hypercover.
	In particular, the above theorem holds even without any paracompactness assumption on $X$.
\end{rem}

\begin{proof}[Proof of \cref{thm:LC_modules_over_chains}]
	The second half is a direct consequence of the first half and Lurie-Barr-Beck's theorem \cite[Theorem 4.7.3.5]{Lurie_Higher_algebra}.
	To compute the endomorphism ring of $j_{x,\sharp}(A)$, we first observe that there is the following natural equivalence:
	\[ \Hom_{\LChyp(X;\Mod_A)}\big(j_{x,\sharp}^{\hyp}(A), j_{x,\sharp}^{\hyp}(A)\big) \simeq j_x^{\ast,\hyp} j_{x,\sharp}^{\hyp}(A) \ . \]
	Observe now that the square
	\[ \begin{tikzcd}
		\Pi_\infty( \Omega_x( X ) ) \arrow{r}{\Gamma} \arrow{d}{\Gamma} & * \arrow{d}{j_x} \\
		* \arrow{r}{j_x} & \Pi_\infty(X) \ .
	\end{tikzcd} \]
	is a pullback in $\cS$, where we wrote $\Gamma$ instead of $\Gamma_{\Omega_x(X)}$ for brevity.
	Unraveling the definitions, this implies that $\Pi_\infty(\Omega_x(X))$ coincides with the comma category
	\[ \{x\} \times_{\Pi_\infty(X)} \Pi_\infty(X)_{j_x/} \ , \]
	so that combining \cref{cor:monodromy_revisited} with the formula for the left Kan extension, we deduce that the Beck-Chevalley transformation
	\[ \Gamma_\sharp^{\hyp} \circ \Gamma^{\ast,\hyp} \to j_x^{\ast,\hyp} \circ j_{x,\sharp}^{\hyp} \]
	is an equivalence.
	Thus, the endomorphism ring of $j_{x,\sharp}(A)$ is canonically identified with
	\[ \Gamma_\sharp^{\hyp} \Gamma^{\ast,\hyp}(A) \simeq \Gamma_\sharp^{\hyp} \Gamma^{\ast,\hyp}(*) \otimes A \ . \]
	By definition, $\Gamma_\sharp^{\hyp} \Gamma^{\ast,\hyp}(*) \in \cS$ is the shape of the $\infty$-topos $\HSh(\Omega_x(X))$.
	Applying \cite[Corollary 3.5]{HPT}, we see that this shape is simply identified with the homotopy type $\Pi_\infty(\Omega_x(X))$.
	Thus, the endomorphism ring of $j_{x,\sharp}(A)$ is identified with
	\[ \Gamma_\sharp^{\hyp} \Gamma^{\ast,\hyp}(A) \simeq \Pi_\infty(\Omega_x(X)) \otimes A \simeq \mathrm C_\ast(\Omega_x(X);A) \ , \]
	where the last equivalence holds by definition of singular chains.
\end{proof}

\begin{rem}
	It is possible to obtain a similar description in the stratified case, in line with \cite[\S6]{Holstein_Categorical_Koszul}.
	It should be possible to obtain an explicit description of the endomorphism ring of the single compact generator provided by \cref{thm:compact_generation} akin to the one of \cref{thm:LC_modules_over_chains}, at least in the setting of conically smooth stratified spaces.
	Indeed, any such description should see the chain algebras on single strata, but at the same time, it should also have a contribution from the links of the stratification.
	We will come back to this subject in a later work.
\end{rem}

\subsection{Exodromy and stalkwise compactness}

Let $(X,P)$ be a conically stratified space with locally weakly contractible strata and let $\cE$ be a presentable $\infty$-category.
The construction of  $\Phi_{X,P}^{\hyp}$ and $\Psi_{X,P}^{\hyp}$ relies a priori on the existence of limits and colimits of diagrams that are typically not finite.
However, the ``regularity'' of conical charts paired with the homotopy-invariance property of hyperconstructible hypersheaves \cite[Theorem 0.4]{HPT} actually implies that all infinite colimits that are involved in the construction of the exodromy adjunction can be ignored.
To give a proper formulation of this idea, let us first introduce the following notation:

\begin{notation}
	Let $(X,P)$ be a stratified space and let $\cE$ be a presentable $\infty$-category.
	Let $\kappa$ be a regular cardinal.
	We denote by $\Cons_{P,\kappa}^{\hyp}(X;\cE)$ the full subcategory of $\Cons_P^{\hyp}(X;\cE)$ spanned by $\cE$-valued hyperconstructible hypersheaves $F$ whose stalks are $\kappa$-compact objects of $\cE$.
	When the stratification $P$ is trivial, we denote this $\infty$-category by $\LChyp_\kappa^{\hyp}(X;\cE)$.
\end{notation}

\begin{warning}
	The $\infty$-category $\Cons_{P,\kappa}^{\hyp}(X;\cE)$ does not coincide neither with $\Cons_P^{\hyp}(X;\cE^\kappa)$ nor with $\Cons_P^{\hyp}(X;\cE)^\kappa$.
	We offer two counterexamples:
	\begin{enumerate}\itemsep=0.2cm
		\item Take $X \coloneqq \coprod_{\mathbb N} *$ to be an infinite disjoint union of points equipped with the trivial stratification.
		Take also $\cE = \cS$ and $\kappa = \omega$.
		Fix $K \in \cS^\kappa$ and let $F \coloneqq \Gamma_X^\ast(K)$ be the hyperconstant hypersheaf associated to $K$.
		Then $F \in \Cons_{P,\kappa}^{\hyp}(X;\cE)$, but $F(X) \simeq \prod_{\mathbb N} K$, which does not belong to $\cS^\omega$ unless $K$ is contractible.
		In particular, $F$ does not belong to $\Cons_P^{\hyp}(X;\cS^\omega)$.
		
		\item Take $X = S^1$, $\cE = \mathrm{Mod}_\C$ and $\kappa = \omega$.
		Once again, equip $X$ with the trivial stratification.
		Then, we have canonical identifications
		\[ \LChyp(S^1;\mathrm{Mod}_\C) \simeq \Fun(\Pi_{\infty}(S^1),\mathrm{Mod}_\C)\simeq \mathrm{QCoh}(\mathbb G_{m,\C}) \ . \]
		Therefore,
		\[ \LChyp(S^1;\mathrm{Mod}_\C)^\omega \simeq \mathrm{Perf}(\mathbb G_{m,\C}) \ . \]
		In particular, the structure sheaf of $\mathbb G_{m,\C}$ corresponds to the functor $F \colon S^1 \to \mathrm{Mod}_\C$ selecting $\C[T,T^{-1}]$ with endomorphism given by multiplication by $T$.
		Then \cref{cor:Phi_formula} implies that the stalk of $\Psi_{S^1}^{\hyp}(F)$ at any point of $S^1$ coincides with $\C[T,T^{-1}]$, which is not compact as an object in $\mathrm{Mod}_\C$.
	\end{enumerate}
\end{warning}

\begin{rem}
	In the second example above, it follows from \cite{Preygel_Integral_transform} that $\LChyp_\kappa(S^1;\mathrm{Mod}_\C)$ corresponds to the full subcategory of $\mathrm{QCoh}(\mathbb G_{m,\C})$ spanned by perfect complexes \emph{with proper support}.
\end{rem}

\begin{prop} \label{prop:exodromy_compact_objects}
	Let $(X,P)$ be a conically stratified space and let $\cE$ be a presentable $\infty$-category.
	Let $\kappa$ be a regular cardinal.
	If the assumptions of \cref{thm:exodromy} are satisfied, then the exodromy equivalence $\Phi_{X,P}^{\hyp} \dashv \Psi_{X,P}^{\hyp}$ restricts to an equivalence
	\[ \Phi_{X,P}^{\hyp} \colon \Fun(\Pi_\infty(X,P), \cE^\kappa) \leftrightarrows \Cons_{P,\kappa}^{\hyp}(X;\cE) \colon \Psi_{X,P}^{\hyp} \ . \]
\end{prop}

\begin{proof}
	It is enough to prove that both $\Phi_{X,P}^{\hyp}$ and $\Psi_{X,P}^{\hyp}$ respect these two full subcategories.
	Let first $F \in \Cons_{P,\kappa}^{\hyp}(X;\cE)$.
	Then we have to check that the functor $\Phi_{X,P}^{\hyp}(F)$ takes values in $\cE^\kappa$.
	To see this, let $x \in \Pi_\infty(X,P)$.
	Then \cref{cor:Phi_formula} provides a canonical identification
	\[ \Phi_{X,P}^{\hyp}(F)(x) \simeq F_x  \]
	so that $\Phi_{X,P}^{\hyp}(F)$ belongs to $\Fun(\Pi_\infty(X,P), \cE^\kappa)$.
	
	\medskip
	
	Let now $G \colon \Pi_\infty(X,P) \to \cE^\kappa$ be a functor.
	Let $x \in X$.
	Then \cref{cor:Phi_formula} and \cref{thm:exodromy} provide the following canonical identifications:
	\[ \Psi_{X,P}^{\hyp}(G)_x \simeq \big( \Phi_{X,P}^{\hyp} \Psi_{X,P}^{\hyp}(G) \big)(x) \simeq G(x) \ . \]
	Thus, the stalks of $\Psi_{X,P}^{\hyp}(G)$ belong to $\cE^\kappa$.
	In other words, $\Psi_{X,P}^{\hyp}(G)$ belongs to $\Cons_{P,\kappa}^{\hyp}(X;\cE)$.
\end{proof}

\subsection{Constructibility and pushforward. A general criterion}

Let $f \colon (Y,Q) \to (X,P)$ be an exodromic morphism of conically stratified spaces.
Then the natural transformation $\psi_f^{\hyp}$ makes the square
\[ \begin{tikzcd}
	\Fun(\Pi_\infty(X,P), \cE) \arrow{r}{\Pi_\infty(f)^\ast} \arrow{d}{\Psi_{X,P}^{\hyp}} & \Fun(\Pi_\infty(Y,Q),\cE) \arrow{d}{\Psi_{Y,Q}^{\hyp}} \\
	\HSh(X;\cE) \arrow{r}{f^{\ast,\hyp}} & \HSh(Y;\cE)
\end{tikzcd} \]
commutative.
In particular, there is an associated Beck-Chevalley transformation
\[ \chi_f^{\hyp} \colon \Psi_{X,P}^{\hyp} \circ \Pi_\infty(f)_\ast \to f_\ast \circ f^{\ast, \hyp} \circ \Psi_{X,P}^{\hyp} \circ \Pi_\infty(f)_\ast \xrightarrow{\psi_f^{\hyp}} f_\ast \circ \Psi_{Y,Q}^{\hyp} \circ \Pi_\infty(f)^\ast \circ \Pi_\infty(f)_\ast \to f_\ast \circ \Psi_{Y,Q}^{\hyp} \ . \]
The identification \eqref{eq:pushforward} has the following immediate consequence:

\begin{lem}\label{constructibility_push_criterion}
	Let $f \colon (Y,Q) \to (X,P)$ be a  morphism of conically stratified spaces with locally weakly contractible strata.
	Let $\cE$ be a presentable $\infty$-category and assume that the conditions of \cref{thm:exodromy} are satisfied.
	Then for every $F \colon \Pi_\infty(Y,Q) \to \cE$, the following statements are equivalent:
	\begin{enumerate}\itemsep=0.2cm
		\item the hypersheaf $f_\ast\big( \Psi_{Y,Q}^{\hyp}(F) \big) \in \HSh(X;\cE)$ is hyperconstructible on $(X,P)$;
		
		\item the  transformation $\chi_f^{\hyp}(F) \colon \Psi_{X,P}^{\hyp}\big( \Pi_\infty(f)_\ast(F) \big) \to f_\ast\big( \Psi_{Y,Q}^{\hyp}(F) \big)$ is an equivalence.
	\end{enumerate} 
\end{lem}

The following observation gives a convenient sufficient condition to check whether pushforward along a morphism $f \colon (Y,Q) \to (X,P)$ preserves hyperconstructible hypersheaves:

\begin{lem}\label{lem:strongly_exodromic_preserve_constructible}
	Let $f \colon (Y,Q) \to (X,P)$ be a strongly exodromic morphism of conically stratified spaces with locally weakly contractible strata.
	Let $\cE$ be a presentable $\infty$-category and assume that the conditions of \cref{thm:exodromy} are satisfied.
	Then $f_\ast$ preserves hyperconstructible hypersheaves.
\end{lem}
\begin{proof}
By assumption, the natural transformation $\phi_f^{\hyp}$ makes the diagram
	\[ \begin{tikzcd}[column sep = large]
		\HSh(X;\cE) \arrow{r}{f^{\ast,\hyp}} \arrow{d}{\Phi_{X,P}^{\hyp}} & \HSh(Y;\cE) \arrow{d}{\Phi_{Y,Q}^{\hyp}} \\
		\Fun(\Pi_\infty(X,P),\cE) \arrow{r}{\Pi_\infty(f)^\ast} & \Fun(\Pi_\infty(Y,Q),\cE)
	\end{tikzcd} \]
	commutative.
	Passing to right adjoints, we deduce that the natural transformation $\chi_f^{\hyp}$ is an equivalence.
	Then \cref{lem:strongly_exodromic_preserve_constructible} follows from \cref{cor:Psi_produces_constructible}.
\end{proof}

\subsection{Constructibility and pushforward along immersions}

\begin{prop} \label{cor:pushforward_from_strata}
	Let $(X,P)$ be a conically stratified space with locally weakly contractible strata. 
	Let $\cE$ be a presentable $\infty$-category.
	Let $S \subseteq P$ be a locally closed subset.
	If the assumption of \cref{thm:exodromy} are satisfied, then
	\[ i_{S,\ast} \colon \HSh(X_S;\cE) \to \HSh(X;\cE) \]
	takes hyperconstructible hypersheaves on $(X_S,S)$ to hyperconstructible hypersheaves on $(X,P)$.
\end{prop}

\begin{proof}
	Combine Corollaries~\ref{cor:Phi_restriction_closed_strata} and \ref{cor:Phi_restriction_open_strata} with \cref{lem:strongly_exodromic_preserve_constructible}.
\end{proof}

For the notion of recollement of $\infty$-categories, let us refer to \cite[A.8.1]{Lurie_Higher_algebra}.
\begin{cor}\label{recollement}
Let $(X,P)$ be a conically stratified space with locally weakly contractible strata.  
Let $\cE$ be a presentable $\infty$-category.
	Let $S \subseteq P$ be a closed subset  and put $U\coloneqq P\smallsetminus S$.
	If the assumptions of \cref{thm:exodromy} are satisfied, then the fully-faithful functors
\[ 
i_{S,\ast} \colon \Cons_S^{\hyp}(X_S;\cE)\to \Cons_P^{\hyp}(X;\cE) \leftarrow  \Cons_U^{\hyp}(X_U;\cE)  \colon  i_{U,\ast} 
\] 
exhibit $\Cons_P^{\hyp}(X;\cE) $ as a recollement of  $\Cons_S^{\hyp}(X_S;\cE)$ and  $\Cons_U^{\hyp}(X_U;\cE)$.
\end{cor}
\begin{proof}
\cref{recollement} follows immediately from the fact that the fully-faithful functors 
\[ 
i_{S,\ast} \colon \HSh(X_S;\cE) \to \HSh(X;\cE)   \leftarrow  \HSh(X_U;\cE)\colon   i_{U,\ast}  
\]
exhibit $\HSh(X;\cE)$ as a recollement of $\HSh(X_S;\cE)$  and $\HSh(X_U;\cE)$.
\end{proof}

     As for the hyperrestriction of a hyperconstructible hypersheaf to a stratum computed in \cref{cor:hyperpullback_computation}, one can compute the fibre of $ i_{S,\ast}$ at a point $x$ explicitly in terms of a conical chart containing $x$.
      This is the following
\begin{lem}\label{fiber_push_from_open}
	Let $Z$ be a weakly contractible locally weakly contractible topological space and let $(Y,Q)$ be a stratified space such that $\Exit(Y,Q)$ is an $\infty$-category and $(X,P) \coloneqq (Z \times C(Y), Q^{\vartriangleleft})$ is conically stratified.
	Let $\cE$ be a presentable $\infty$-category.
	Let $F \in \ConsPhyp(X;\cE)$ and let $x\in Z$.
	If the assumptions of \cref{thm:exodromy} are satisfied, there is a canonical identification 
	\[ (i_{Q,\ast} F)_x \simeq \lim_{\Exit(Y,Q)} \Phi_{X,P}^{\hyp}(F)|_{\Exit(Y,Q)}\]
	where in the right-hand side, the restriction is performed along the equivalence $\Exit(Z\times \mathbb{R}_{>0}\times Y,Q) \to \Exit(Y,Q)$ induced by the canonical projection.
\end{lem}
\begin{proof}
Write $\cB_x$ for the collection of weakly contractible open neighborhoods of $x$ inside $Z$.
Then
	\[ i_{Q,\ast}(F)_x \simeq \colim_{U \in \cB_x} \colim_{\varepsilon \in (0,1)} F( U \times (0,\varepsilon) \times Y ) \ . \]
	From \cref{thm:exodromy} and \cref{formula_for_psi}, we deduce
		\[ i_{Q,\ast}(F)_x \simeq \colim_{U \in \cB_x} \colim_{\varepsilon \in (0,1)} \lim_{\Exit(U\times (0,\varepsilon)\times Y,Q)}\Phi_{X,P}^{\hyp}(F) \ . \]
		Since the $U\in \cB_x$ are weakly contractible, we have further
			\[  \Exit(U\times (0,\varepsilon)\times Y,Q) \simeq \Exit(U,\ast) \times\Exit((0,\varepsilon),\ast) \times \Exit(Y,Q) \simeq \Exit(Y,Q)  \ . \]
		Hence, the above colimits are constant and \cref{fiber_push_from_open} thus follows.
\end{proof}


\begin{lem}\label{lem:compact_homotopy_type}
	Let $\cC \in \Cat_\infty^\omega$ be a compact object in $\Cat_\infty$ and let $\cE$ be a presentable $\infty$-category.
	Assume that filtered colimits are left exact in $\cE$.
	Then the functor
	\[ \Gamma_{\cC,\ast} \colon \Fun(\cC, \cE) \to \cE \]
	commutes with filtered colimits.
	When $\cE$ is stable, both the left adjoint $\Gamma_{\cC,!}$ and the right adjoint $\Gamma_{\cC,\ast}$ to $\Gamma_\cC^\ast$ restrict to functors
	\[ \Gamma_{\cC,!} , \ \Gamma_{\cC,\ast} \colon \Fun(\cC,\cE^\omega) \to \cE^\omega \ . \]
%
%
%
\end{lem}

\begin{proof}
	The functor $\Gamma_\cC^\ast$ takes $\cE^\omega$ in $\Fun(\cC,\cE^\omega)$ by definition.
	It is therefore enough to prove that $\Gamma_{\cC,\ast}$ commutes with filtered colimits and that in the stable case it restricts it takes $\Fun(\cC,\cE^\omega)$ to $\cE^\omega$.	
	Let $\cT_0(\cE)$ be the full subcategory of $\Cat_\infty$ spanned by those $\infty$-categories $\cC$ for which the functor $\Gamma_{\cC,\ast}$ commutes with filtered colimits.
	Let $\cT(\cE)$ be the full subcategory of $\Cat_\infty$  spanned by those $\infty$-categories $\cC$ for which the functors $\Gamma_{\cC,!}$ and $\Gamma_{\cC,\ast}$ take $\Fun(\cC,\cE^\omega)$ to $\cE^\omega$.
	Inspection reveals that $\cT(\cE)$  is closed under finite colimits, retractions and contain all finite $1$-categories (that is, $1$-categories having a finite number of objects and of morphisms).
	Since $\Cat_\infty$ is compactly generated by the $1$-categories $\Delta^n$, the conclusion follows.
	\personal{
		\begin{enumerate}
			\item \emph{Closure under finite colimits}. Indeed, since both the compactly generated and the stable presentable case filtered colimits are left exact, the results of \cite[\S 8.2]{Porta_Yu_Higher_analytic_stacks_2014} imply that $\cT_0(\cE)$ is closed under finite colimits.
			Moreover, when $\cE$ is stable we also know that $\cE^\omega$ is a stable subcategory of $\cE$ and it is therefore closed under finite limits.
			This implies the desired statement for $\Gamma_{\cC,\ast}$.
			The closure of $\cE^\omega$ under finite colimits implies the desired statement for $\Gamma_{\cC,!}$.
			Thus, it follows in the same way that $\cT_1(\cE)$ is closed under finite colimits as well.
			\item \emph{Closure under retractions}. Consider a retraction diagram:
			\[ \begin{tikzcd}[column sep = small,ampersand replacement = \&]
				\cC \arrow{r}{f} \& \cC' \arrow{r}{g} \& \cC \ .
			\end{tikzcd} \]
			The transitivity of Beck-Chevalley transformations yields a canonical identification of the composite
			\[ \Gamma_{\cC,\ast} \xrightarrow{\mathrm{BC}_g} \Gamma_{\cC,\ast} \circ g_\ast \circ g^\ast \simeq \Gamma_{\cC',\ast} \circ g^\ast \xrightarrow{\mathrm{BC}_f} \big( \Gamma_{\cC',\ast} \circ f_\ast \big) \circ \big( f^\ast \circ g^\ast \big) \simeq \Gamma_{\cC,\ast}  \]
			with the Beck-Chevalley transformation associated to the composite $\mathrm{id}_\cC \simeq g \circ f$.
			Therefore, the functor $\Gamma_{\cC,\ast}$ is a retract of $\Gamma_{\cC'} \circ g^\ast$.
			Since $\cE^\omega$ is closed under retracts, it follows that both $\cT_0(\cE)$ and $\cT_1(\cE)$ are closed under retracts.
			A similar argument applies to $\Gamma_{\cC,!}$.
			\item \emph{Finite $1$-categories}. Both $\cT_0(\cE)$ and $\cT_1(\cE)$ contain all finite $1$-categories.
			For $\cT_0(\cE)$, this follows from the fact that filtered colimits are left exact in $\cE$, and for $\cT_1(\cE)$ it follows from the fact that in the stable case $\cE^\omega$ is closed under finite limits and finite colimits.
	\end{enumerate}}
\end{proof}

\begin{prop} \label{prop:open_pushforward}
	Let $(X,P)$ be a conically stratified space with locally weakly contractible strata and let $S \subseteq P$ be a locally closed subset.
	Assume that $(X,P)$ is locally categorically compact and  that $P$ is finite.
	 Then for every stable $\cE\in \PrL$, the functor
	\[ i_{S,\ast} \colon \Cons_S^{\hyp}(X_S;\cE) \to \Cons_P^{\hyp}(X;\cE) \]
	commutes with colimits and restricts to a functor
	\[ i_{S,\ast} \colon \Cons_{S,\omega}^{\hyp}(X_S;\cE) \to \Cons_{P,\omega}^{\hyp}(X;\cE) \ . \]
\end{prop}
\begin{proof}
	When $S$ is closed inside $P$, the functor $i_{S,\ast}$ coincides with the extension by zero and \cref{prop:open_pushforward} is  trivial in that case.
	It is then enough to consider the case where $S$ is open inside $P$.
	We proceed by induction on $\mathrm{depth}(P)$.
	When $\mathrm{depth}(P) = 0$, the stratification is trivial and therefore we have either $S = \emptyset$ or $S = P$.
	In both cases, the statement is obvious.\\ \indent
	Assume now that $\mathrm{depth}(X) > 0$.
	Since $P$ is finite, the subset $M \subseteq P$ of its minimal elements is finite as well.
	Writing $M = \{p_1, \ldots, p_n\}$, we see that $\{ X_{\geqslant p_i} \}_{i = 1, \ldots, n}$ is an open cover of $X$.
	Since both the compacity of stalks and the formation of filtered colimits are local statements on $X$, we can assume that $P$ has a minimum $p$.
	Using the inductive hypothesis, we further reduce to the case where $S = P \smallsetminus M$.\\ \indent	
	Let now $F \in \Cons_{S,\omega}^{\hyp}(X_S;\cE)$.
	We have to prove that the stalks of $i_{S,\ast}(F)$ belong to $\cE^\omega$.
	Since $i_{S,\ast}$ is fully faithful, we only have to prove this statement for the stalks at a point $x\in X_p$.
	From \cref{local_cat_compact_Prop_def}, there exists  a conical chart of the form $Z \times C(Y)$, where $(Y,S)$ is categorically compact and where $Z$ is weakly contractible and locally weakly contractible.
	From \cref{fiber_push_from_open}, we have a canonical identification 
	\[ (i_{S,\ast} F)_x \simeq \lim_{\Exit(Y,S)} \Phi_{X,P}^{\hyp}(F)|_{\Exit(Y,S)}\]
	Since $\Exit(Y,S)$ is a compact object of $\Cat_{\infty}$ and since the functor $\Phi_{X,P}^{\hyp}(F)$ takes values in $\cE^{\omega}$, \cref{prop:open_pushforward} thus follows from \cref{lem:compact_homotopy_type}.\\ \indent
	The same method guarantees that $i_{S,\ast}$ commutes with filtered colimits.
	To deduce that $i_{S,\ast}$ commutes with colimits, it is thus enough to show that $i_{S,\ast}$ commutes with finite colimits.
	 Since $\cE$ is stable, so are $\Cons_S^{\hyp}(X_S;\cE)$  and $\Cons_P^{\hyp}(X;\cE)$ by \cref{cor:tensor_decomposition}.
	Hence, it is enough to show that $i_{S,\ast}$  commutes with limits.
	This follows immediately from \cref{cor:constructible_limits_colimits}.
	
\end{proof}

\begin{cor}\label{recollement_compact}
    Let $(X,P)$ be a conically stratified space with locally weakly contractible strata.  
	Assume that $(X,P)$ is locally categorically compact and  that $P$ is finite.
	Let $\cE$ be a stable presentable $\infty$-category.
	Let $S \subseteq P$ be a closed subset  and put $U\coloneqq P\smallsetminus S$.
	Then the fully-faithful functors
\[ i_{S,\ast} \colon \Cons_{S,\omega}^{\hyp}(X_S;\cE)\to \Cons_P^{\hyp}(X;\cE) \leftarrow  \Cons_{U,\omega}^{\hyp}(X_U;\cE)  \colon  i_{U,\ast} \] 
exhibits $\Cons_{P,\omega}^{\hyp}(X;\cE) $ as a recollement of  $\Cons_{S,\omega}^{\hyp}(X_S;\cE)$ and  $\Cons_{U,\omega}^{\hyp}(X_U;\cE)$.
\end{cor}
\begin{proof}
Follows immediately from \cref{recollement} and  \cref{prop:open_pushforward}.
\end{proof}

\subsection{Constructibility and exceptional inverse image}

\begin{notation}\label{i_S_shrieck}
	Let $(X,P)$ be a stratified space. 
	Let $S \subset P$ be a locally closed subset.
	Put
	\[ \geq S \coloneqq\{p\in P | \exists s\in S \text{ with } p\geq s\} \]
	The set $\geq S$ is open in $P$ and $S$ is closed in $\geq S$.
	Thus, the inclusion $i_S \colon X_S \to X$ factors as
	\[ \begin{tikzcd}[column sep = small]
		X_S \arrow{r}{\iota_S} & X_{\geqslant S} \arrow{r}{i_{\geqslant S}} & X \ ,
	\end{tikzcd} \]
	where $\iota_S$ is a closed immersion and $i_{\geqslant S}$ is an open immersion.
	For a presentable  $\infty$-category $\cE$,  the functor
	\[ 
	i_S^{!, \hyp} \coloneqq \iota_S^{!,\hyp} \circ i_{\geqslant S}^{\ast,\hyp} \ \colon \HSh(X;\cE) \to \HSh(X_S;\cE)  
	\]
	is right adjoint to  $i_{\geq S,!} \circ \iota_{S,\ast} : \HSh(X_S ;\cE) \to \HSh(X;\cE)$.
\end{notation}

\begin{warning}
The functor $i_S^{!, \hyp}$ may not preserve $P$-hyperconstructibility without any further assumption on $(X,P)$.
\end{warning}

\begin{prop}\label{lem:shriek_constructible}
	Let $(X,P)$ be  conically stratified space with locally weakly contractible strata.
	Let $S\subset P$ be a locally closed subset.
	Let $\cE$ be a presentable stable $\infty$-category.
	Assume that the conditions of \cref{thm:exodromy} are satisfied.
	Then, the functor $i_S^{!,\hyp} \colon \HSh(X;\cE) \to \HSh(X_S;\cE)  $ restricts to a functor
	\[ i_S^{!,\hyp} \colon \ConsPhyp(X;\cE) \to \ConsShyp(X_S;\cE)  \ . \]
\end{prop}

\begin{proof}
	We can assume without loss of generality that $X = X_{\geqslant S}$.
	In this case, we simply write $i \colon X_S \hookrightarrow X$ for the natural inclusion, and $j \colon X_{P\smallsetminus S} \hookrightarrow X$ for the inclusion of the open complementary.
	For every $F \in \HSh(X;\cE)$, we have the following fiber sequence
\[
		i_\ast i^{!,\hyp}(F) \to F \to j_\ast j^{\ast,\hyp}(F) 
\]
	computed in $\HSh(X;\cE)$.
	Assume now that $F$ is $P$-hyperconstructible.
	Then the same goes for $j^{\ast,\hyp}(F)$ and \cref{cor:pushforward_from_strata} implies that $j_\ast j^{\ast,\hyp}(F)$ is $P$-hyperconstructible as well.
	Thus, \cref{cor:constructible_limits_colimits} implies that $i_\ast i^{!,\hyp}(F)$ belongs to $\ConsPhyp(X;\cE)$.
	In particular, $i^{!,\hyp}(F) \simeq i^{\ast,\hyp} i_\ast i^{!,\hyp}(F)$ belongs to $\ConsShyp(X_S;\cE)$.
\end{proof}

\begin{prop}
   Let $(X,P)$ be  conically stratified space with locally weakly contractible strata.
    Assume that $(X,P)$ is locally categorically compact and  that $P$ is finite.
	Let $S\subset P$ be a locally closed subset and let $\cE$ be a presentable stable $\infty$-category.
	Then the functor 
	\[ i_S^{!,\hyp} \colon \ConsPhyp(X;\cE) \to \ConsShyp(X_S;\cE)  \]
	 commutes with colimits and restricts to a functor
    \[ i_S^{!,\hyp} \colon \Cons_{P,\omega}^{\hyp}(X;\cE) \to \Cons_{S,\omega}^{\hyp}(X_S;\cE)\ .  \]
\end{prop}
\begin{proof}
We can assume without loss of generality that $X = X_{\geqslant S}$.
	In this case, we simply write $i \colon X_S \hookrightarrow X$ for the natural inclusion, and $j \colon X_{P\smallsetminus S} \hookrightarrow X$ for the inclusion of the open complementary.
	By  \cref{cor:constructible_limits_colimits} and the fact that  $i^{!,\hyp} \colon \HSh(X;\cE) \to \HSh(X_S;\cE)  $ commutes with limits,  the functor $ i^{!,\hyp} \colon \ConsPhyp(X;\cE) \to \ConsShyp(X_S;\cE) $ commutes with limits as well.
	Since $\ConsPhyp(X;\cE)$  and $\ConsShyp(X_S;\cE)$ are stable in virtue of \cref{cor:constructible_presentable}, we deduce that $i^{!,\hyp} \colon \ConsPhyp(X;\cE) \to \ConsShyp(X_S;\cE) $ commutes with finite colimits.
	Hence, we are left to show the commutation with filtered colimits.
This follows by considering the fibre sequence
\[
		i_{\ast} i^{!,\hyp}(F) \to F \to j_{\ast} j^{\ast,\hyp}(F) 
\]	
and using that $j_{\ast}  \colon \Conshyp_{P\smallsetminus S}(X_{P\smallsetminus S};\cE) \to \ConsPhyp(X;\cE)$ commutes with filtered colimits in virtue of \cref{prop:open_pushforward}.
The last claim follows immediately from \cref{prop:open_pushforward} combined with the above fibre sequence.
\end{proof}

\subsection{Change of coefficients revisited}

\begin{lem}\label{change_coef_preserves_cons}
	Let $(X,P)$ be a conically stratified space with locally weakly contractible strata.
	Let $L : \cE \to \cD$ be a morphism in $\PrL$ with right adjoint $R : \cD \to \cE$.
	Assume that the conditions of \cref{thm:exodromy}  are satisfied.
	For every $p\in P$,  the following commutative diagram
	$$
	\begin{tikzcd}
		\HSh(X_p;\cD)  \arrow{r}{R^{\hyp} } \arrow{d}{i_{p,\ast}} & \HSh(X_p;\cE)\arrow{d}{i_{p,\ast}} \\
		\HSh(X;\cD)\arrow{r}{R^{\hyp} } & \HSh(X;\cE) \ .
	\end{tikzcd}
	$$
	is verticaly left adjointable on $\ConsPhyp(X;\cD)$.
	That is, for every $F\in \ConsPhyp(X;\cD)$, the Beck-Chevalley transformation
	\[ i_p^{\ast  ,\hyp}\circ R^{\hyp} (F)\to R^{\hyp} \circ  i_p^{\ast  ,\hyp}(F) \ \]
	is an equivalence.
\end{lem}
\begin{proof}
	The question is local on $X$.
	Hence, we can suppose that $(X,P)$ is of the form $Z\times C(Y)$ where $Z$ is a locally weakly contractible topological open subset of $X_p$ and where $(Y,P_{>p})$ is a stratified space.
	Let $U$ be an open subset of $Z$. 
	Then, \cref{cor:hyperpullback_computation} gives 
	\[ (R^{\hyp} \circ  i_p^{\ast  ,\hyp}(F))(U)  = R^{\hyp} (i_p^{\ast  ,\hyp}(F)(U))\simeq R^{\hyp} (F(U\times C(Y))) \]
	On the other hand, we have 
	\[   (i_p^{\ast  ,\hyp}\circ R^{\hyp}(F))(U) \simeq \colim_{\varepsilon \in (0,1)} R^{\hyp}(F( U \times C_{\varepsilon}(Y) ) )  \]
	We know from \cref{cor:criterion_constructibility} that for every $\varepsilon \in (0,1)$, the restriction morphism
	\[ F( U \times C(Y))  \to F( U \times C_{\varepsilon}(Y) ) \]
	is an equivalence. 
	Hence, the above colimit is constant and we get
	\[ (i_p^{\ast  ,\hyp}\circ R^{\hyp}(F))(U)  \simeq R^{\hyp}(F(U\times C(Y))) \]
	\cref{change_coef_preserves_cons} is thus proved.
\end{proof}

The following \cref{change_coef_Exodromy} contrasts with \cref{no_right_adjoint}.

\begin{prop}\label{change_coef_Exodromy}
	Let $(X,P)$ be a conically stratified space with locally weakly contractible strata.
	Let $L : \cE \to \cD$ be a morphism in $\PrL$ with right adjoint $R : \cD \to \cE$.
	Assume that the conditions of \cref{thm:exodromy} are satisfied.
	Then, the following statements hold :
	\begin{enumerate}\itemsep=0.2cm
    	\item  For every $F\in  \ConsPhyp(X;\cD)$, the functor $R\circ F : \mathrm{Open}(X)\op \to \cE$ lies in $\ConsPhyp(X;\cE)$.
	    In particular, the adjunction 
\[ 
L^{\hyp} \colon \HSh(X;\cE) \leftrightarrows \HSh(X;\cD)  \colon R^{\hyp} 
\] 	   
restricts to  an adjunction
		\[L^{\hyp}  \colon \ConsPhyp(X;\cE)\leftrightarrows \ConsPhyp(X;\cD)\colon R^{\hyp}     \]
    
	    \item The Exodromy equivalence induces an equivalence of adjunctions
	    \[ \begin{tikzcd}
			L^{\hyp} \colon \ConsPhyp(X;\cE) \arrow[r, shift left=2pt]  \arrow{d}{\Phi_{X,P}^{\hyp,\cE}} & \ConsPhyp(X;\cD)\arrow{d}{\Phi_{X,P}^{\hyp,\cD}} \arrow[l, shift 	left=2pt] \colon R^{\hyp}    \\
			L \circ - \colon \Fun(\Pi_\infty(X,P), \cE) \arrow[r, shift left=2pt]& 	\Fun(\Pi_\infty(X,P), \cD) \arrow[l, shift left=2pt] \colon R \circ - \ .
		\end{tikzcd} \]
	\end{enumerate}
\end{prop}	
\begin{proof}
	
	Item $(1)$ follows from \cref{change_coef_preserves_cons} and \cref{g_preserves_LC}.
	Item $(2)$ follows from \cref{cor:exodromy_change_of_coefficients}.
\end{proof}

\begin{lem}\label{change_coef_Cons}
	Let $f : (X,P)\to (Y,Q)$ be a morphism of conically stratified spaces with locally weakly contractible strata.
	Let $L : \cE \to \cD$ be a morphism in $\PrL$ with right adjoint $R : \cD \to \cE$.
	Assume that the conditions of \cref{thm:exodromy}  are satisfied.
	Then, the commutative square
	$$
	\begin{tikzcd}
		\HSh(X;\cD)     \arrow{r}{R^{\hyp} } \arrow{d}{f_*} & \HSh(X;\cE)  \arrow{d}{f_*} \\
		\HSh(Y;\cD)  \arrow{r}{R^{\hyp} } &  \HSh(Y;\cE)  
	\end{tikzcd}
	$$
	is vertically left adjointable on $Q$-hyperconstructible hypersheaves.
	That is, for every $F\in \ConsQhyp(Y;\cD)$, the Beck-Chevalley transformation
	\[ f^{\ast  ,\hyp}\circ R^{\hyp} (F)\to R^{\hyp} \circ  f^{\ast  ,\hyp}(F) \ \]
	is an equivalence.	
\end{lem}
\begin{proof}
	Apply  \cref{cor:induced_stratification} and \cref{change_coef_Exodromy}-(2).
\end{proof}

\begin{rem}
	Put in a loose way, in the setting of \cref{thm:exodromy}, the functor $R^{\hyp} $ acquires on hyperconstructible hypersheaves the symmetry that holds easily for $L^{\hyp}$ on hypersheaves (\cref{hyp_coeff_change_inverse_image}).
	In \cref{change_coefficient_push} and \cref{lem:commuting_shriek_change_of_coefficients} below, we show that in the setting of \cref{thm:exodromy} and for the specific case of the locally closed immersion $i_S : X_S \to X$ where $S\subset P$ is locally closed, $L^{\hyp}$ acquires on hyperconstructible hypersheaves the symmetries that hold easily for $R^{\hyp} $ on hypersheaves (\cref{right_adjoint_and_!_closed_immersion} and commutation with push-forward).
\end{rem}

\begin{lem}\label{change_coefficient_push}
	Let $(X,P)$ be a conically stratified space with locally weakly contractible strata and let $S \subseteq P$ be a locally closed subset.
	Assume that $(X,P)$ is locally categorically compact and  that $P$ is finite.
	 Then for every morphism $L \colon \cE \to \cD$  in $\PrL$ where $\cE,\cD$ are stable, the commutative square
	\[ \begin{tikzcd}
		\Cons_P^{\hyp}(X;\cE) \arrow{r}{i_S^{\ast,\hyp}} \arrow{d}{L^{\hyp}} & \Cons_S^{\hyp}(X_S;\cE) \arrow{d}{L^{\hyp}} \\
		\Cons_P^{\hyp}(X;\cD) \arrow{r}{i_S^{\ast,\hyp}} & \Cons_S^{\hyp}(X_S;\cD)
	\end{tikzcd} \]
	is horizontally right adjointable.
	That is, the Beck-Chevalley transformation
	\[ L^{\hyp} \circ i_{S,\ast} \to i_{S,\ast}  \circ L^{\hyp} \]
	is an equivalence.	
\end{lem}
\begin{proof}
	Note that the above square commutes in virtue of \cref{hyp_coeff_change_inverse_image}.
      	By \cref{prop:open_pushforward}, the functor $ i_{S,\ast} \colon \Cons_S^{\hyp}(X_S;\cE) \to \Cons_P^{\hyp}(X;\cE)$
commutes with colimits and thus decomposes via the equivalences
	\[
\Cons_P^{\hyp}(X;\cE)  \simeq \Cons_P^{\hyp}(X;\Sp)\otimes \cE 	
	\]
and 
\[
 \Cons_S^{\hyp}(X_S;\cE)  \simeq \Cons_S^{\hyp}(X_S;\Sp)\otimes \cE 
\]
supplied by \cref{cor:tensor_decomposition} as the tensor product 
\[
i_{S,\ast}  \otimes \id_{\cE} : \Cons_S^{\hyp}(X_S;\Sp)\otimes \cE 	 \to \Cons_P^{\hyp}(X;\Sp)\otimes \cE \ .
\]
    Then the desired base change is obvious.
\end{proof}

%

In the next lemma, the  \cref{i_S_shrieck}  are in used.

\begin{prop}\label{lem:commuting_shriek_change_of_coefficients}
     Let $(X,P)$ be a conically stratified space with locally weakly contractible strata and let $S \subseteq P$ be a locally closed subset.
	Assume that $(X,P)$ is locally categorically compact and  that $P$ is finite.
	 Then for every morphism $L \colon \cE \to \cD$  in $\PrL$ where $\cE,\cD$ are stable, the commutative square
	\[ \begin{tikzcd}
		\HSh(X_S;\cE)    \arrow{r}{L^{\hyp} } \arrow{d}{i_{\geq S,!} \circ i_{S,\ast} } & \HSh(X_S;\cD)   \arrow{d}{i_{\geq S,!} \circ i_{S,\ast} } \\
		\HSh(X;\cE)  \arrow{r}{L^{\hyp}} & \HSh(X;\cE)  
	\end{tikzcd} \]
	is vertically right adjointable on $P$-hyperconstructible hypersheaves.
	That is,  for every $F\in \ConsPhyp(X;\cE)$,  the Beck-Chevalley transformation
	\[  L^{\hyp} \circ i_S^{!,\hyp}(F) \to i_S^{!,\hyp} \circ L^{\hyp}(F)  \]
	is an equivalence.
\end{prop}

\begin{proof}
	Recall that the above diagram indeed commutes in virtue of \cref{right_adjoint_and_!_closed_immersion}.
	From \cref{hyp_coeff_change_inverse_image},  we can  assume that $S\subset P$ is closed.
	In this case, we simply write $i \colon X_S \hookrightarrow X$ for the canonical closed immersion and $j \colon X_{P\smallsetminus S} \hookrightarrow X$ for the inclusion of the open complementary.
	In view of the morphism of cofiber sequences 
	\[ \begin{tikzcd}
		L^{\hyp} i^{!,\hyp}(F) \arrow{r} \arrow{d}  & L^{\hyp}i^{\ast,\hyp} (F) \arrow{r} \arrow{d} & L^{\hyp} i^{\ast,\hyp}  j_\ast j^{\ast,\hyp}(F) \arrow{d} \\
		i^{!,\hyp} L^{\hyp}(F) \arrow{r} &i^{\ast,\hyp}  L^{\hyp}(F) \arrow{r} &i^{\ast,\hyp}  j_\ast j^{\ast,\hyp} L^{\hyp}(F) 
	\end{tikzcd} \]	
	the conclusion follows from \cref{hyp_coeff_change_inverse_image} and \cref{change_coefficient_push}.
\end{proof}

\subsection{Constructibility, pushforward and weakly stratified bundles}

We now present a second result concerning proper pushforward playing an essential role in \cite{Porta_Teyssier_Stokes}.
Lurie proved proper non abelian base change for sheaves on locally compact Hausdorff spaces.
See  \cite[Corollary 7.3.1.18]{HTT}.
Note however that it is not clear that proper base change holds for hypersheaves.
We are going to see that under some additional assumptions on the stratifications involved, proper base change holds for \emph{hyperconstructible} hypersheaves.

\begin{observation} \label{obs:proper_implies_amenable}
	Let $f \colon Y \to X$ be a proper morphism between topological spaces. 
Following \cite[Definition 7.3.1.14]{HTT}, this means for us that $f$ is universally closed.
	Let $C \subseteq X$ be a locally closed subset of $X$.
	Write $\cB_C$ for the collection of open neighborhoods of $C$ inside $X$.
	Then $\{ f\inv(U) \}_{U \in \cB_C}$ is a fundamental system of open neighborhoods for the inverse image $f\inv(C)$.
	Indeed, since $f$ is universally closed we can localize on $X$ and therefore assume that $C$ is closed.
	Let now $V$ be an open neighborhood of $f\inv(C)$ inside $Y$.
	Since $f$ is closed, $f( Y \smallsetminus V )$ is a closed subset of $X$.
	Furthermore $f( Y \smallsetminus V ) \cap C = \emptyset$, so $U \coloneqq X \smallsetminus f( Y \smallsetminus V )$ is an open neighborhood of $C$ inside $X$.
	We now observe that if $y \in f\inv(U)$, then $f(y) \notin f( Y \smallsetminus V )$, which in turn implies that $y \in V$.
	Therefore, $f\inv(U) \subseteq V$.
\end{observation}

Before proving the sought after proper base change for hyperconstructible hypersheaves, we need a strengthening of the notion of finality introduced in \cref{Excellent_at_S_stratified space}.

\begin{defin}\label{def:hereditary_excellent}
	Let $(X,P)$ be a conically stratified space and let $S \subseteq P$ be a subset.
	We say that $(X,P)$ is \emph{hereditary final at $S$} if for every open subset $U \subseteq X$, the stratified space $(U,P)$ is final at $S$.
\end{defin}

\begin{eg}
	Locally finitely triangulable stratified spaces $(X,P)$ are hereditary final at every locally closed subset of $P$.
	This is proven in \cite[Proposition 2.4.7]{Porta_Teyssier_Stokes}.
\end{eg}

\begin{prop} \label{prop:constructible_base_change}
	Let $f \colon (Y,Q) \to (X,P)$ be a morphism of conically stratified spaces with locally weakly contractible strata and let $S \subseteq P$ be a locally closed subset.
	Let $\varphi \colon Q \to P$ be the underlying morphism of posets and set $R \coloneqq \varphi\inv(S)$.
	Consider the induced commutative square
	\[ \begin{tikzcd}
		(Y_R,R) \arrow{r}{j} \arrow{d}{g} & (Y,Q) \arrow{d}{f} \\
		(X_S,S) \arrow{r}{i} & (X,P) \ .
	\end{tikzcd} \]
	Assume that:
	\begin{enumerate}\itemsep=0.2cm
		\item the underlying morphism $f \colon Y \to X$ is proper;
		\item $(Y,Q)$ is hereditary final at $R$.
	\end{enumerate}
	Then for every presentable $\infty$-category $\cE$ satisfying the assumptions of \cref{thm:exodromy} and every $F \in \Cons_Q^{\hyp}(Y;\cE)$ the canonical map
	\[ i^{\ast,\hyp}( f_\ast(F) ) \to g_\ast( j^{\ast,\hyp}(F) ) \]
	is an equivalence.
\end{prop}

\begin{proof}
	The statement is local on $X$.
	We can therefore suppose that $S$ is a closed downwards subset of $P$.
	Since source and target are hypersheaves, it is enough to prove that for every open subset $U$ of $X_S$ the induced morphism
	\[ i\inv(f_\ast(F))(U) \to g_\ast(j^{\ast, \hyp}(F))(U) \]
	is an equivalence.
	Let $\cB_U$ for the collection of open neighborhoods of $U$ inside $X$.
	Since $f$ is proper and $U$ is locally closed inside $X$, \cref{obs:proper_implies_amenable} shows that $\{f\inv(V)\}_{V \in \cB_U}$ is a fundamental system of open neighborhoods of $g\inv(U) = f\inv(U)$ inside $Y$.
	Since $(Y,Q)$ is hereditary final at $R$, the collection $\cB_{g\inv(U)}^{\mathrm{exc}}$ of final open neighborhoods of $g\inv(U)$ inside $Y$ is a fundamental system of open neighborhoods for $g\inv(U)$.
	Thus, \cref{thm:exodromy} and \cref{lem:RKE} provide the following chain of natural equivalences:
	\[ i\inv(f_\ast(F))(U) \simeq \colim_{V \in \cB_{U}} F(f^{-1}(V))\simeq \colim_{W \in \cB_{g\inv(U)}^{\mathrm{exc}}} F(W) \simeq \colim_{W \in \cB_{g\inv(U)}^{\mathrm{exc}}} \lim_{\Pi_\infty(W,Q)} \Phi_{Y,Q}^{\hyp}(F) \ . \]
	Since each $W \in \cB_{g\inv(U)}^{\mathrm{exc}}$ is final at $R$, the functor $\Pi_\infty(g\inv(U), R) \to \Pi_\infty(W,Q)$ is final.
	Therefore the colimit on the right is constant, and we deduce 
	\[
i\inv(f_\ast(F))(U) \simeq 	\lim_{\Pi_\infty(g\inv(U), R)} \Phi_{Y,Q}^{\hyp}(F)|_{\Pi_\infty(Y_R, R)}
	\]
	Since $j : Y_R \to Y$ is exodromic, we deduce
	\[ i\inv(f_\ast(F))(U) \simeq \lim_{\Pi_\infty(g\inv(U),R)} \Phi_{Y_R,R}^{\hyp}(j^{\ast,\hyp}(F)) \simeq j^{\ast,\hyp}(F)(g\inv(U)) \simeq g_\ast (j^{\ast,\hyp}(F))(U) \ . \]
	The conclusion follows.
\end{proof}

\begin{rem}
	In the above proof, properness is only used in the form of \cref{obs:proper_implies_amenable}.
	It would therefore be enough to ask that for every $x \in X$ the collection $\{f\inv(V)\}_{V \in \cB_x}$ form a fundamental system of open neighborhood for the fiber $f\inv(x)$ inside $Y$ (where $\cB_x$ denotes the collection of open neighborhoods of $x$ inside $X$).
\end{rem}

\begin{defin}
	We say that a morphism between conically stratified spaces $f \colon (Y,Q) \to (X,P)$ is a \emph{weak stratified bundle} if for every $p \in P$ every point $x \in X_p$ admits an open neighborhood $U$ in $X_p$ such that there exists a conically stratified space $(W,R)$ and an isomorphism of stratified spaces
	\[ (f\inv(U),Q) \xrightarrow{\sim} (U \times W, R) \]
	over $U$.
\end{defin}

\begin{prop} \label{prop:stratified_bundle}
	Let $f \colon (Y,Q) \to (X,P)$ be a morphism of conically stratified spaces with locally weakly contractible strata. 
	Let $\cE$ be a presentable $\infty$-category.
	Assume that:
	\begin{enumerate}\itemsep=0.2cm
		\item the underlying morphism $f \colon Y \to X$ is proper;
		
		\item $f$ is a weak stratified bundle whose fibres are conically stratified spaces with locally weakly contractible strata;
		
		\item $(Y,Q)$ is hereditary final at every locally closed subset of $Q$;
		
		\item the assumptions of \cref{thm:exodromy} are satisfied.
	\end{enumerate}
	Then 
	\begin{enumerate}\itemsep=0.2cm
	\item[(a)] the functor $f_\ast \colon \HSh(Y;\cE) \to \HSh(X;\cE)$ maps  $\Cons_Q^{\hyp}(Y;\cE)$ in  $\Cons_P^{\hyp}(X;\cE)$;
	
	\item[(b)] for every $F \in \Cons_Q^{\hyp}(Y;\cE)$ and every $x \in X$ there is a canonical equivalence
	\[ f_\ast(F)_x \simeq \Gamma(Y_x, j_x^{\ast, \hyp}(F)) \ , \]
	where $j_x \colon Y_x \hookrightarrow Y$ denotes the inclusion of the fiber.	
	\end{enumerate}
\end{prop}

\begin{proof}
	Using \cref{prop:constructible_base_change}, we immediately reduce to the case where $P$ is trivial.
	Since both statements are local on the target and since $f$ is a weak stratified bundle, we can assume that $(Y,Q) \simeq (X\times W, R)$ for some conically stratified space $(W,R)$ and that $f$ coincides with the canonical projection to $X$.
	Let $F \in \Cons_Q^{\hyp}(Y;\cE)$.
	We first show that $f_\ast(F)$ is locally hyperconstant on $X$.	
	To do this, \cite[Proposition 3.1]{HPT} ensures that it is enough to prove that for every inclusion $U \subseteq V$ of weakly contractible open subsets in $X$,  the restriction map
	\[ f_\ast(F)(V) \to f_\ast(F)(U) \]
	is an equivalence.
	Since $F$ is hyperconstructible, \cref{cor:Psi_open_restriction} and \cref{thm:exodromy} allow to rewrite this map as
	\[ \lim_{\Pi_\infty(V \times W,R)} \Phi_{X \times W, R}^{\hyp}(F) \to \lim_{\Pi_\infty(U \times W, R)} \Phi_{X \times W, R}^{\hyp}(F) \ . \]
	Since $\Pi_\infty$ commutes with finite products, we have
	\[ \Pi_\infty(V \times W, R) \simeq \Pi_\infty(V) \times \Pi_\infty(W,R) , \qquad \Pi_\infty(U \times W, R) \simeq \Pi_\infty(U) \times \Pi_\infty(W,R) \ . \]
	\personal{Mauro : I don't think we can drop the assumption that $(W,R)$ is conically stratified because products in $\Cat_\infty$ also need to be derived. Concretely, if $K$ and $H$ are two simplicial sets, it's not enough for one of the two to be a quasi-category for $K \times H$ to be a product in $\Cat_\infty$.}
	Since the map $\Pi_\infty(U) \to \Pi_\infty(V)$ is an equivalence, the conclusion follows.\\ \indent
	Let now $x\in X$.
	To compute $f_\ast(F)_x$, we can further suppose that $X$ is weakly contractible.
In that case, the local hyperconstancy of $f_\ast(F)$ combined with \cite[Proposition 3.1]{HPT} ensures that the canonical map
\[
f_\ast(F)(X)\to f_\ast(F)_x 
\]
is an equivalence.
On the other hand, \cref{cor:Psi_open_restriction} combined with the fact that the morphism $Y_x=\{x\}\times W \to Y$ is exodromic gives a chain of equivalences
\begin{align*}
f_\ast(F)(X)& \simeq \lim_{\Pi_\infty(X \times W,R)} \Phi_{X \times W, R}^{\hyp}(F)\simeq \lim_{\Pi_\infty(\{x\} \times W,R)} \Phi_{X \times W, R}^{\hyp}(F)  \\
                 & \simeq 
\lim_{\Pi_\infty(Y_x,R)} \Phi_{Y_x, R}^{\hyp}(j_x^{\ast, \hyp} (F))\simeq \Gamma(Y_x, j_x^{\ast, \hyp}(F))
\end{align*}
The proof of \cref{prop:stratified_bundle} is thus complete.
\end{proof}


\begin{thebibliography}{BBDG18}
	
	\bibitem[AF59]{Andreotti_Frankel}
	A.~Andreotti and T.~Frankel, \emph{The {L}efschetz theorem on hyperplane
		sections}, Ann. of Math. (2) \textbf{69} (1959), 713--717. \MR{177422}
	
		\bibitem[AFR19]{Ayala_Francis_Rozenblyum_Stratified}
	D.~Ayala, J.~Francis, and N.~Rozenblyum, \emph{A stratified homotopy hypothesis}, J. Eur. Math. Soc. (JEMS) \textbf{21} (2019), no.~4, 1071--1178.
	\MR{3941460}

	\bibitem[AFT17a]{Ayala_Factorization_homology}
	D.~Ayala, J.~Francis, and H.~L. Tanaka, \emph{Factorization homology of
		stratified spaces}, Selecta Math. (N.S.) \textbf{23} (2017), no.~1, 293--362.
	\MR{3595895}
	
	\bibitem[AFT17b]{Ayala_Francis_Tanaka_Local_structures}
	\bysame, \emph{Local structures on stratified spaces}, Adv. Math. \textbf{307}
	(2017), 903--1028. \MR{3590534}

	\bibitem[APT22]{Pantev_Arinkin_Toen}
	D.~Arinkin, T.~Pantev, and B.~To\"en, \emph{Relative shifted symplectic forms},
	In preparation, 2022.

	\bibitem[BGH18]{Barwick_Exodromy}
	C.~Barwick, S.~Glasman, and P.~J. Haine, \emph{Exodromy}, arXiv preprint
	arXiv:1807.03281 (2018).
	
	\bibitem[Bis61]{Bishop_Stein_embeddings}
	E.~Bishop, \emph{Mappings of partially analytic spaces}, Amer. J. Math.
	\textbf{83} (1961), 209--242. \MR{123732}
	
	\bibitem[BZNP17]{Preygel_Integral_transform}
	D.~Ben-Zvi, D.~Nadler, and A.~Preygel, \emph{Integral transforms for coherent
		sheaves}, J. Eur. Math. Soc. (JEMS) \textbf{19} (2017), no.~12, 3763--3812.
	\MR{3730514}
	
	\bibitem[CL26]{Christ_Lampetti_Lagrangian}
	M.~Christ, E.~Lampetti, \emph{Lagrangian structures on the derived moduli of constructible sheaves},  arXiv preprint arXiv:2603.04983
	(2026).
			
	
	\bibitem[Cis19]{Cisinski_Higher_Category}
	D.-C. Cisinski, \emph{Higher categories and homotopical algebra}, Cambridge
	Studies in Advanced Mathematics, vol. 180, Cambridge University Press, 2019.
	
	\bibitem[C{\O}J22]{Clausen_Jansen}
	D.~Clausen and M.~{\O}rsnes~Jansen, \emph{The reductive {B}orel--{S}erre
		compactification as a model for unstable algebraic {$\mathrm{K}$}-theory},
	\href{https://arxiv.org/abs/2108.01924}{\nolinkurl{arXiv:2108.01924}}, March
	2022.
	
	\bibitem[HL22]{Holstein_Categorical_Koszul}
	J.~V.~S. Holstein and A.~Lazarev, \emph{Categorical {K}oszul duality}, Adv.
	Math. \textbf{409} (2022), Paper No. 108644, 52. \MR{4477015}
	
	\bibitem[Hol15a]{Holstein_Morita_cohomology_I}
	J.~V.~S. Holstein, \emph{Morita cohomology}, Math. Proc. Cambridge Philos. Soc.
	\textbf{158} (2015), no.~1, 1--26. \MR{3300312}
	
	\bibitem[Hol15b]{Holstein_Morita_cohomology_II}
	\bysame, \emph{Morita cohomology and homotopy locally constant sheaves}, Math.
	Proc. Cambridge Philos. Soc. \textbf{158} (2015), no.~1, 27--35. \MR{3300313}
	
	\bibitem[HPT20]{HPT}
	P.~J. Haine, M.~Porta, and J.-B. Teyssier, \emph{The homotopy-invariance of
		constructible sheaves of spaces}, arXiv: Algebraic Topology (2020), Accepted
	for publication in Homotopy, Homology and Applications.

	\bibitem[HPT24]{Beyond_conicality}
	\bysame, \emph{Exodromy beyond conicality}, arXiv preprint arXiv:2401.12825
	(2024).

	\bibitem[HPT26]{Haine_Porta_Teyssier_perverse}
	\bysame, \emph{The derived moduli of perverse sheaves}, 
	(2026).

	
	\bibitem[KS69]{Kirby_Siebenmann_Hauptvermutung}
	R.~C. Kirby and L.~C. Siebenmann, \emph{On the triangulation of manifolds and
		the {H}auptvermutung}, Bull. Amer. Math. Soc. \textbf{75} (1969), 742--749.
	\MR{242166}
	
		\bibitem[L25a]{Lampetti_Good_moduli}
	E.~Lampetti,\emph{Good moduli for moduli of objects}, arXiv preprint arXiv:2510.23835
	(2025).
	
		\bibitem[L25b]{Lampetti_Stokes}
	E.~Lampetti, \emph{Good moduli space for constructible sheaves and Stokes functors}, arXiv preprint arXiv:2510.23850
	(2025).
		
	\bibitem[Lej21]{Lejay_Constructible}
	D.~Lejay, \emph{Constructible hypersheaves via exit paths}, arXiv preprint
	arXiv:2102.12325 (2021).
	
	\bibitem[HTT]{HTT}
	J.~Lurie, \emph{Higher topos theory}, Annals of Mathematics Studies, vol. 170,
	Princeton University Press, Princeton, NJ, 2009. \MR{2522659 (2010j:18001)}
	
	\bibitem[HA]{Lurie_Higher_algebra}
	J.~Lurie, \emph{Higher algebra}, Preprint, August 2012.
	
	\bibitem[SAG]{Lurie_SAG}
	J.~Lurie, \emph{Spectral algebraic geometry}, Preprint, 2018.
	
	\bibitem[Mis22]{Mistry_Character_variety}
	V.~Mistry, \emph{On the cohomological hall algebra of a character variety},
	arXiv preprint arXiv:2209.00680 (2022).
	
	\bibitem[Nar60]{Narasimhan_Stein_embeddings}
	R.~Narasimhan, \emph{Imbedding of holomorphically complete complex spaces},
	Amer. J. Math. \textbf{82} (1960), 917--934. \MR{148942}

	\bibitem[NV23]{Nocera_Volpe_Whitney_stratifications}
	Guglielmo Nocera and Marco Volpe, \emph{Whitney stratifications are conically
		smooth}, Selecta Math. (N.S.) \textbf{29} (2023), no.~5, Paper No. 68, 20.
	\MR{4642716}
	
	\bibitem[{\O}J20]{Jansen_Stratified_Borel_Serre}
	M.~{\O}rsnes~Jansen, \emph{The stratified homotopy type of the reductive
		borel-serre compactification}, arXiv preprint arXiv:2012.10777 (2020).
	
	\bibitem[{\O}J23]{Jansen_Mgn}
	\bysame, \emph{Moduli stack of stable curves from a stratified homotopy
		viewpoint}, arXiv preprint arXiv:2308.09551 (2023).
	
		\bibitem[{\O}J24]{Jansen_toolbox}
		\bysame, \emph{Stratified homotopy theory of topological $\infty$-stacks: A toolbox},
	Journal of Pure and Applied Algebra \textbf{228} (2024).
	
	
	\bibitem[PS23]{Porta_Sala_Categorified}
	Mauro Porta and Francesco Sala, \emph{Two-dimensional categorified {H}all
		algebras}, J. Eur. Math. Soc. (JEMS) \textbf{25} (2023), no.~3, 1113--1205.
	\MR{4577961}
	
	\bibitem[PT24]{Porta_Teyssier_Stokes}
	M.~Porta and J.-B. Teyssier, \emph{The derived moduli of Stokes data}, arXiv preprint arXiv:2401.12335 (2024).
	
	\bibitem[PY16]{Porta_Yu_Higher_analytic_stacks_2014}
	M.~Porta and T.~Y. Yu, \emph{Higher analytic stacks and {GAGA} theorems},
	Advances in Mathematics \textbf{302} (2016), 351--409.
	
	\bibitem[Rem56]{Remmert_Stein_embeddings}
	R.~Remmert, \emph{Sur les espaces analytiques holomorphiquement s\'{e}parables
		et holomorphiquement convexes}, C. R. Acad. Sci. Paris \textbf{243} (1956),
	118--121. \MR{79808}
	
	\bibitem[Tre09]{Treumann_Exit_paths}
	D.~Treumann, \emph{Exit paths and constructible stacks}, Compos. Math.
	\textbf{145} (2009), no.~6, 1504--1532. \MR{2575092}

	\bibitem[SS03]{Schwede_Shipley}
	S.~Schwede and B. Shipley, \emph{Stable model categories are categories of modules},
	Topology \textbf{42} (2003), 103--153.
		
	\bibitem[Vol22]{Volpe_Verdier_duality}
	M.~Volpe, \emph{Verdier duality on conically smooth stratified spaces}, arXiv
	preprint arXiv:2206.02728 (2022), To appear in Algebraic and Geometric
	topology.
	
	\bibitem[Vol24]{Volpe_Finiteness}
	\bysame, \emph{Finiteness and finite domination in stratified homotopy theory},
	In preparation, 2024.
	
\end{thebibliography}

\def\cprime{$'$}
\providecommand{\bysame}{\leavevmode\hbox to3em{\hrulefill}\thinspace}
\providecommand{\MR}{\relax\ifhmode\unskip\space\fi MR }
\providecommand{\MRhref}[2]{%
	\href{http://www.ams.org/mathscinet-getitem?mr=#1}{#2}
}
\providecommand{\href}[2]{#2}

\end{document}